\documentclass[11pt]{article}
\usepackage{tgpagella,eulervm}
\usepackage{setspace} 
\usepackage{framed} 
\usepackage{relsize}

\usepackage{color,soul}
\usepackage[nottoc,numbib]{tocbibind}

\usepackage[framemethod=default]{mdframed} % default setting

\usepackage{amssymb,mathtools,amsthm, amsmath, amsfonts}
\usepackage[refpage]{nomencl}

\usepackage{authblk}
\usepackage{fullpage}
\usepackage{enumitem}
\usepackage{graphicx} % Required for inserting images
\usepackage{tikz-cd}
\usetikzlibrary{fit,arrows,calc,positioning}
\usepackage[hidelinks]{hyperref}
\hypersetup{
	colorlinks   = true,
	citecolor    = red
}
\hypersetup{linkcolor=blue}

\usepackage[capitalize,nameinlink]{cleveref}
\usepackage{array}
\usepackage{mathbbol}
\usepackage{listings}% http://ctan.org/pkg/listings  pseudocode with math mode
\usepackage{algorithm}
\usepackage{algcompatible}
\algnewcommand\algorithmicreturn{\textbf{return}}
\algnewcommand\RETURN{\State \algorithmicreturn}

\allowdisplaybreaks

\crefname{figure}{Figure}{Figures}

\usepackage[normalem]{ulem}
\usepackage{mathrsfs} % required for mathscr font
\newtheorem{corollary}{Corollary}[section]
\newtheorem{theorem}{Theorem}

\newtheorem{remark}[corollary]{Remark}
\newtheorem{lemma}[corollary]{Lemma}
\newtheorem{proposition}[corollary]{Proposition}
\newtheorem{definition}[corollary]{Definition}

\newtheorem{example}[corollary]{Example}

\newtheorem{notation}[corollary]{Notation}

\newcommand{\Ffunc}          	{\mathsf{F}} % filtration
\newcommand{\Gfunc}          	{\mathsf{G}} % filtration
\newcommand{\Hfunc}          	{\mathsf{H}} % filtration
\newcommand{\fnc}          	    {\mathsf{Fnc}} % category of functions from edit distance paper
\newcommand{\Fil}          	    {\mathsf{Fil}} % category of 1-parameter filtrations
\newcommand{\Zfunc}          	{\mathsf{Z}} %cycle space
\newcommand{\Bfunc}          	{\mathsf{B}} %boundary space
\newcommand{\Mfunc}          	{\mathsf{M}} 
\newcommand{\Nfunc}          	{\mathsf{N}}

\newcommand{\Lfunc}          	{\lambda} 
\newcommand{\ZB}          	    {\mathsf{ZB}} % birth-death space
\newcommand{\LK}                {\mathrm{LK}} % Laplacian kernel
\newcommand{\gr}          	    {\mathsf{Gr}} %Grassmanian
\newcommand{\pd}          	    {\mathsf{PD}} %classical pd
\newcommand{\Seg}          	    {\mathsf{Seg}} %set of segments
\newcommand{\diag}          	{\mathsf{diag}} %diagonal
\newcommand{\rhi}          	    {\mathsf{LOI}_\supseteq} % reverse inclusion LOI
\newcommand{\prhi}          	{\mathsf{LOI}_\times} % product LOI
\newcommand{\mobeq}          	{\simeq_{\mathsf{M\ddot{o}b}}} % Mobius equivalent
\newcommand{\parti}          	{\mathsf{Part}} % partitions
\newcommand{\spart}          	{\mathsf{SubPart}} %sub partitions
\newcommand{\conn}          	{\mathsf{Conn}} %connected components
\newcommand{\inn}          	{\cap\text{-}\mathsf{Mon}} % TO BE CHANGED
\newcommand{\inndgm}          	{\mathsf{GrPD}} % TO BE CHANGED
\newcommand{\cost}          	{\mathsf{cost}} % cost
\newcommand{\subcx}          	{\mathsf{SubCx}} %sub partitions
\newcommand{\HB}          	{\mathsf{HB}} %harmonic barcode
\newcommand{\VR}          	{\mathsf{VR}} %Vietoris-Rips
\newcommand{\goi}          	{\mathsf{OI}} % generalized orthogonal inversion
\newcommand{\Ffrak}          	{\mathfrak{F}}

\newcommand{\lp}                     {\mathbb{L}}%Linear poset
\newcommand{\dgr}                     {\rho}%degree
\newcommand\ladj[1]{#1_{\diamond}} %left adjoint
\newcommand\radj[1]{#1^{\diamond}} %right adjoint
\newcommand\mmi[2]{\partial_{#1}^\mathsf{Mon}\left(#2\right)} %left adjoint
\newcommand{\dis}{\mathrm{dis}}  %distortion

\newcommand{\R}{\mathbb{R}}
\newcommand{\N}{\mathbb{N}}

\newcommand{\Z}{\mathbb{Z}}

\newcommand{\incord}{\leq_{\mathsmaller{\supseteq}}}
\newcommand{\prodord}{\leq_{{\times}}}

\DeclareMathOperator{\proj}{proj}
\DeclareMathOperator{\Ima}{im}
\DeclareMathOperator{\spn}{span}
\DeclareMathOperator{\rank}{rank}

\title{Grassmannian Persistence Diagrams:\\ Special Properties in the 1-Parameter Setting}
\author[1]{Aziz Burak G\"ulen\footnote{\href{mailto:aziz.burak.guelen@duke.edu}{aziz.burak.guelen@duke.edu}, 120 Science Dr, Durham, NC, 27708}}
\author[2]{Facundo M\'emoli\footnote{\href{mailto:facundo.memoli@gmail.com}{facundo.memoli@gmail.com}, 110 Frelinghuysen Road, Piscataway, NJ 08854}}
\author[3]{Zhengchao Wan\footnote{\href{mailto:zwan@missouri.edu}{zwan@missouri.edu}, 810 Rollins St, Columbia, MO 65201}}
\affil[1]{Department of Mathematics, Duke University}

\affil[2]{Department of Mathematics, Rutgers University}

\affil[3]{Department of Mathematics, University of Missouri}

\date{}

\makenomenclature

\begin{document}

\maketitle

\begin{abstract}
In this paper, we explore the discriminative power of Grassmannian persistence diagrams of $1$-parameter filtrations, examine their relationships with other related constructions, and study their computational aspects. Grassmannian persistence diagrams are defined through Orthogonal Inversion, a notion analogous to M\"obius inversion. We focus on the behavior of this inversion for the poset of segments of a linear poset. We demonstrate how Grassmannian persistence diagrams of 1-parameter filtrations are connected to persistent Laplacians via a variant of orthogonal inversion tailored for the reverse-inclusion order on the poset of segments. Additionally, we establish an explicit isomorphism between Grassmannian persistence diagrams and Harmonic Barcodes via a projection. Finally, we show that degree-0 Grassmannian persistence diagrams are equivalent to treegrams, a generalization of dendrograms. Consequently, we conclude that finite ultrametric spaces can be recovered from the degree-0 Grassmannian persistence diagram of their Vietoris-Rips filtrations.
\end{abstract}

\newpage
\printnomenclature[2cm]
\addcontentsline{toc}{section}{Nomenclature}
\newpage
\tableofcontents

\section{Introduction}
Persistence diagrams are fundamental invariants in topological data analysis, providing concise summaries of the birth and death of homological features in filtered topological spaces \cite{edelsbrunner2010computational,carlsson2009topology,rabadan2019topological}.
Traditionally, i.e. in the case of filtrations over a linear poset, persistence diagrams can be obtained through different ways:
\begin{enumerate}
    
    \item via M\"obius inversion of the rank function \cite{landi1997new,cohen-steiner2007, Patel2018},
    \item via decomposition theorems of the quiver representations of the linear quiver \cite{zomorodian2005, crawley2015decomposition} (see also \cite{Abeasis1981}).
\end{enumerate}

Notably, persistence diagrams can be computed efficiently using polynomial-time algorithms \cite{Edelsbrunner2002,edelsbrunner2010computational,Bauer2021Ripser}. However, despite their efficiency, they do not always serve as complete invariants of the underlying filtrations. In particular, when these filtrations arise as the Vietoris-Rips filtration of finite metric spaces, persistence diagrams may fail to distinguish non-isometric spaces. For example, it is well-known that there are infinitely many pairs of non-isometric ultrametric spaces which are confounded the persistence diagrams of their Vietoris-Rips filtrations; see \Cref{fig:ums} and \cite[Example 3.18]{zhou2024ephemeral}.

\begin{figure}
    \centering
\includegraphics[width=\linewidth]{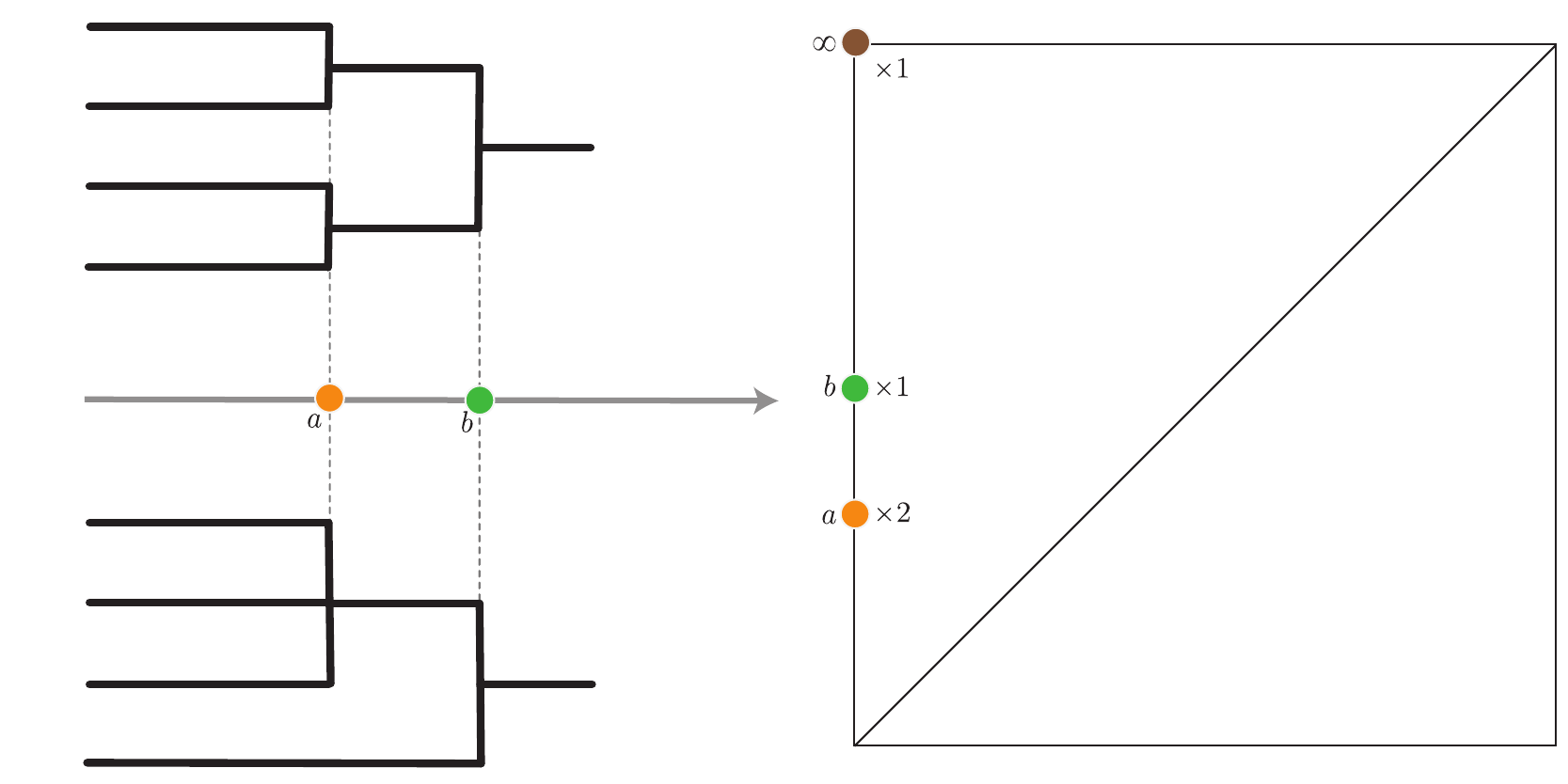}

    \caption{Two families of ultrametric spaces (represented via their corresponding dendrograms and parametrized by $a,b\geq 0$ s.t. $b>a$) having the same Vietoris-Rips persistence diagrams for all degrees. The figure shows the common degree-0 diagram as all other ones are trivial; see \cite[Example 3.18]{zhou2024ephemeral} for details.}
    \label{fig:ums}
\end{figure}

Motivated by this limitation, it is therefore tempting to try to enrich persistence diagrams in order to strengthen their distinguishing power while maintaining computational tractability. Several approaches have been explored in this direction:

\begin{enumerate}
    \item Algebraic decomposition of filtered chain complexes and ephemeral barcodes~\cite{usher, ephemral} and augmented persistence diagrams \cite{fasy2019faithful},
    \item Following~\cite{Patel2018}, M\"obius inverting functions beyond the usual rank function, such as cup-length induced functions~\cite{cup-persistent}, birth-death functions~\cite{edit, saecular, gal-conn}.
\end{enumerate}

In this paper, we study an alternative enrichment of persistence diagrams based on the M\"obius inversion approach. As explained in \cite{edit,gal-conn}, the classical degree-$\dgr$ persistence diagram, $\pd_\dgr^\Ffunc$, of a filtration $\Ffunc$ indexed by the linear poset $\{1<\cdots<n \}$ can be recast as the M\"obius inversion of the birth-death functions as follows:
\footnotesize
\begin{equation}\label{eqn: mob inv of dim zb}
    \pd_\dgr^\Ffunc((i,j)) = \dim \left( \ZB_\dgr^\Ffunc ((i,j)) \right) - \dim \left( \ZB_\dgr^\Ffunc ((i,j-1)) \right) + \dim \left( \ZB_\dgr^\Ffunc ((i-1,j-1)) \right) - \dim \left( \ZB_\dgr^\Ffunc ((i-1,j)) \right),
\end{equation}
\normalsize
where $\ZB_\dgr^\Ffunc$ denotes the degree-$\dgr$ \emph{birth-death space} of the filtration $\Ffunc$; see~\cref{defn: bd space}.\footnote{However, see \cite{landi1997new} for the first manifestation of such a formula in the context of topological persistence (in degree 0).}

Instead of restricting the Möbius formula (\cref{eqn: mob inv of dim zb}) to dimensions of birth-death spaces, we apply it directly at the level of these vector spaces, carefully reinterpreting the sums and differences through the constructions introduced in~\cref{defn: x-harmonic inversion}. This reinterpretation is achieved through orthogonal complementation, assuming an inner product structure at the chain level. As a result, we arrive at the notion of \emph{Grassmannian persistence diagrams}; see~\cref{defn: 1p gpd from zb}.

Grassmannian persistence diagrams provide a strict enrichment of classical persistence diagrams. In classical persistence diagrams, each segment $(i,j)$ of the undelying poset indexing the filtration is assigned a numerical multiplicity, whereas in Grassmannian persistence, each segment is endowed with \emph{a canonically assigned subspace of the chain space}; see~\cref{fig:3d-gpd}. These subspaces correspond to generators of the underlying persistent homology classes, ensuring a direct correspondence between points in the classical persistence diagrams and their associated cycle representatives. Notably, when the multiplicity of a segment $(i,j)$ is $1$, the subspace assigned to this segment by the Grassmannian persistence diagram is $1$-dimensional. Consequently, this subspace has a unique generator up to rescaling, providing a canonical representative for the corresponding point $(i,j)$ in the persistence diagram. We formalize this idea in~\cref{lem: linear varphi isom} by showing that there is a \emph{bijection} between subspaces determined by Grassmannian persistence diagram and the space of homology classes that are born at $i$ and die at $j$.\footnote{Cf. \cite[Theorem 6]{gpd-multi} where, for general posets, we only conclude the existence of a surjection.}

\begin{figure}
    \centering

\includegraphics[width=0.1\linewidth]{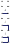}
\includegraphics[width=0.55\linewidth]{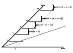}

    \caption{Grassmannian persistence diagram of the 1-parameter filtration depicted on the left. Grassmanian persistence diagrams  retain information about cycle spaces associated to different segmemts. For example, for the segment $(1,2)$ the Grassmanian persistence diagram not only captures the multiplicity of that interval as the dimension of the space $\mathrm{span}\{a-c\}$ but also provides cycles that are precisely born at $1$ and die at $2$.}
    \label{fig:3d-gpd}
\end{figure}

We prove that our notion of Grassmannian persistence diagrams not only remains polynomial-time computable but also exhibits stability in the sense given by a suitably defined non-trivial edit distance. Importantly, Grassmannian persistence diagrams significantly enhance the distinguishing power of classical persistence: we show that they can differentiate any two non-isometric finite ultrametric spaces via their Vietoris-Rips filtrations; see~\cref{cor: finite ums recovered}. This result marks the first known instance, to the best of our knowledge, in which a persistence-based invariant fully reconstructs ultrametric spaces up to isometry.

Additionally, we establish a deep connection between Grassmannian persistence diagrams and the notion of the persistent Laplacian, a concept introduced in~\cite{lieutier-pdf} and developed in~\cite{persSpecGraph,persLap}. It is known that the kernel of the persistent Laplacian is isomorphic to the persistent homology space of a given filtration, implying that the nullity of the persistent Laplacian agrees with the rank function~\cite[Theorem 2.7]{persLap}. By extending our M\"obius inversion approach to the kernels of the persistent Laplacian, we prove that the resulting generalized persistence diagram coincides with our Grassmannian persistence diagram outside of diagonal points.

Finally, we construct an explicit isomorphism linking Grassmannian persistence diagrams with the notion of \emph{Harmonic Barcodes} introduced by Basu and Cox in~\cref{subsec: harmonicph}. This connection provides a deeper understanding of recent efforts to leverage inner product structures on the chain space to enhance classical persistence diagrams.

\subsection{Contributions and Organization of the Paper}
In~\cref{sec: orthogonal inversion}, we introduce the notion of \emph{$\times$-Linear Orthogonal Inversion} (denoted $\prhi$), a notion analogous to classical M\"obius inversion on the poset of segments of a finite linear poset (with the product order). We establish that $\prhi$ serves a functor between two categories; see~\cref{prop: prhi is functor}. And, as a result of this functoriality, we prove its stability; see~\cref{thm: general stability}.

In~\cref{sec: applications}, we build upon the notion of $\times$-Linear Orthogonal Inversion to introduce, in~\cref{defn: 1p gpd from zb}, the notion of degree-$\dgr$ Grassmannian persistence diagram of a $1$-parameter filtration as the $\times$-orthogonal Inverse of the $\dgr$-th birth-death spaces. In~\cref{subsec: orthogonal inversion of bd}, we prove stability of these diagrams in by utilizing the functoriality of $\prhi$; see \cref{thm: stability of prhi of zb}. In~\cref{subsec: inter and canon}, we explore the interpretability and canonicality of Grassmannian persistence diagrams. Specifically, in~\cref{prop: born and dies exactly} we prove  that for a $1$-parameter filtration $\Ffunc : \{ 1< \cdots < n \} \to \subcx(K)$, the subspace determined by Grassmannian persistence diagram of $\Ffunc$ at the segment $(i,j)$ consists of cycles that are born exactly at $i$ and die exactly at $j$. In~\cref{subsec: computation of loi zb}, we present a polynomial-time algorithm for computing the Grassmannian persistence diagram of a $1$-parameter filtration.

In~\cref{sec: relation to other}, we explore the relationship between our construction of Grassmannian of persistence diagrams and the other known constructions in the literature. 
\begin{itemize}
    \item In~\cref{subsec: relation to classical pd}, we show that the classical persistence diagram of a $1$-parameter filtration can be derived from its Grassmannian persistence diagram; see~\cref{prop: dim of loi is classical pd}. Moreover, in~\cref{thm: classical pd lower bound}, we provide a lower bound for the edit distance between two Grassmannian persistence diagrams through the edit distance between their respective classical persistence diagrams. By showing that this lower bound can be $0$ while the edit distance between the Grassmannian persistence diagrams remain positive, we conclude that Grassmannian persistence diagrams are strictly more discriminative than the classical persistence diagrams; see~\cref{example: same pd different 0-prhi}.
    \item In~\cref{subsec: harmonicph}, we establish an isomorphism between the subspaces determined by Grassmannian persistence diagrams and those defined by Harmonic Barcodes, a concept introduced in~\cite{harmonicph}. This isomorphism is realized through a projection, as shown in~\cref{thm: proj isom}.
    \item In~\cref{subsec: orthogonal inversion of Laplacian kernels}, we establish a deep connection between Grassmannian persistence diagrams and persistent Laplacians. Specifically, we introduce a variant of orthogonal inversion in~\cref{defn: sup-harmonic inversion}, called \emph{$\supseteq$-Linear Orthogonal Inversion}, which is designed to be compatible with the reverse-inclusion order on segments. Using this framework, we show that the Grassmannian persistence diagram of a $1$-parameter filtration can be recovered as the $\supseteq$-Linear Orthogonal Inverse of the persistent Laplacian kernels; see~\cref{thm: equality of different HIs}.
    \item In~\cref{subsec: treegrams}, we examine degree-$0$ Grassmannian persistence diagrams and establish their equivalence to the notion of \emph{treegrams}, a generalization of \emph{dendrograms}. This equivalence is formally proven in~\cref{thm: equivalence of treegrams and degree 0 orthogonal inversions}, with an algorithmic and constructive proof provided in~\cref{appendix: construction}. As a direct consequence of this result, we conclude in~\cref{cor: finite ums recovered} that finite ultrametric spaces can be fully reconstructed from the degree-$0$ Grassmannian persistence diagram of their Vietoris-Rips filtrations.
\end{itemize}

In \cref{sec:disc} we collect a few questions that might motivate further research.

\medskip
Finally, we point the reader to the companion paper \cite{gpd-multi} where we  study Grassmannian persistence diagrams in the case of multi-parameter filtrations.

\subsection{Related work}
At a high level, our work can be regarded as providing a bridge between the ideas from \cite{edit} and \cite{harmonicph}, further integrating these concepts with the notion of the persistent Laplacian \cite{lieutier-pdf,persSpecGraph,persLap}. We provide a detailed study of the relationship, in the case of 1-parameter filtrations, between our constructions and the ones in~\cite{edit} and~\cite{harmonicph} in \cref{subsec: relation to classical pd}, \cref{appendix: edit}, and \cref{subsec: harmonicph},
respectively.

\paragraph{Beyond Persistence Modules: Invariants of Filtered Chain Complexes.}
Persistent homology applies the homology functor to simplicial filtrations, implicitly passing through \emph{filtered chain complexes} (FCCs). Several chain-level invariants have been studied, including: (i) algebraic decompositions and verbose barcodes~\cite{usher, ephemral,landi2021invariants,chacholski2023decomposing}; (ii) birth-death functions and spaces~\cite{saecular, edit, morozov-patel}; and (iii) persistent Laplacians~\cite{lieutier-pdf, persSpecGraph, persLap}.
\
In \cite{edit}, persistence diagrams derived via Möbius inversion rely solely on the dimensions of birth-death spaces, leaving unexplored their richer geometric content, such as information about cycle representatives. Persistent Laplacians, defined using inner products on chain groups~\cite{lieutier-pdf, persSpecGraph, persLap}, assign inner product spaces (0-eigenspaces) to intervals. This enriched structure enables canonical representative selection~\cite{harmonicph} and, save for~\cite{harmonicph}, current approaches capture only kernel dimensions (persistent Betti numbers)~\cite[Theorem 2.7]{persLap}, without exploiting the full geometric information available.

 \paragraph{Edit Distance Stability of Persistence Diagrams.}
The notions of edit distance between filtrations and persistence diagrams were introduced by McCleary and Patel in~\cite{edit}, along with an associated stability theorem. Their approach utilizes an \emph{algebraic} Möbius inversion—defined via group completion—which is functorial precisely when the poset of intervals is equipped with the product order. Our stability results (\cref{thm: general stability}) similarly rely on a functorial property, but we replace algebraic Möbius inversion applied to dimension functions with $\times$-Linear Orthogonal Inversion applied directly to interval-associated birth-death vector spaces endowed with inner products. In particular, we prove that our stability result improves upon that of \cite{edit}; see   \cref{thm: classical pd lower bound}, \cref{rmk: gpd more discriminative} and~\cref{example: same pd different 0-prhi}.

\paragraph{Harmonic Persistent Homology.} 
Basu and Cox~\cite{harmonicph} introduced the concept of harmonic persistence diagram of a 1-parameter filtration, which assigns a subspace of the cycle space to every interval in the classical persistence diagram of a filtration. Under this assignment, Basu and Cox prove that the multiplicity of each interval in the persistence diagram and the dimension of the subspace assigned to that point are the same. Furthermore, Basu and Cox also establish the stability of their construction.

Our Grassmannian persistence diagrams (produced via $\times$-Linear Orthogonal Inversion) also assign a subspace to every interval in the persistence diagram whose dimension agrees with the multiplicity of the point. We show that their construction and ours are isomorphic through a certain explicit projection map (see~\cref{thm: proj isom}) that we identify with the caveat that this projection map is however not an isometry; see~\cref{rmk: non-isometry}.

Our construction, differs from theirs in several senses:
\begin{itemize}
\item First, for an interval $(b,d)$ in the standard persistence diagram of a filtration, every nonzero cycle in the subspace assigned to this point through the Grassmannian persistence diagram is guaranteed to be born exactly at time $b$ and to become a boundary exactly at time $d$; see~\cref{prop: born and dies exactly}. Their construction does not satisfy this property, as can be verified in~\cite[Example 1.1]{harmonicph}. 
The reason is that, 
whereas we operate at the level of chain groups where ``death" of a cycle can be directly encoded as the condition that said cycle becomes a boundary, the notions of birth and death considered in \cite{harmonicph}
arise at (harmonic) homology level, where death of a homology class is necessarily encoded as the condition that the class ``merges" with an ``older" one.

\item Consistent with the fact that our construction takes place at chain level, Grassmanian persistence diagrams capture ephemeral features.  Such features are not captured by \cite{harmonicph} since Basu and Cox carry our they constructions at the level of homology.

\item Their stability result requires certain genericity conditions for the persistence diagrams whereas our edit distance stability (\cref{thm: stability of prhi of zb}), assumes no such genericity condition on the persistence diagrams.

\item Finally, even though Basu and Cox  only consider 1-parameter  filtrations, our framework can deal with  the more general setting of filtrations defined on any finite poset. In particular, we generalize their specific and complete description of birth and death of homology classes (see \cite[Definition 3.5]{harmonicph} and \cite[Definition 4.8]{gpd-multi} to the general poset setting and we utilize this new description to further discover exhaustiveness property of our Grasmannian persistence diagrams in the homology level; see \cite[Theorem 6]{gpd-multi} and \cref{lem: linear varphi isom}. 

\end{itemize}

Inspired by \cite{harmonicph}, a  very  recent paper by Parsa and Wang explores yet another notion of harmonic barcodes; see \cite{parsa2024harmonic}.

\subsection*{Acknowledgements}
This work was partially supported by NSF DMS \#2301359, NSF CCF \#2310412, NSF CCF \#2112665, NSF CCF \#2217058, and NSF RI \#1901360.

\section{Preliminaries}\label{sec: prelim}
In this section, we recall the background concepts that will be used in the upcoming sections of the paper.

\paragraph{Grassmannian.} Let $V$ be a finite-dimensional vector space. The set of all $d$-dimensional linear subspaces of $V$ is a well-studied topological space, called the~\emph{$d$-Grassmannian of $V$} and denoted $\gr(d,V)$. As we will be working with linear subspaces with varying dimensions, we consider the disjoint union
\[
\gr(V) := \coprod_{0\leq d\leq \dim(V)} \gr(d,V)
\]
\noindent which is called the~\emph{Grassmannian of $V$}. Note that $\gr(V)$ is closed under the sum of subspaces. That is, for two linear subspaces $W_1, W_2 \subseteq V$, their sum $W_1 + W_2 := \{w_1 + w_2 \mid w_1 \in W_1, w_2 \in W_2 \} \subseteq V$ is also a linear subspace. Thus, the triple $(\gr(V), +, \{ 0\})$ forms a commutative monoid. Moreover, $(\gr(V), \subseteq)$ is a poset. 
\nomenclature[01]{$\gr(V)$}{Grassmannian of $V$}

\begin{framed}
In this paper, we only utilize the monoidal structure and  the natural partial order on $\gr(V)$, without referring to its topology.
\end{framed}

\subsection{Edit Distance} 
The notion of \emph{edit distance} employed in this paper is closely related to the one considered in~\cite{edit}, where the authors introduce their version as the categorification of the Reeb graph edit distance discussed in~\cite{DiFabio2012, DiFabio2016, bauer-edit-2016, Bauer2020}. Let $\mathcal{C}$ be a category and assume that for every morphism $f : A \to B$ in $\mathcal{C}$, there is a cost $c_\mathcal{C} (f) \in \R^{\geq 0}$ associated to it. For two objects $A$ and $B$ in this category, a~\emph{path}, $\mathcal{P}$, is a finite sequence of morphisms
\[
\mathcal{P} : \; \; A \xleftrightarrow{\makebox[1cm]{$f_1$}} D_1 \xleftrightarrow{\makebox[1cm]{$f_2$}} \cdots \xleftrightarrow{\makebox[1cm]{$f_{k-1}$}} D_{k-1} \xleftrightarrow{\makebox[1cm]{$f_k$}} B
\]
where $D_i$s are objects in $\mathcal{C}$ and $\leftrightarrow$ indicates a morphism in either direction. The~\emph{cost of a path} $\mathcal{P}$, denoted $c_\mathcal{C} (\mathcal{P})$, is the sum of the cost of all morphisms in the path.
\[
c_\mathcal{C} (\mathcal{P}) := \sum_{i=1}^k c_\mathcal{C} (f_i).
\]

\begin{definition}[Edit distance]\label{defn: edit dist}
    The~\emph{edit distance} $d_{\mathcal{C}}^E (A,B)$ between two objects $A$ and $B$ in $\mathcal{C}$ is the infimum, over all paths between $A$ and $B$, of the cost of such paths.
    \[
    d_\mathcal{C}^E(A,B) := \inf_{\mathcal{P}} c_\mathcal{C}(\mathcal{P}).
    \]
\end{definition}
\nomenclature[02]{$d_{\mathcal{C}}^E(A,B)$}{Edit distance between objects $A$ and $B$ in a category $\mathcal{C}$}

While the definition of edit distance may appear rather abstract, it is worth noticing that in~\cite[Theorem 9.1]{edit}, the authors constructed a category of persistence diagrams  and established that the edit distance within this context is bi-Lipschitz equivalent to the well-known bottleneck distance~\cite{cohen-steiner2007}.

\paragraph{Poset(s) of Segments.}\label{parag: poset of int}
Let $(P, \leq)$ be a poset. For any $a\leq b$, we refer to pair $(a,b) \in P\times P$ as a \emph{bounded segment}. For every $a\in P$, we introduce a distinguished pair, denoted $(a,\infty)$, and we refer to these pairs as \emph{unbounded segments}. Unbounded segments, as defined here, enable us to conceptualize the absence of a maximum element in a segment, thus allowing the identification of cycles in a filtration that never die (or, die at infinity). We denote the collection of bounded and unbounded segments in $P$ by $\Seg(P)$, which we refer to as the \emph{set of segments} of $P$. We denote by $\diag(P):= \{ (a,a) \in \Seg(P) \mid a \in P \}$ the~\emph{diagonal} of $\Seg(P)$. 
\nomenclature[03]{$\Seg(P)$}{The set of segments of a poset $P$}
\nomenclature[04]{$\diag(P)$}{The diagonal of $\Seg(P)$}

The \emph{product order} on $\Seg (P)$, $\prodord$, is given by the restriction of the product order on $P\times P$ to $\Seg(P)$. More precisely,  
\[
    (b_1,d_1) \prodord (b_2,d_2) \iff b_1\leq b_2 \text{ and } d_1\leq d_2 
\]
where we assume $p < \infty$ for all $p \in P$. We will also use another order on $\Seg (P) \setminus \diag(P)$, called the \emph{reverse inclusion order}. The reverse inclusion order on $\Seg (P)\setminus \diag(P)$, denoted $\incord$, is given by
\[
    (b_1,d_1) \incord (b_2,d_2) \iff b_1\leq b_2 \text{ and } d_1\geq d_2. 
\]
We denote by
\begin{itemize}
\item $\overline{P}^\times := (\Seg (P), \prodord)$ the poset of segments with the product order;
\item  $\overline{P}^\supseteq := (\Seg(P)\setminus \diag(P), \incord)$ the poset of non-diagonal segments with the reverse inclusion order. 
\end{itemize}
\nomenclature[05]{$\prodord$}{The product order on the segments of a poset}
\nomenclature[06]{$\incord$}{The reserve inclusion order on the segments of a poset}
\nomenclature[07]{$\overline{P}^\times$}{The poset of segments of $P$ with the product order}
\nomenclature[08]{$\overline{P}^\supseteq$}{The poset of segments of $P$ with the reverse inclusion order}

\begin{framed}
    In this paper, we exclusively work with finite linear posets, which we denote by $\lp := \{ \ell_1 < \cdots < \ell_n \}$, along with their posets of segments, $\overline{\lp}^\times$ and $\overline{\lp}^\supseteq$, where each segment is denoted as $(\ell_i, \ell_j)$. 
\end{framed}

Notice that if $f: P \to Q$ is an order-preserving map between two posets $P$ and $Q$, then $f$ induces an order-preserving map between the posets of segments $\overline{P}^\times$ and $\overline{Q}^\times$. We denote by $\overline{f} : \overline{P}^\times \to \overline{Q}^\times$ the map induced by $f$ that acts to segments component-wisely. That is, $\overline{f} ((b,d)) := (f(b), f(d))$ and $\overline{f} ((b,\infty)) := (f(b), \infty)$ for every $b\leq d \in P$. Note that with the convention described above, this condition means that $\overline{f}$ preserve the type of segments, i.e., it maps (un)bounded segments to (un)bounded segments. 

\paragraph{Metric Posets.} 
A finite~\emph{(extended) metric poset} is a pair $(P, d_P)$ where $P$ is a finite poset and $d_P : P \times P \to \R \cup \{ \infty \}$ is an (extended) metric
 such that for every $p_1\leq p_2 \leq p_3 \in P$, $d_P (p_1,p_2) \leq d_P (p_1,p_3)$ and $d_P(p_2,p_3) \leq d_P (p_1,p_3)$. 
A \emph{morphism of finite metric posets} $\alpha: (P, d_P) \to (Q, d_Q)$ is an order-preserving map $\alpha: P \to Q$. The \emph{distortion} of a morphism $\alpha : (P, d_P) \to (Q,d_Q)$, denoted $\dis(\alpha)$, is
\[
\dis(\alpha) := \max_{p_1,p_2 \in P} |d_P (p_1,p_2) - d_Q (\alpha(p_1), \alpha(p_2))|.
\]
\nomenclature[09]{$\dis(\cdot )$}{Distortion of a morphism between two metric posets}

For every finite metric poset $(P, d_P)$, its poset of segments $\overline{P}^\times$ is also a metric poset with 
$$d_{\overline{P}^\times} ((b_1,d_1), (b_2,d_2)) := \max \{d_P(b_1,b_2), d_P(d_1,d_2) \}.$$ A morphism of finite metric posets $\alpha : (P, d_P) \to (Q, d_Q)$ induces a morphism of finite metric posets $\overline{\alpha} : \left(\overline{P}^\times ,d_{\overline{P}^\times}\right) \to \left(\overline{Q}^\times, d_{\overline{Q}^\times}\right)$ via $\overline{\alpha} ((b,d)) = (\alpha(b), \alpha(d))$, with $\dis(\alpha) = \dis(\overline{\alpha})$; see  \cite[Proposition 3.4]{edit}.

\paragraph{M\"obius Inversion.} Let $P$ be a poset and $\mathcal{M}$ be a commutative monoid. Let $ \kappa (\mathcal{M})$ be the Grothendieck group completion of $\mathcal{M}$ (see~\cref{appendix:details}). The abelian group $\kappa (\mathcal{M})$ consists of equivalence classes of pairs $(m,n) \in \mathcal{M} \times \mathcal{M}$ under the equivalence relation described in~\cref{appendix:details}. Let $\varphi_\mathcal{M} : \mathcal{M} \to \kappa (\mathcal{M})$ denote the canonical morphism that maps $m\in \mathcal{M}$ to the equivalence class of $(m,0)$ in $\kappa (\mathcal{M})$. Let $m :P \to \mathcal{M}$ be a function. Whenever it exists, we define the~\emph{algebraic M\"obius inverse} of $m$ to be the unique function $\partial_P (m) : P \to \kappa (\mathcal{M}) $ satisfying
    \[
    \sum_{p'\leq p} \partial_P (m) (p') = \varphi_\mathcal{M} (m(p))
    \]
for all $p\in P$. 

\nomenclature[10]{$\partial_P(\cdot)$}{M\"obius inverse of a function defined on a poset $P$}

Let $\mathcal{G}$ be an abelian group. Then, the group completion of $\mathcal{G}$ is isomorphic to $\mathcal{G}$. That is, $\mathcal{G}$ and $\kappa (\mathcal{G})$ can be identified and the canonical map $\varphi_\mathcal{G}$ can be taken as the identity map. In this case, if $g : P \to \mathcal{G}$ is a function, then its algebraic M\"obius inverse, whenever it exists, is the unique function $\partial_P g : P \to \mathcal{G}$ satisfying
    \[
    \sum_{p'\leq p} \partial_P (g) (p') = g(p)
    \]
for all $p\in P$.

\begin{proposition}[{\cite[Proposition 2.6]{gpd-multi}}]\label{prop: algebraic mobius inversion formulas}
    Let $\lp = \{ \ell_1 < \cdots < \ell_n \}$ be a finite linearly ordered set, let $m : \overline{\lp}^\times \to \mathcal{G}$ be any function and let $m' := m|_{\overline{\lp}^\supseteq} : \overline{\lp}^\supseteq \to \mathcal{G}$ be its restriction to non-diagonal points. Then, the algebraic M\"obius inverses of $m$ and $m'$, $\partial_{\overline{\lp}^\times}(m) : \overline{\lp}^\times \to \mathcal{G}$ and $\partial_{\overline{\lp}^\supseteq}(m') : \overline{\lp}^\supseteq \to \mathcal{G}$, are given by
\begin{align}
    \partial_{\overline{\lp}^\times}(m) ((\ell_i, \ell_j)) &= m((\ell_i, \ell_{j})) - m((\ell_i, \ell_{j-1})) + m((\ell_{i-1}, \ell_{j-1})) - m((\ell_{i-1}, \ell_{j})), \label{eqn: times algebraic mob inv1} \\
    \partial_{\overline{\lp}^\times}(m) ((\ell_i, \infty)) &= m((\ell_i, \infty)) - m((\ell_i, \ell_{n})) + m((\ell_{i-1}, \ell_{n})) - m((\ell_{i-1}, \infty)), \label{eqn: times algebraic mob inv2} \\
    \partial_{\overline{\lp}^\times}(m) ((\ell_i, \ell_i)) &= m ((\ell_i, \ell_i)) - m ((\ell_{i-1}, \ell_{i})), \label{eqn: times algebraic mob inv3} \\
    \nonumber \text{and}\\
    \partial_{\overline{\lp}^\supseteq}(m') ((\ell_i, \ell_j)) &= m((\ell_i, \ell_{j})) - m((\ell_i, \ell_{j+1})) + m((\ell_{i-1}, \ell_{j+1})) - m((\ell_{i-1}, \ell_{j})),\label{eqn: times algebraic mob inv4} \\
    \partial_{\overline{\lp}^\supseteq}(m') ((\ell_i, \infty)) &= m((\ell_i, \infty)) - m((\ell_{i-1}, \infty))\label{eqn: times algebraic mob inv5}.
\end{align}
for $1\leq i  < j \leq n$, where we follow the convention that the expressions of the form $(\ell_i, \ell_{n+1})$ are considered to be the segments $(\ell_i ,\infty)$ and the expressions of the form $m((\ell_0, \ell_j))$ and $m((\ell_0 , \infty))$ are assumed to be $0$. 
\end{proposition}

\begin{remark}\label{remark: convention edge cases}
    Breaking down the conventions of the proposition above, we have:
    \begin{align*}
        \partial_{\overline{\lp}^\times}(m) ((\ell_1, \ell_1)) &= m((\ell_1, \ell_{1})), \\
        \partial_{\overline{\lp}^\times}(m) ((\ell_1, \infty)) &= m((\ell_1, \infty)) - m((\ell_1 , \ell_n)), \\
        \partial_{\overline{\lp}^\times}(m) ((\ell_1, \ell_j)) &= m((\ell_1, \ell_{j})) - m((\ell_1, \ell_{j-1})) \text{ for } 1 < j \leq n, \\
        \text{and}\\
        \partial_{\overline{\lp}^\supseteq}(m') ((\ell_1, \infty)) &= m((\ell_1, \infty)), \\
        \partial_{\overline{\lp}^\supseteq}(m') ((\ell_1, \ell_n)) &= m((\ell_1, \ell_{n})) - m((\ell_1,\infty)), \\
        \partial_{\overline{\lp}^\supseteq}(m') ((\ell_1, \ell_j)) &= m((\ell_1, \ell_{j})) - m((\ell_1, \ell_{j+1})) \text{ for } j < n.
    \end{align*}

\end{remark}

\subsection{Galois Connections}

\begin{definition}[Galois connections]\label{defn: galois connections}
    Let $P$ and $Q$ be any two posets (not necessarily finite). A pair, $(\ladj{f}, \radj{f})$, of order-preserving maps, $\ladj{f} : P \to Q$ and $\radj{f} : Q \to P$, is called a~\emph{Galois connection} if they satisfy
    \[
    \ladj{f}(p) \leq q \iff p \leq \radj{f} (q)
    \]
    for every $p\in P$, $q\in Q$. We refer to $\ladj{f}$ as the~\emph{left adjoint} and refer to $\radj{f}$ as the~\emph{right adjoint}. We will also use the notation $\ladj{f} : P \leftrightarrows Q : \radj{f}$ to denote a Galois connection.
\end{definition}

The left and right adjoints of a Galois connection can be expressed in terms of each
other as follows.
\begin{align*}
    \ladj{f} (p) &= \min \{ q\in Q \mid p \leq \radj{f}(q) \} \\
    \radj{f} (q) &= \max \{ p\in P \mid \ladj{f}(p) \leq q \}
\end{align*}

\begin{example}
    Consider the inclusion $\iota: \Z \hookrightarrow \R$, the ceiling function $\lceil \cdot \rceil : \R \to \Z$, and the floor function $\lfloor \cdot \rfloor : \R \to \Z$. The pairs $\left(\lceil \cdot \rceil, \iota\right)$ and $\left(\iota, \lfloor \cdot \rfloor\right)$ are Galois connections.
\end{example}

\begin{remark}
    Notice that the composition of Galois connections $\ladj{f} : P \leftrightarrows Q : \radj{f}$ and $\ladj{g} : Q \leftrightarrows R : \radj{g}$ is also a Galois connection $$\ladj{g} \circ \ladj{f} : P \leftrightarrows R : \radj{f} \circ \radj{g}.$$ Also, note that a Galois connection $\ladj{f} : P \leftrightarrows Q : \radj{f}$ induces a Galois connection on the poset of segments $\overline{\ladj{f}} : \overline{P}^\times \leftrightarrows \overline{Q}^\times :\overline{\radj{f}}$. These properties of Galois connections will later be utilized in defining morphisms in certain categories that we will introduce in~\cref{sec: orthogonal inversion} .
\end{remark}

\begin{definition}[Pushforward and pullback]
Let $f: P \to Q$ be any order-preserving map between two posets, and let $m : P \to \mathcal{G}$ be any function. The~\emph{pushforward} of $m$ along $f$ is the function $f_\sharp m : Q \to \mathcal{G}$ given by
$$f_\sharp m (q) := \sum_{p \in f^{-1}(q)} m(p).$$
Let $h : Q \to \mathcal{G}$ be any function. The~\emph{pullback} of $h$ along $f$ is the function $f^\sharp h : P \to \mathcal{G}$ given by
$$(f^\sharp h )(p) := h( f(p)).$$
\end{definition}
\nomenclature[11]{$(\cdot)_\sharp$}{Pushforward along a map}
\nomenclature[12]{$(\cdot)^\sharp$}{Pullback along a map}

The following theorem, Rota's Galois Connection Theorem (RGCT), describes how M\"obius inversion behaves when a Galois connection exists between two posets.

\begin{theorem}[RGCT~{\cite[Theorem 3.1]{gal-conn}}]\label{thm:rgct}
    Let $P$ and $Q$ be finite posets and $(\ladj{f}, \radj{f})$ be a Galois connection. Then,
    \begin{equation}\label{eqn: rgct push vs pull}
        (\ladj{f})_\sharp \circ \partial_P = \partial_Q \circ (\radj{f})^\sharp.
    \end{equation}
\end{theorem}

\begin{figure}
    \centering
    \includegraphics[scale=20]{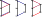}
    \caption{An illustration of RGCT.}
    \label{fig:RGCTdemo}
\end{figure}

\begin{example}\label{ex:gal conn demo}
    Let $P = \{ p_1 < p_2 < p_3 \}$ and $Q = \{ q_1 < q_2\}$ be two posets. Let
    \begin{align*}
        \ladj{f} : P &\to Q & \radj{f}: Q &\to P \\
                   p_1 &\mapsto q_1 & q_1 &\mapsto p_1\\
                   p_2 &\mapsto q_2 & q_2 &\mapsto p_3\\
                   p_3 &\mapsto q_2
    \end{align*}  
    Then, $(\ladj{f}, \radj{f})$ is a Galois connection. In~\cref{fig:RGCTdemo}, $\ladj{f}$ is depicted with the red arrows and $\radj{f}$ is depicted with the blue arrows. In the middle of~\cref{fig:RGCTdemo}, we illustrate two functions. First, $m : P \to \Z$ is a function on $P$, whose values are given by $m(p_1) =1$, $m(p_2) =2$ and $m(p_3) =5$. Second, $(\radj{f})^\sharp m : Q \to \Z$ is a function on $Q$, whose values are given by $(\radj{f})^\sharp m (q_1) := m(\radj{f}(q_1)) = m(p_1) =1$ and similarly $(\radj{f})^\sharp m (q_2):= m(\radj{f}(q_2)) = m(p_3) =5$. On the right of~\cref{fig:RGCTdemo}, we illustrate the M\"obius inverses of these functions, namely, $\partial_P (m)$ and $\partial_Q ((\radj{f})^\sharp m)$. Notice that the pushforwad of $\partial_P (m)$ along $\ladj{f}$ is equal to $\partial_Q ((\radj{f})^\sharp m)$. That is, $\partial_Q ((\radj{f})^\sharp m) = (\ladj{f})_\sharp (\partial_P (m))$ as stated in~\cref{thm:rgct}.
\end{example}

\subsection{Simplicial Complexes and Filtrations}

In this section, we introduce fundamental concepts and definitions, including simplicial complexes, filtrations, persistent Betti numbers, the concepts of birth and death of cycles, and birth-death spaces.

\paragraph{Simplicial Complexes and Chain Spaces.} An (abstract) \emph{finite simplicial complex} $K$ over a finite ordered vertex set $V$ is a non-empty collection of non-empty subsets of $V$ with the property that for every $\sigma \in K$, if $\tau\subseteq \sigma$, then $\tau \in K$. An element $\sigma \in K$ is called a \emph{$\dgr$-simplex} if the cardinality of $\sigma$ is $\dgr+1$. An \emph{oriented simplex}, denoted $[\sigma]$, is a simplex $\sigma \in K$ whose vertices are ordered. We always assume that ordering on simplices is inherited from the ordering on $V$. Let $\mathfrak{s}_\dgr^K$ denote the set of all oriented $\dgr$-simplices of $K$. 
\nomenclature[13]{$\mathfrak{s}_\dgr^K$}{set of all oriented $\dgr$-simplices of a simplicial complex $K$}

The \emph{$\dgr$-th chain space of $K$}, denoted $C_\dgr^K$, is the vector space over $\R$ with basis $\mathfrak{s}_\dgr^K$. Let $n_\dgr^K := |\mathfrak{s}_\dgr^K| = \dim_\R (C_\dgr^K)$. The \emph{$\dgr$-th boundary operator} $\partial_\dgr^K : C_\dgr^K \to C_{\dgr-1}^K$ is defined by
\[
    \partial_\dgr^K ([v_0,\ldots,v_\dgr]) := \sum_{i=0}^{\dgr}(-1)^i[v_0,\ldots,\hat{v}_i,\ldots,v_\dgr]
\]
for every oriented $\dgr$-simplex $[\sigma] = [v_0,\ldots,v_\dgr]\in \mathfrak{s}_\dgr^K$, where $[v_0,\ldots,\hat{v}_i,\ldots,v_\dgr]$ denotes the omission of the $i$-th vertex, and extended linearly to $C_\dgr^K$. We denote by $\Zfunc_\dgr(K)$ the space of $\dgr$-cycles of $K$, that is $$\Zfunc_\dgr(K) := \ker \left(\partial_\dgr^K\right),$$ and we denote by $\Bfunc_\dgr(K)$ the space of $\dgr$-boundaries of $K$, that is $$\Bfunc_\dgr(K):= \Ima \left(\partial_{\dgr+1}^K\right).$$ Additionally, we denote by $H_\dgr(K)$ the $\dgr$-th homology group of $K$, that is $$H_\dgr(K) := \frac{\Zfunc_\dgr(K)}{ \Bfunc_\dgr(K)}.$$
\nomenclature[14]{$C_\dgr^K$}{$\dgr$-th chain group of $K$}
\nomenclature[15]{$\partial_\dgr^K$}{$\dgr$-th boundary operator of a simplicial complex $K$}
\nomenclature[16]{$\Zfunc_\dgr(\cdot)$}{Space of $\dgr$-cycles}
\nomenclature[17]{$\Bfunc_\dgr(\cdot)$}{Space of $\dgr$-boundaries}
\nomenclature[18]{$H_\dgr(\cdot)$}{$\dgr$-th homology group}

For each integer $\dgr\geq 0$, we define an inner product, $\langle \cdot , \cdot \rangle_{C_\dgr^K}$, on $C_\dgr^K$ as follows:
\[
\langle [\sigma] , [\sigma'] \rangle_{C_\dgr^K} := \delta_{[\sigma], [\sigma']} \text{, for all } [\sigma],[\sigma'] \in \mathfrak{s}_\dgr^K,
\]
where $\delta_{\bullet, \bullet}$ is the Kronecker delta.
That is, we declare that $\mathfrak{s}_\dgr^K$ is an orthonormal basis for $C_\dgr^K$. We will refer to $\langle \cdot , \cdot \rangle_{C_\dgr^K}$ as the~\emph{standard} inner product on $C_\dgr^K$. We will omit the subscript from the notation $\langle \cdot , \cdot \rangle_{C_\dgr^K}$ when the context is clear. We denote by $\left(\partial_\dgr^K\right)^* : C_{\dgr-1}^K \to C_\dgr^K$ the adjoint of $\partial_\dgr^K$ with respect to the standard inner products  on $C_\dgr^K$ and $C_{\dgr-1}^K$.
\nomenclature[19]{$\langle \cdot , \cdot \rangle_{C_\dgr^K}$}{Standard inner product on $C_\dgr^K$}

\paragraph{Simplicial Filtrations.}\label{parag: simplicial filtration} For a finite simplicial complex $K$, let $\subcx(K)$ denote the poset of subcomplexes of $K$, ordered by inclusion. A \emph{simplicial filtration of $K$} is an order-preserving map $\Ffunc : P \to \subcx (K)$, where $P$ is a finite poset. 
A \emph{$1$-parameter} filtration of $K$ is a filtration $\Ffunc : \lp \to \subcx(K)$ where $\lp = \{\ell_1<\cdots<\ell_n \}$ is a finite linearly ordered set and $\Ffunc(\ell_n) = K$. When $\Ffunc : \{ \ell_1<\cdots<\ell_n\} \to \subcx (K)$ is a $1$-parameter filtration, we use the notation $K_i$ for the simplicial complex $\Ffunc(\ell_i)$ and succinctly write $\Ffunc = \{K_i \}_{i=1}^n$ to denote the simplicial filtration. Note that $K_n = K$. 
\nomenclature[20]{$\subcx(\cdot)$}{Poset of subcomplexes of a simplicial complex, ordered by inclusion}

\begin{definition}[Persistent Betti numbers / rank invariant \cite{Edelsbrunner2002}]\label{defn: persistent betti rank inv}
    Let $\Ffunc : P \to \subcx (K)$ be a filtration. For $(p,p')\in \Seg(P)$, let $\iota_\dgr^{p,p'} : H_\dgr (\Ffunc(p)) \to H_\dgr (\Ffunc(p'))$ denote the homomorpshim induced by the inclusion $\Ffunc(p) \hookrightarrow \Ffunc(p')$. We define the \emph{$\dgr$-th persistent Betti number} for the segment $(p,p')$ as
    \[
    \beta_\dgr^{p,p'} := \rank \left(\iota_\dgr^{p,p'}\right).
    \]
    As is customary in applied algebraic topology, we also use the term \emph{rank invariant} to refer to persistent Betti numbers. 
\end{definition}
\nomenclature[21]{$\beta_\dgr^{\cdot, \cdot}$}{$\dgr$-th persistent Betti numbers}

Let $\Ffunc = \{K_i \}_{i=1}^n$ be a $1$-parameter filtration of $K$. Observe that, for any dimension $\dgr\geq 0$, the inclusion of simplicial complexes $K_i \subseteq K_j$, for $i \leq j$, induces canonical inclusions on the cycle and boundary spaces. In particular, for any $i = 1,\ldots,n$, the $\dgr$-th cycle and boundary spaces of $K_i$ can be identified with subspaces of $C_\dgr^{K_n} = C_\dgr^K$:

\begin{center}
    \begin{tikzcd}
\Zfunc_\dgr(K_i) \arrow[r, hook]                 & \Zfunc_\dgr(K_j) \arrow[r, hook]                 & \Zfunc_\dgr(K) \arrow[r, hook] & C_\dgr^K \\
\Bfunc_\dgr(K_i) \arrow[r, hook] \arrow[u, hook] & \Bfunc_\dgr(K_j) \arrow[r, hook] \arrow[u, hook] & \Bfunc_\dgr(K) \arrow[u, hook] &      
    \end{tikzcd}
\end{center}

\begin{definition}[Birth-death spaces]\label{defn: bd space}
    Let $\Ffunc : P \to \subcx (K)$ be a filtration. For any degree $\dgr\geq 0$, the $\dgr$-th~\emph{birth-death spaces} associated to $\Ffunc$ is defined as the function $\ZB_\dgr^\Ffunc : \overline{P}^\times\to \gr\left(C_\dgr^K\right)$ given by
    \begin{align*}
        \ZB_\dgr^\Ffunc ((b, d)) &:= \Zfunc_\dgr \big(\Ffunc(b)\big)\cap \Bfunc_\dgr \big(\Ffunc(d)\big), \\
        \ZB_\dgr^\Ffunc ((b,\infty)) &:= \Zfunc_\dgr \big(\Ffunc(b)\big).
    \end{align*} 

\end{definition}
\nomenclature[22]{$\ZB_\dgr^\Ffunc$}{$\dgr$-th birth-death spaces associated to a filtration $\Ffunc$}

Informally, when $b\leq d$, for a cycle $z$ to be in the birth-death space $\ZB_\dgr^\Ffunc((b,d))$ means that $z$ becomes ``alive" at or before $b$ and that it ``dies''  (i.e. it becomes a boundary) at or before $d$.  

\begin{remark}
The classical definition of persistence diagrams, as in~\cite{cohen-steiner2007}, utilizes persistent Betti numbers. Birth-death spaces were introduced in~\cite{edit, saecular} as an alternative way to define persistence diagrams. For a $1$-parameter filtration, classical persistence diagrams and persistence diagrams obtained through utilizing the dimension of birth-death spaces coincide as shown in~\cite[Section 9.1]{edit}. McCleary and Patel showed that the use of birth-death spaces for defining persistence diagrams enables us to organize the persistent homology pipeline in a functorial way~\cite{edit}. Additionally, this functoriality leads to the edit distance stability of persistence diagrams as in~\cite[Theorem 8.4]{edit}.
\end{remark}

We now recall the definition of \emph{lifetime of cycles} (i.e., birth time and death time) and \emph{ephemeral cycles} from~\cite{gpd-multi}, which are formulated using the birth-death spaces $\ZB_\dgr^\Ffunc$. Note that in TDA, the death time is typically used in reference to  homology classes as opposed to cycles and 
    the death time of a homology class refers to the first time when the class merges with an ``older" one, following the ``elder rule" \cite{edelsbrunner2010computational,curry2018fiber}. We refer to~\cite[Remark 2.16]{gpd-multi} for the motivation behind the following definition.

\begin{definition}[Lifetime of cycles / ephemeral cycles]\label{def:bd-cycles} 
    Let $P$ be any finite poset and let $\Ffunc : P \to \subcx(K)$ be a filtration. Let $(b,d) \in \Seg(P)$. We say that a nonzero cycle $z \in C_\dgr^K$ has \emph{a lifetime $(b,d)$} if the following two conditions are met:

    \begin{itemize}
        \item $z \in \ZB_\dgr^\Ffunc((b,d))$, and
        \item $z \notin  \sum_{(a,c)<_\times (b,d)} \ZB_\dgr^\Ffunc((a,c))$. 
    \end{itemize}

    When a cycle $z$ is born at $b$ and dies at $b$ (i.e., $b=d$), we say that $z$ is an \emph{ephemeral} cycle.

\end{definition}

\begin{remark}
    For a $1$-parameter filtration $\Ffunc : \lp = \{ \ell_1 < \cdots < \ell_n \} \to \subcx(K)$, we will see in~\cref{prop: dim of loi is classical pd} that the number of linearly independent cycles that are born at $\ell_j$ and die at $\ell_i$ (with $i<j$), which is given by $$\dim \left(\frac{\ZB_\dgr^\Ffunc((\ell_i,\ell_j))}{\sum_{(\ell_k, \ell_l) <_\times (\ell_i, \ell_j)}\ZB_\dgr^\Ffunc((\ell_k, \ell_l))}\right)$$
    is precisely the the multiplicity of the segment $(\ell_i,\ell_j)$ in the classical degree-$\dgr$ persistence diagram of $\Ffunc$.
\end{remark}

\subsection{Monoidal M\"obius Inverses and Orthogonal Inversion}\label{sec: monoidal mobious inversion}

The ``algebraic" M\"obius inverse of a function $m : P \to \mathcal{M}$, where $\mathcal{M}$ is a commutative monoid, involves the group completion of $\mathcal{M}$, denoted $\kappa(\mathcal{M})$. This is required in order to ``make sense'' of the minus operations that may appear in M\"obius inversion formulas such as the one in~\cref{prop: algebraic mobius inversion formulas}. However, the group completion of a commutative monoid could be the trivial group, yielding a trivial algebraic M\"obius inverse. In particular, for any vector space $V$, the group completion of $\gr(V)$ is the trivial group; see~\cref{appendix:details}. In this case, the algebraic M\"obius inverse of any map $m : P \to \gr(V)$ is the trivial map $ \partial_P (m) : P \to \{ 0 \} = \kappa (\gr(V))$. This suggests considering a notion of M\"obius inverse that does not involve group completion.

\begin{definition}[Monoidal M\"obius inverses~{\cite[Definition 2.20]{gpd-multi}}]\label{defn: monoidal mobius inversion}
    Let $\mathcal{M}$ be a commutative monoid. Let $m : P \to \mathcal{M}$ be a function, then a function $m' : P \to \mathcal{M}$ is called a~\emph{monoidal M\"obius inverse} of $m$ if it satisfies 
    \[
    \sum_{p'\leq p} m'(p') = m(p)
    \]
    for all $p \in P$. We denote by $\mmi{P}{m}$ the set of all monoidal M\"obius inverses of $m$.
\end{definition}
\nomenclature[23]{$\mmi{P}{\cdot}$}{Set of all monoidal inverses of a function defined on a poset $P$}

Notice that if $\mathcal{G}$ is an abelian group and $g : P \to \mathcal{G}$ is any function, then the algebraic M\"obius inverse of $g$ is a monoidal M\"obius inverse of $g$. Indeed, the algebraic M\"obius inverse of $g$ is the unique monoidal M\"obius inverse of $g$ in this case. However, if $\mathcal{M}$ is a commutative monoid that is not an abelian group and $m : P \to \mathcal{M}$ is a function, then the algebraic M\"obius inverse of $m$ and a monoidal M\"obius inverse of $m$ have different codomains as functions. 

While the algebraic M\"obius inverse of $m$ is always guaranteed to exist when $P$ is finite, a monoidal M\"obius inverse of $m$ might not exist even in this case~\cite[Example 2.21]{gpd-multi}. Moreover, in the case when both inverses exist,  there might be more than one monoidal M\"obius inverse of $m$ whereas the algebraic M\"obius inverse is necessarily unique~\cite[Example 2.22]{gpd-multi}. 

The fact that the monoidal M\"obius inverse may not be unique, as demonstrated in~\cite[Example 2.22]{gpd-multi}, motivates the following definition, which introduces an equivalence relation linking all functions serving as M\"obius inverses of the same function.

\begin{definition}[M\"obius equivalence~{\cite[Definition 2.23]{gpd-multi}}]
    Two functions $m_1, m_2 : P \to \mathcal{M}$ are said to be \emph{M\"obius equivalent} if
    \[
    \sum_{p'\leq p} m_1 (p') = \sum_{p'\leq p} m_2 (p')
    \]
    for all $p\in P$. In this case, we write $m_1 \mobeq m_2$.
\end{definition}
\nomenclature[24]{$\mobeq$}{M\"obius equivalence}

The following definition, introduced in \cite{gpd-multi}, is used to construct a monoidal Möbius inverse for order-preserving functions from a finite poset to the Grassmannian of an inner product space.

\begin{definition}[Difference of subspaces~{\cite[Definition 3.2]{gpd-multi}}]\label{defn: difference of spaces}
    Let $V$ be an inner product space and let $W_1, W_2 \subseteq V$ be subspaces. We define the~\emph{difference} of two subspaces as
    \[
    W_1 \ominus W_2 := W_1 \cap W_2^\perp.
    \]
\end{definition}
\nomenclature[25]{$\ominus$}{Difference of subspaces}

\begin{remark}\label{remark: dimension difference}
    In the definition above, if $W_2 \subseteq W_1$, then $W_1 \ominus W_2 = W_1 \cap W_2^\perp$ is the orthogonal complement of $W_2$ inside of $W_1$. In this case, $\dim (W_1 \ominus W_2) = \dim W_1 - \dim W_2$. 
    In general (i.e. when $W_2 \nsubseteq W_1$), we have that 
    \[
    W_1 \ominus W_2 = W_1 \ominus \proj_{W_1} (W_2),
    \]
    where $\proj_{W_1} : V \to W_1$ is the orthogonal projection. This can be informally interpreted as expressing that $\proj_{W_1} (W_2)$ and $W_2$ are treated as being ``quasi-isomorphic'' with respect to $W_1$. See~\cref{appendix: details for orthogonal inversion} for the proof of the equality $W_1 \ominus W_2 = W_1 \ominus \proj_{W_1} (W_2) $.
\end{remark}

Although this paper focuses on a notion of M\"obius inversion on the poset of segments of a linear poset, it is, in fact, a special case of a broader construction, which we now recall.

\begin{definition}[Orthogonal Inversion~{\cite[Definition 3.4]{gpd-multi}}]\label{defn: goi}
    Let $R$ be a finite poset and let $\Ffrak: R \to\gr(V)$ be an order-preserving function. We define the \emph{Orthogonal Inverse of $\Ffrak$} to be the function $\goi( \Ffrak) : R \to \gr(V)$ given by
    \[
    \goi (\Ffrak) (r) := \Ffrak(r) \ominus \left(\sum_{r'< r} \Ffrak(r')\right).
    \]
\end{definition}
\nomenclature[26]{$\goi$}{Orthogonal Inversion}

As we will see in \cref{prop: loi and goi agrees}, one of the main constructions in this paper, namely $\times$-Linear Orthogonal Inversion (\cref{defn: x-harmonic inversion}), is actually a special case of Orthogonal Inversion (\cref{defn: goi}). We will therefore leverage some of the properties that Orthogonal Inversion satisfies, one of which is the following.

\begin{proposition}[{\cite[Proposition 3.6]{gpd-multi}}]\label{prop: goi is monoidal mi}
    Let $R$ be a finite poset and let $\Ffrak : R \to \gr(V)$ be an order-preserving function. Then, $\goi (\Ffrak)$ is a monoidal M\"obius inverse of $\Ffrak$, i.e., $\goi (\Ffrak) \in \mmi{R}{\Ffrak}$. That is,
    \[
    \sum_{r' \leq r} \goi (\Ffrak)(r') = \Ffrak(r)
    \]
    for every $r\in R$.
\end{proposition}

\subsection{A Monoidal Rota's Galois Connection Theorem}

Rota's Galois Connection Theorem (RGCT)~\cite{gal-conn} describes how algebraic
M\"obius inversion behaves when a Galois connection exists between two posets. The functoriality of algebraic M\"obius inversion~\cite{edit, gal-conn} is indeed a direct consequence of the RGCT.
Now, we recall a monoidal analog of the RGCT. The Monoidal RGCT has been utilized to conclude functoriality of certain constructions in~\cite{gpd-multi} 
 and will allow us to establish the functoriality of our constructions in~\cref{sec: orthogonal inversion}.

\begin{theorem}[Monoidal RGCT~{\cite[Theorem 2]{gpd-multi}}]\label{thm: monoidal rgct}
    Let $P$ and $Q$ be finite posets, $\ladj{f} : P \leftrightarrows Q : \radj{f}$ be a Galois connection, and $m: P \to \mathcal{M}$ be any function. Assume that $m' : P \to \mathcal{M}$ is a monoidal M\"obius inverse of $m$. Then, $(\ladj{f})_\sharp (m')$ is a monoidal M\"obius inverse of $(\radj{f})^\sharp m$, i.e., $(\ladj{f})_\sharp (m') \in \mmi{Q}{(\radj{f})^\sharp m}$.

\end{theorem}

\begin{example}[Monoidal RGCT]
    Let $P = \{p_1 < p_2 <p_3\}$, $Q = \{ q_1 <q_2\}$ and $\ladj{f} : P \leftrightarrows Q : \radj{f}$ be as in~\cref{ex:gal conn demo} (which are illustrated in~\cref{fig:RGCTdemo}). Let $m$ be the function defined by
    \begin{align*}
        m : P &\to \gr(\R^3) \\
            p_1 &\mapsto \spn \{e_1 \} \\
            p_2 &\mapsto \spn \{e_1, e_2 \} \\
            p_3 &\mapsto \spn \{e_1, e_2, e_3 \}.
    \end{align*}
    Observe that the function defined by 
    \begin{align*}
        m' : P &\to \gr(\R^3) \\
                                    p_1 &\mapsto \spn \{e_1 \} \\
                                     p_2 &\mapsto \spn \{ e_2 \} \\
                                     p_3 &\mapsto \spn \{e_3 \}.
    \end{align*}
    is a monoidal M\"obius inverse of $m$, i.e. $m' \in \mmi{P}{m}$. Also, the function defined by
    \begin{align*}
        n : Q &\to \gr(\R^3) \\
                                    q_1 &\mapsto \spn \{e_1 \} \\
                                     q_2 &\mapsto \spn \{ e_2+e_1, e_3+e_1 \}
    \end{align*}
    is a monoidal M\"obius inverse of $(\radj{f})^\sharp m$, i.e., $n\in \mmi{Q}{(\radj{f})^\sharp m}$. Observe that $(\ladj{f})_\sharp m' $ and $ n$ do not coincide as functions because
    \[
    (\ladj{f})_\sharp (\radj{f})^\sharp m (q_2) = \spn \{ e_2, e_3 \} \neq \spn \{ e_2+e_1, e_3+e_1 \} = n (q_2).
    \]
    Nevertheless, it holds that 
    \[
    (\ladj{f})_\sharp m \in \mmi{Q}{(\radj{f})^\sharp m},
    \]
    as stated in~\cref{thm: monoidal rgct}.
\end{example}

\section{~\texorpdfstring{$\times$}{}-Linear Orthogonal Inversion on ~\texorpdfstring{$\Seg(\lp)$}{}}\label{sec: orthogonal inversion}

In this section, we introduce the notion of \emph{$\times$-Linear Orthogonal Inversion}, a notion analogous to classical M\"obius inversion on the poset of segments of a finite linear poset (with the product order). Let $\lp$ be a finite linear poset, $\overline{\lp}^\times = (\Seg(\lp), \prodord)$ be the poset of segments of $\lp$ with the product order, and $V$ be a finite-dimensional inner product space. The $\times$-Linear Orthogonal Inversion, denoted $\prhi$, takes an order-preserving function $\overline{\Ffunc} : \overline{\lp}^\times \to \gr(V)$, subject to an intersection property described in~\cref{def: int-monotone space function}, as input. It then produces an output function $\prhi \left(\overline{\Ffunc}\right) : \overline{\lp}^\times \to \gr(V)$ which satisfies a certain transversality condition specified in~\cref{defn: grassmannian persistence diagrams}.

Our main results in this section are the functoriality (\cref{prop: prhi is functor}) and stability (\cref{thm: general stability}) of $\times$-Linear Orthogonal Inversion. In~\cref{subsec: int-monotone and gpd categories}, we introduce the source and the target categories on which $\prhi$ operates, while in~\cref{subsec: times orthogonal inversion functor}, we delve into the construction of $\prhi$ and provide proofs of its functoriality and stability.

\begin{remark}
    The key distinction between the more general notion of Orthogonal Inversion, $\goi$, and the approach we take in this section with $\prhi$ lies in the structure of 1-parameter filtrations. In the setting of $1$-parameter filtrations, the birth-death spaces satisfy a specific intersection property, which we abstract and formalize in~\cref{def: int-monotone space function}. This intersection property, in turn, ensures a notion of transversality for $1$-parameter Grassmannian persistence diagrams. While a weaker form of transversality applies to the general Orthogonal Inversion for order-preserving functions from arbitrary posets to the Grassmannian of an inner product space (\cite[Proposition 3.9]{gpd-multi}), the transversality condition satisfied by 1-parameter Grassmannian persistence diagrams establishes a crucial connection to classical persistence diagrams. In fact, lower bounding the edit distance between 1-parameter Grassmannian persistence diagrams using the edit distance between classical persistence diagrams (\cref{thm: classical pd lower bound}) relies on this stronger transversality property enjoyed by 1-parameter Grassmannian persistence diagrams.
\end{remark}

\subsection{Source and Target Categories}\label{subsec: int-monotone and gpd categories}

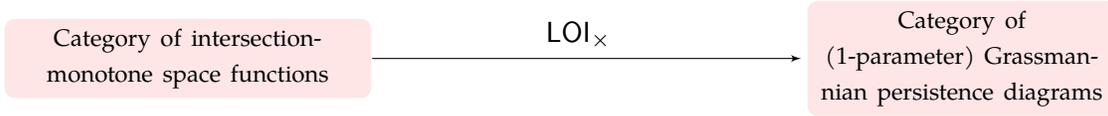
\begin{figure}[ht]

    \centering
\tikzstyle{l} = [draw, -latex',thick]
\begin{tikzpicture}[auto,
 block/.style ={rectangle, thick, fill=red!10, text width=10.0em,align=center, rounded corners},
 longblock/.style ={rectangle, thick, fill=red!10, text width=12.0em,align=center, rounded corners},
 boundblock/.style ={rectangle, thick, fill=red!10, text width=5.5em,align=center, rounded corners},
 longboundblock/.style ={rectangle, thick, fill=red!10, text width=7em,align=center, rounded corners},
 line/.style ={draw, -latex', shorten >=2pt},
  ]
    \node [longblock] (inputofmob)  {\footnotesize Category of intersection-monotone  space functions };
    \node [block, right =15em of inputofmob] (outputofmob) {\footnotesize Category of \\ ($1$-parameter) Grassmannian persistence diagrams};
    \path [line, black] (inputofmob) -- node [text width=3em, above] {{$\prhi$} }(outputofmob); 
\end{tikzpicture}
\caption{Source and target categories of $\times$-Linear Orthogonal Inversion.}
\label{fig: times-orthogonal-pipeline}
\end{figure}

In this section, we introduce two categories: the category of intersection-monotone space functions, and the category of $1$-parameter Grassmannian persistence diagrams over a fixed finite-dimensional inner product space $V$. These categories, as depicted in~\cref{fig: times-orthogonal-pipeline}, will serve as the source and the target of a functor, namely $\prhi$, that we will construct in~\cref{subsec: times orthogonal inversion functor}.

\begin{definition}[Intersection-monotone space functions]\label{def: int-monotone space function}
    Let $(\lp,d_\lp)$ be a metric poset where $\lp$ is a finite linear poset.  Write $\lp = \{ \ell_1 < \cdots < \ell_n \}$.  A function $\overline{\Ffunc} : \overline{\lp}^\times \to \gr(V)$ is called an~\emph{intersection-monotone space function} if 
    \begin{enumerate}
        \item $\overline{\Ffunc}$ is order preserving. That is, for every $I \prodord J \in \overline{\lp}^\times$, it holds that $\overline{\Ffunc}(I) \subseteq \overline{\Ffunc}(J)$ \label{cond: order-pres}
        \item For all $1 \leq i < j \leq n$, it holds that
    \[
    \overline{\Ffunc} ((\ell_{i+1}, \ell_j)) \cap \overline{\Ffunc} ((\ell_i, \ell_{j+1})) = \overline{\Ffunc} ((\ell_{i}, \ell_{j})) 
    \] \label{cond: intersecion}
    \end{enumerate}
    where $(\ell_i, \ell_{n+1}) := (\ell_i, \infty)$ by convention.
    
\end{definition}

\begin{figure}[th]
    \centering
    \includegraphics[scale=20]{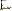}
    \caption{Two ``natural and independent'' directions in the poset of segments are depicted.}
    \label{fig:directonsInt}
\end{figure}

    The conditions 1 and 2 in~\cref{def: int-monotone space function} represent properties of birth-death spaces of a $1$-parameter filtration. We abstract and formalize these concepts to define intersection-monotone space functions. Therefore, intersection-monotone space functions can be seen as a generalization of birth-death spaces of a $1$-parameter filtration.

\begin{remark}[Note about the metric on the poset]
    For the several results that follow, the metric $d_\lp$ on the poset $\lp$ is irrelevant. We only require a metric structure on $\lp$ to ensure that the relevant functions associated with these statements fall within  appropriate categories.  We eventually exploit the metric structure in proving the stability result; see~\cref{thm: general stability}.
\end{remark}

\begin{remark}
    The intersection condition in~\cref{def: int-monotone space function} can be interpreted as follows. In the poset of segments $\overline{\lp}^\times$, there are two ``natural and independent'' directions towards which the order increases, up and right. Namely, for a segment $(\ell_i, \ell_j) \in \overline{\lp}^\times$, the segment $(\ell_i, \ell_{j+1})$ is one unit above $(\ell_i, \ell_j)$. Similarly, the segment $(\ell_{i+1}, \ell_{j})$ is one unit to the right of $(\ell_i, \ell_j)$, see~\cref{fig:directonsInt}. When $\overline{\Ffunc} : \overline{\lp}^\times \to \gr(V)$ is an order-preserving map, it already holds that 
    \[
    \overline{\Ffunc} ((\ell_{i+1}, \ell_j)) \cap \overline{\Ffunc} ((\ell_i, \ell_{j+1})) \supseteq \overline{\Ffunc} ((\ell_{i}, \ell_{j})) 
    \]
    as $\overline{\Ffunc} ((\ell_{i+1}, \ell_j)) \supseteq \overline{\Ffunc} ((\ell_{i}, \ell_{j}))$ and $\overline{\Ffunc} ((\ell_{i}, \ell_{j+1})) \supseteq \overline{\Ffunc} ((\ell_{i}, \ell_{j}))$. Then, the condition
    \[
    \overline{\Ffunc} ((\ell_{i+1}, \ell_j)) \cap \overline{\Ffunc} ((\ell_i, \ell_{j+1})) = \overline{\Ffunc} ((\ell_{i}, \ell_{j})) 
    \]
    can be interpreted as expressing the property that the two enlargements of $\overline{\Ffunc} ((\ell_{i}, \ell_{j}))$ in two independent directions, up and right, are also independent.
\end{remark}

\begin{definition}[Intersection-monotone space preserving morphism]
    An~\emph{intersection-monotone space preserving morphism} from an intersection-monotone space function $\overline{\Ffunc} : \overline{\lp_1}^\times \to \gr(V)$ to another intersection-monotone space function $\overline{\Gfunc} : \overline{\lp_2}^\times \to \gr(V)$ is any Galois connection $f = (\ladj{f}, \radj{f})$, $\ladj{f} : \lp_1 \leftrightarrows \lp_2 : \radj{f}$ such that
    \[
    \overline{\Ffunc} \circ \overline{\radj{f}} = \overline{\Gfunc},
    \]
    where $\overline{\radj{f}} : \overline{\lp_2}^\times \to \overline{\lp_1}^\times$ is the order-preserving map on the poset of segments induced by $\radj{f}$. 
\end{definition}

\begin{notation}[$\inn (V)$]
    We denote by $\inn (V)$ the category where
    \begin{itemize}
        \item Objects are intersection-monotone space functions,
        \item Morphisms are intersection-monotone space preserving morphisms. 
    \end{itemize} 
\end{notation}
\nomenclature[27]{$\inn(\cdot)$}{Category of intersection-monotone space functions}

\begin{definition}[Cost of a morphism in $\inn (V)$]
    The \emph{cost} of a morphism $f = (\ladj{f},\radj{f})$ in $\inn(V)$, denoted $\cost_{\inn(V)} (f)$, is defined to be $\cost_{\inn(V)} (f) := \dis(\ladj{f})$, the distortion of the left adjoint $\ladj{f}$. 
\end{definition}

\begin{remark}
    In an analogous situation, in~\cite{edit}, the authors define the cost of a morphism $f = (\ladj{f},\radj{f})$ via the distortion of $\overline{\ladj{f}}$ and repeatedly rely on the equality $\dis(\ladj{f}) = \dis(\overline{\ladj{f}})$ (see~\cite[{Proposition 3.4}]{edit}). In this paper, however, we consistently define the cost of a morphism directly through the distortion of $\ladj{f}$, thereby eliminating the need for repeated references to~\cite[{Proposition 3.4}]{edit}.
\end{remark}

\begin{figure}[t]
    \centering
    \includegraphics[scale=12]{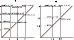}
    \caption{Two intersection-monotone space functions $\overline{\Ffunc}$ and $\overline{\Gfunc}$ are shown. The domain of $\overline{\Ffunc}$ is the poset of segments of $\{ 1 < 2 < 3 \} $, and the domain of $\overline{\Gfunc}$ is the poset of segments of $\{ 1.5 < 2.5 \}$. The values of $\overline{\Ffunc}$ and $\overline{\Gfunc}$ are displayed on top of the corresponding points on the diagrams.}
    \label{fig:monotone space functions}
\end{figure}

\begin{example}[Intersection-monotone space functions]\label{ex: monotone space functions and morphism}
    Let $V = \R^4$ with the standard inner product and let $\{e_i\}_{i=1}^4$ denote the canonical basis elements for $\R^4$ and let $\lp_1 = \{1<2<3 \}$ and $\lp_2 = \{1.5<2.5 \}$ both with the restriction of the Euclidean distance on the real line. In~\cref{fig:monotone space functions}, we illustrate two intersection-monotone space functions $\overline{F}$ and $\overline{G}$. The diagram on the left of~\cref{fig:monotone space functions} is a visualization of an intersection-monotone space function $\overline{\Ffunc} : \overline{\lp_1}^\times \to \gr(\R^4)$. The subspace at each point on the diagram represents the value of $\overline{\Ffunc}$ on the corresponding segment. The diagram on the right of~\cref{fig:monotone space functions} is another intersection-monotone space function $\overline{\Gfunc} : \overline{\lp_2}^\times \to \gr(\R^4)$. Both $\overline{\Ffunc}$ and $\overline{\Gfunc}$ satisfy the intersection condition of~\cref{def: int-monotone space function}. Let 
    \begin{align*}
        \ladj{f} : \lp_1 &\to \lp_2 & \radj{f}: \lp_2 &\to \lp_1 \\
                   1 &\mapsto 1.5 & 1.5 &\mapsto 1\\
                   2 &\mapsto 2.5 & 2.5 &\mapsto 3\\
                   3 &\mapsto 2.5
    \end{align*}    
    Then, $f:= (\ladj{f}, \radj{f})$ is Galois connection that determines an intersection-monotone space preserving morphism from $\overline{\Ffunc}$ to $\overline{\Gfunc}$ as it holds that $\overline{\Ffunc} \circ \overline{\radj{f}} = \overline{\Gfunc}$. And, we have that $\cost_{\inn(V)} (f) = \dis(\ladj{f}) = 1$.
\end{example}

Below, we introduce the category of \emph{$1$-parameter Grassmannian persistence diagrams}. This category will be the target category of the functor that we will define in~\cref{subsec: times orthogonal inversion functor}. We first introduce the notion of \emph{transversity} --- a notion that will be utilized to define $1$-parameter Grassmannian persistence diagrams.

\begin{definition}[Transversity]\label{defn: transversity}
    Let $V$ be any finite-dimensional vector space.
    Two families $\{ W_i\}_{i=1}^m$ and $\{ U_j \}_{j=1}^m$ of subspaces of $V$ are said to be \emph{transversal to each other} (or that $\{ W_i\}_{i=1}^n$ is transversal to $\{ U_j \}_{j=1}^m$, and vice versa) if 
    \[
    \dim \left ( \sum_{i=1}^n W_i + \sum_{j=1}^m U_j \right ) = \sum_{i=1}^n \dim (W_i) + \sum_{j=1}^m \dim (U_j).
    \]
    A family $\{W_i \}_{i=1}^n$ of subspaces of $V$ is called a~\emph{transverse family} if it is transversal to $\{ \{0 \} \}$, where $\{ 0 \} \subseteq V$ is the zero subspace of $V$.

\end{definition}

\begin{remark}
    Note that, a family $\{ W_i\}_{i=1}^n$ is transverse if and only if
    \[
    \dim \left ( \sum_{i=1}^n W_i \right ) = \sum_{i=1}^n \dim (W_i).
    \]
    Hence, one can see that~\cref{defn: transversity} generalizes the notion of transversity described in~\cite[Definition 2.1]{gpd-multi}. Note also that if two families $\{ W_i\}_{i=1}^n$ and $\{ U_j \}_{j=1}^m$ are transversal to each other, then the families $\{ W_i\}_{i=1}^n$ and $\{ U_j \}_{j=1}^m$ are transverse families.
\end{remark}

\begin{definition}[$1$-Parameter Grassmannian persistence diagram]\label{defn: grassmannian persistence diagrams}
    Let $(\lp,d_\lp)$ be a metric poset where $\lp$ is any finite linear poset. A function $\Mfunc : \Seg(\lp) \to \gr(V)$ is called a~\emph{$1$-parameter Grassmannian persistence diagram} whenever $\{ \Mfunc (I) \}_{I\in \Seg(\lp)}$ is a transverse family.
\end{definition}

In this paper, we will use the expression ``Grassmannian persistence diagrams'' to refer to $1$-parameter Grassmannian persistence diagrams as defined above in~\cref{defn: grassmannian persistence diagrams}. However, in~\cite{gpd-multi} the name ``Grassmannian persistence diagram'' will encompass a broader meaning as a notion of persistence diagrams for filtrations over arbitrary finite posets.

\begin{remark}[Caveats]
    Note that
    \begin{itemize}
        \item[(1)] The metric on the poset $\lp$ in the definition of ($1$-parameter) Grassmannian persistence diagrams will be utilized in order to assign a cost to a morphism between Grassmannian persistence diagrams. We define a morphism between Grassmannian persistence diagrams in~\cref{defn: trans preserving morphism}, and we define the cost of a morphism in~\cref{defn: cost of a trans preserving morphism}.
        \item[(2)] According to \Cref{defn: grassmannian persistence diagrams}, a Grassmannian persistence diagram is a function from $\Seg(\lp)$, the \emph{set} of segments of $\lp$, to $\gr(V)$. Even though a partial order on $\Seg(\lp)$ does not appear in the definition, we will make a slight abuse of notation and ocassionally write $\overline{\lp}^\times$ or $\overline{\lp}^\supseteq$ for the domain of Grassmannian persistence diagrams to indicate that the relevant Grassmannian persistence diagram is obtained from an invariant that is compatible with the specific partial order on $\Seg(\lp)$. For example, in~\cref{subsec: times orthogonal inversion functor}, we will assign a Grassmannian persistence diagram to every object in $\inn(V)$. As the objects in $\inn(V)$ are functions that are monotone with respect to the product order on $\Seg (\lp)$, we will write $\overline{\lp}^\times$ for the domain of Grassmannian persistence diagrams in~\cref{subsec: times orthogonal inversion functor}. Similarly, in~\cref{subsec: orthogonal inversion of Laplacian kernels}, we will use $\overline{\lp}^\times$ and $\overline{\lp}^\supseteq$ for the domain of Grassmannian persistence diagrams that are obtained from birth-death spaces and persistent Laplacians respectively. 
    \end{itemize}    
\end{remark}

\begin{example}[Trivial Grassmanian persistence diagrams]
    Let $\lp = \{ \ell_1 < \cdots < \ell_n \}$ be a linear poset and consider any classical persistence diagram, i.e., any function $m: \Seg (\lp) \to \N$. We will construct a ``trivial'' Grassmannian persistence diagram that extends $m$. Since $\Seg(\lp)$ is a finite set, we can enumerate it, say $\Seg(\lp) = \{ I_1, \ldots, I_k \}$. Let
    \[
    N := \sum_{i=1}^k m(I_i),
    \]
    and let $V = \R^N$ with the standard inner product. Let $\{ e_i \}_{i=1}^N$ denote the canonical basis elements for $\R^N$. We now construct a ``trivial'' Grassmannian persistence diagram $\Mfunc : \Seg(\lp) \to \gr(\R^N)$ as follows:
    \begin{equation*}
        \Mfunc(I_1) := 
            \begin{cases}
            \{ 0 \} & \text{if } m(I_1) =0 \\
            \spn \left\{ e_1,\ldots,e_{m(I_1)} \right\} & \text{if } m(I_1) > 0,
            \end{cases}
    \end{equation*}
    and, for $i > 1$, let $n_i := \sum_{j=1}^i m(I_j)$, and define
    \begin{equation*}
        \Mfunc(I_i) := 
            \begin{cases}
            \{ 0 \} & \text{if } m(I_i) =0 \\
            \spn \left\{ e_{n_{i-1}+1},e_{n_{i-1}+2},\ldots,e_{n_{i}} \right\} & \text{if } m(I_i) > 0.
            \end{cases}
    \end{equation*}
    Then, we have that $\Mfunc : \Seg(\lp) \to \gr(\R^N)$ is a Grassmannian persistence diagram with 
    \[
    \dim (\Mfunc (I_i)) = m(I_i)
    \]
    for all $i=1,\ldots,k$.
\end{example}

\begin{definition}[Transversity-preserving morphism]\label{defn: trans preserving morphism}
    A~\emph{transversity-preserving morphism} from a Grassmannian persistence diagram $\Mfunc : \Seg(\lp_1) \to \gr(V)$ to another Grassmannian persistence diagram $\Nfunc : \Seg(\lp_2) \to \gr(V)$ is a pair $(f, \zeta_{\lp_2})$ where $f = (\ladj{f}, \radj{f})$ is a Galois connection $\ladj{f} : \lp_1 \leftrightarrows \lp_2 : \radj{f}$ and $\zeta_{\lp_2} : \Seg(\lp_2) \to \gr(V)$ is a function supported on $\diag (\lp_2)$ such that
    \begin{enumerate}
        \item $ \{ \zeta_{\lp_2} (J) \}_{J\in \Seg(\lp_2)}$ is transversal to $\{ \Nfunc(J) \}_{J \in \Seg(\lp_2)} $,
        \item $\left(\overline{\ladj{f}}\right)_\sharp \Mfunc \mobeq (\Nfunc + \zeta_{\lp_2})$,
    \end{enumerate}
    where $\overline{\ladj{f}} : \Seg(\lp_1) \to \Seg(\lp_2)$ is the order-preserving map on the poset of segments (with the product order) induced by $\ladj{f}$ and $(\Nfunc + \zeta_{\lp_2}) (J) := \Nfunc(J) + \zeta_{\lp_2}(J)$.

\end{definition}

\begin{figure}[t]
    \centering
    \includegraphics[scale=12]{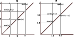}
    \caption{Two Grassmannian persistence diagrams $\Mfunc$ and $\Nfunc$ are shown. The values of $\Mfunc$ and $\Nfunc$ are displayed on top of the corresponding points on the diagrams.}
    \label{fig:transverse space functions}
\end{figure}

\begin{example}[Grassmannian persistence diagram]\label{ex: transverse space functions and  morphism}
    Let $V = \R^4$, $\lp_1 =\{ 1<2<3 \}$, $\lp_2 = \{ 1.5<2.5\}$, $f = (\ladj{f}, \radj{f})$ be as in~\cref{ex: monotone space functions and morphism}. In~\cref{fig:transverse space functions}, we illustrate two Grassmannian persistence diagrams $\Mfunc$ and $\Nfunc$. The diagram on the left of~\cref{fig:transverse space functions} is a visualization of the Grassmannian persistence diagram $\Mfunc :\Seg( \{1<2<3\}) \to \gr(\R^4)$. The subspace on each point on the diagram represents the value of $\Mfunc$ on the corresponding segment. The diagram on the right of~\cref{fig:transverse space functions} is another Grassmannian persistence diagram $\Nfunc : \Seg ( \{ 1.5<2.5 \} ) \to \gr(\R^4)$. Both $\{ \Mfunc(I) \}_{I \in \Seg (\lp_1)}$ and $\{ \Nfunc(J) \}_{J\in \Seg (\lp_2)}$ are transverse families. In this example, it holds that 
    \[
    (\overline{\ladj{f}})_\sharp \Mfunc \mobeq \Nfunc.
    \]
    Observe that $(\overline{\ladj{f}})_\sharp \Mfunc$ and $\Nfunc$ do not agree as functions because $(\overline{\ladj{f}})_\sharp \Mfunc ((2.5, 2.5)) = \spn\{ e_2\}$ whereas $ \Nfunc((2.5, 2.5)) = \spn\{ e_2 - e_3\}$; see~\cref{fig:push m vs n} for an illustration. However, these functions are M\"obius equivalent. Indeed, let $\zeta_{\lp_2} : \Seg(\lp_2) \to \gr(\R^4)$ be the zero map. That is, $\zeta_{\lp_2} (J) = \{ 0 \} \subseteq \R^4$ for all $J \in \Seg(\lp_2)$. Then, we can see that $(f, \zeta_{\lp_2})$ is a transversity-preserving morphism from $\Mfunc$ to $\Nfunc$.
\end{example}

\begin{figure}[t]
    \centering
    \includegraphics[scale=12]{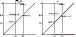}
    \caption{Two functions $\left(\overline{\ladj{f}}\right)_\sharp \Mfunc$ and $\Nfunc$ with domain $\Seg (\{ 1.5 < 2.5 \})$ and codomain $\gr(\R^4)$ are illustrated. These functions are not equal but they are M\"obius equivalent.}
    \label{fig:push m vs n}
\end{figure}

The following proposition shows that the class of Grassmannian persistence diagrams forms a category where the morphisms are given by transversity-preserving morphisms.

\begin{proposition}\label{prop: comp of transversity-pres.}
    The composition of transversity-preserving morphisms is a transversity-preserving morphism.
\end{proposition}
Proof of~\cref{prop: comp of transversity-pres.} can be found in~\cref{appendix: details for orthogonal inversion}.

\begin{notation}
    We denote by $\inndgm (V)$ the category where 
    \begin{itemize}
        \item Objects are Grassmannian persistence diagrams,
        \item Morphisms are transversity-preserving morphisms.
    \end{itemize}
\end{notation}
\nomenclature[28]{$\inndgm(\cdot)$}{Category of (1-parameter) Grassmannian persistence diagrams}

\begin{definition}[Cost of a morphism in $\inndgm(V)$]\label{defn: cost of a trans preserving morphism}
    The \emph{cost} of a morphism $(f, \zeta_{\lp})$ in $\inndgm(V)$, where $f = (\ladj{f},\radj{f})$, is defined to be $\cost_{\inndgm(V)} (f) := \dis(\ladj{f})$, the distortion of the left adjoint $\ladj{f}$. 
\end{definition}

\begin{remark}
    In both categories $\inn(V)$ and $\inndgm(V)$, morphisms are assigned a non-negative cost. Thus, each of these categories is endowed with an edit distance, $d_{\inn(V)}^E$ and $d_{\inndgm(V)}^E$, between their objects, see~\cref{sec: prelim}.
    Note that both in $\inn(V)$ and in $\inndgm(V)$, we have defined the cost of a morphism as the distortion of the left adjoint of the Galois connection that determines the morphism. Although employing right adjoints to define the cost of morphisms in both categories would not affect our stability result (\cref{thm: general stability}), we choose to use the left adjoints to be consistent with~\cite{edit}, where the edit distance between persistence diagrams is introduced. This way, we are able to compare our stability result with the stability result in~\cite[Theorem 8.4]{edit}; see~\cref{thm: classical pd lower bound} and~\cref{appendix: edit}.
\end{remark}

\begin{remark}\label{rem: diagonal blind edit}
    The edit distance in $\inndgm(V)$ is insensitive to the points on the diagonal. That is, if $\Mfunc_1 : \Seg({\lp}) \to \gr (V)$ and $\Mfunc_2 : \Seg({\lp}) \to \gr(V)$ are two Grassmannian persistence diagrams that are the same on $\Seg({\lp}) \setminus \diag({\lp})$, then $d_{\inndgm(V)}^E (\Mfunc_1, \Mfunc_2) = 0$. To see this, let $\Mfunc : \Seg({\lp}) \to \gr(V)$ be defined by 
    \begin{equation*}
        \Mfunc (I) =
        \begin{cases}
            \Mfunc_1 (I) := \Mfunc_2 (I) & \text{if } I\in \Seg({\lp}) \setminus \diag({\lp}) \\
            0            & \text{if } I \in \diag({\lp})
        \end{cases}
    \end{equation*}
    Also, let $\zeta_{\lp}^i : \Seg({\lp}) \to \gr(V)$ be defined by
    \begin{equation*}
        \zeta_{\lp}^i (I) :=
        \begin{cases}
            0 & \text{if } I\in \Seg({\lp}) \setminus \diag({\lp}) \\
            \Mfunc_i (I)            & \text{if } I \in \diag({\lp})
        \end{cases}
    \end{equation*}   
    for $i=1,2$. Observe that the identity pair $\mathrm{id} := (\mathrm{id}_{\lp}, \mathrm{id}_{\lp})$ is a Galois connection from ${\lp}$ to ${\lp}$. Therefore, the pair $(\mathrm{id}, \zeta_{\lp}^i)$ is a transversity-preserving morphism from $\Mfunc_i$ to $\Mfunc$ for $i=1,2$, with cost $\dis(\mathrm{id}_{\lp}) = 0$. Therefore, $d_{\inndgm(V)}^E (\Mfunc_i, \Mfunc) = 0$. Hence, by triangle inequality, $d_{\inndgm(V)}^E (\Mfunc_1, \Mfunc_2) = 0$.
\end{remark}

Note that a Galois connection $(\ladj{f}, \radj{f})$ is a part of the definition of a morphism in both $\inn(V)$ and $\inndgm(V)$. Moreover, the cost of a morphism, in each category, is determined by the distortion of the left adjoint, $\dis(\ladj{f})$. Recall from~\cref{sec: prelim} that the edit distance between two objects is obtained by infimizing the cost of paths between these two objects. Note that under a functor $\mathcal{O}: \inn(V) \to \inndgm(V)$, a path between two objects $\overline{F}$ and $\overline{G}$ in $\inn(V)$ determines a path between $\mathcal{O}(\overline{F})$ and $\mathcal{O}(\overline{\Gfunc})$ in $\inndgm(V)$. If the functor $\mathcal{O}$ maps a morphism in $\inn(V)$ determined by a Galois connection $(\ladj{f}, \radj{f})$ to a morphism in $\inndgm(V)$ determined by the same Galois connection, then, we would obtain stability. Namely, 
\[
d_{\inndgm(V)}^E \left(\mathcal{O}\left(\overline{\Ffunc}\right), \mathcal{O}\left(\overline{\Gfunc}\right)\right) \leq d_{\inn(V)}^E \left(\overline{\Ffunc}, \overline{\Gfunc}\right).
\]

This is because every path between $\overline{F}$ and $\overline{G}$ in $\inn(V)$ induces a path between $\mathcal{O}(\overline{\Ffunc})$ and $\mathcal{O}(\overline{\Gfunc})$ in $\inndgm(V)$ with the same cost. In the next section, we construct a functor, $\times$-Linear Orthogonal Inversion, from $\inn(V)$ to $\inndgm(V)$ that maps a morphism in $\inn(V)$ determined by a Galois connection $(\ladj{f}, \radj{f})$ to a morphism in $\inndgm(V)$ determined by the same Galois connection.

\subsection{The~\texorpdfstring{$\times$}{}-Linear Orthogonal Inversion Functor}\label{subsec: times orthogonal inversion functor}

In this section, we construct a functor that we call the \emph{$\times$-Linear Orthogonal Inversion}
\[
\prhi : \inn(V) \to \inndgm(V),
\]
see \cref{defn: x-harmonic inversion}. This functor, as the most fundamental construction in this paper, imitates the (algebraic) M\"obius inversion and outputs a monoidal M\"obius inverse for every object $\overline{\Ffunc} : \overline{\lp}^\times \to \gr(V)$ in $\inn(V)$ as shown in~\cref{thm: prhi is monoidal inversion}. Our main results in this section are
\begin{itemize}
    \item Functoriality (\cref{prop: prhi is functor}): $\times$-Linear Orthogonal Inversion 
    \[ \prhi : \inn(V) \to \inndgm(V)\] is a functor.
    \item Stability (\cref{thm: general stability}): For any two intersection-monotone space functions $\overline{\Ffunc}$ and $\overline{\Gfunc}$ in $\inn(V)$, we have
    \[
d_{\inndgm(V)}^E \left(\prhi\left(\overline{\Ffunc}\right), \prhi\left(\overline{\Gfunc}\right)\right) \leq d_{\inn(V)}^E \left(\overline{\Ffunc}, \overline{\Gfunc}\right).
\]

\end{itemize}

In~\cref{subsubsec: construction of prhi}, we present the $\times$-Linear Orthogonal Inversion construction,~\cref{defn: x-harmonic inversion}. We then prove its functoriality and stability in~\cref{subsubsec: functor and stable}.

\subsubsection{Construction of~\texorpdfstring{$\prhi$}{}}\label{subsubsec: construction of prhi}
Let $\mathcal{M}$ be a commutative monoid and let $\varphi_\mathcal{M} : \mathcal{M} \to \kappa(\mathcal{M})$ be the canonical map where $\kappa(\mathcal{M})$ is the group completion of $\mathcal{M}$. Let $\lp = \{ \ell_1 <\cdots<\ell_n\}$ and let $m : \overline{\lp}^\times \to \mathcal{M}$ be a function. The algebraic M\"obius inverse of $m$, $\partial_{\overline{\lp}^\times} (m) : \overline{\lp}^\times \to \kappa(\mathcal{M})$, is given by
\begin{equation}\label{eq: rearranged mobius inversion formula}
\footnotesize
    \partial_{\overline{\lp}^\times} (m) ((\ell_i, \ell_j)) = \Big ( \varphi_\mathcal{M} \big ( m ((\ell_i, \ell_j)) \big ) - \varphi_\mathcal{M} \big (m ((\ell_i, \ell_{j-1})) \big ) \Big ) - \Big ( \varphi_\mathcal{M} \big( m ((\ell_{i-1}, \ell_{j})) \big) - \varphi_\mathcal{M} \big ( m ((\ell_{i-1}, \ell_{j-1})) \big) \Big ),
\end{equation}
after rearranging terms appearing in~\cref{eqn: times algebraic mob inv1} in~\cref{prop: algebraic mobius inversion formulas}.
However, as the Grothendieck group completion of $\gr(V)$ is trivial, see~\cref{appendix:details}, the algebraic M\"obius inverse of an object $\overline{\Ffunc} : \overline{\lp}^\times \to \gr(V)$ in $\inn(V)$ is also trivial. This is exactly the motivation for considering the notion of monoidal M\"obius inversion. In order to construct this notion, we interpret the ``minus sign'' in~\cref{eq: rearranged mobius inversion formula} as the difference of subspaces which is described~\cref{defn: difference of spaces}.

\begin{definition}[$\times$-Linear Orthogonal Inversion]\label{defn: x-harmonic inversion}
    Let $\lp = \{ \ell_1<\cdots<\ell_n \}$ be a finite linearly ordered metric poset. For an object $\overline{\Ffunc} : \overline{\lp}^\times \to \gr(V)$ in $\inn(V)$, we define its \emph{$\times$-Linear Orthogonal Inverse}, denoted $\prhi \left(\overline{\Ffunc}\right)$, to be the function $\prhi \left(\overline{\Ffunc}\right) : \overline{\lp}^\times \to \gr(V)$ given by
    \begin{align*}
    \prhi \left(\overline{\Ffunc}\right) ((\ell_i, \ell_j)) &:= \big ( \overline{\Ffunc}((\ell_i, \ell_{j})) \ominus \overline{\Ffunc}((\ell_i, \ell_{j-1})) \big ) \ominus \big ( \overline{\Ffunc}((\ell_{i-1}, \ell_{j})) \ominus \overline{\Ffunc}((\ell_{i-1}, \ell_{j-1}))  \big ), \\
    \prhi \left(\overline{\Ffunc}\right) ((\ell_i, \infty)) &:= \big ( \overline{\Ffunc}((\ell_i, \infty)) \ominus \overline{\Ffunc}((\ell_i, \ell_{n})) \big ) \ominus \big ( \overline{\Ffunc}((\ell_{i-1}, \infty)) \ominus \overline{\Ffunc}((\ell_{i-1}, \ell_{n}))  \big ), \\
    \prhi \left(\overline{\Ffunc}\right) ((\ell_i, \ell_i)) &:= \overline{\Ffunc} ((\ell_i, \ell_i)) \ominus \overline{\Ffunc} ((\ell_{i-1}, \ell_{i})),
    \end{align*}
    for $1\leq i  < j \leq n$. 
\end{definition}
\nomenclature[29]{$\prhi$}{$\times$-Linear Orthogonal Inversion}

Notice that our $\times$-Linear Orthogonal Inversion definition is analogous to the algebraic M\"obius inversion formula in~\cref{prop: algebraic mobius inversion formulas} after rearranging terms appearing in~\cref{eqn: times algebraic mob inv1,eqn: times algebraic mob inv2,eqn: times algebraic mob inv3}. Also, we follow the same convention for the boundary cases as described in~\cref{remark: convention edge cases}. To be precise,
\begin{align*}
        \prhi \left(\overline{\Ffunc}\right)  ((\ell_1, \ell_1)) &:= \overline{\Ffunc}((\ell_1, \ell_{1})), \\
        \prhi \left(\overline{\Ffunc}\right)  ((\ell_1, \infty)) &:= \overline{\Ffunc}((\ell_1, \infty)) \ominus \overline{\Ffunc}((\ell_1 , \ell_n)), \\
        \prhi \left(\overline{\Ffunc}\right)  ((\ell_1, \ell_j)) &:= \overline{\Ffunc}((\ell_1, \ell_{j})) \ominus \overline{\Ffunc}((\ell_1, \ell_{j-1})) \text{ for } 1 < j \leq n.
\end{align*}

\begin{remark}
    Notice that since we define the $\times$-Linear Orthogonal Inversion analogously to \cref{eqn: times algebraic mob inv1,eqn: times algebraic mob inv2,eqn: times algebraic mob inv3}, we are indeed utilizing the product order on $\Seg(\lp)$. We will also introduce a variant of orthogonal inversion, $\supseteq$-Linear Orthogonal Inversion (\cref{defn: sup-harmonic inversion}), in which the reverse inclusion order on $\Seg(\lp) \setminus \diag(\lp)$ is utilized.
\end{remark}

\subsubsection{Functoriality and Stability of~\texorpdfstring{$\prhi$}{}}\label{subsubsec: functor and stable}

We will first show that $\prhi \left( \overline{\Ffunc}\right)$ is an object in $\inndgm(V)$,~\cref{prop: prhi in inndgm}. To do so, we will need the following facts described in~\cref{prop: linAlg Mobius}, \cref{cor: linAlg Mobius dim}, and~\cref{thm: prhi is monoidal inversion}. The proof of~\cref{prop: linAlg Mobius} is given in~\cref{appendix: details for orthogonal inversion}.

\begin{proposition}\label{prop: linAlg Mobius}
Let $A, B, C \subseteq V$ be subspaces of an inner product space $V$ such that $A \supseteq B,C$. Then, 
\[
((A\ominus B) \ominus (C \ominus (B\cap C) )) = A \ominus (B+C).
\]
\end{proposition}

\begin{corollary}\label{cor: linAlg Mobius dim}
    Let $A, B, C \subseteq V$ be subspaces of an inner product space $V$ such that $A \supseteq B,C$. Then, 
    \[
    \dim ( (A\ominus B) \ominus (C \ominus (B\cap C) ) ) = (\dim A - \dim B) - (\dim C - \dim (B\cap C))
    \]
\end{corollary}

\begin{proof}
    \begin{align*}
    \dim ( (A\ominus B) \ominus (C \ominus (B\cap C)) ) &= \dim (A \ominus (B+C)) \\
                                                &= \dim A - \dim (B+C) \\
                                                &= \dim A - (\dim B + \dim C - \dim (B\cap C)) \\
                                                &= (\dim A - \dim B) - (\dim C - \dim (B\cap C)).
\end{align*}
\end{proof}

\begin{remark}\label{rem: importance and generalization}
\cref{prop: linAlg Mobius} and \cref{cor: linAlg Mobius dim} play a fundamental role in this paper in the following ways:
\begin{enumerate}
\item In \cref{prop: loi and goi agrees}, we establish a connection between the notions of $\prhi$ and $\goi$ using \cref{prop: linAlg Mobius}.
\item This connection, combined with the fact that Orthogonal Inversion yields monoidal Möbius inverses (\cref{prop: goi is monoidal mi}), ensures that $\prhi$ also produces monoidal Möbius inverses, as shown in \cref{thm: prhi is monoidal inversion}.
\item Finally, \cref{cor: linAlg Mobius dim} guarantees that $\prhi$ maps objects from $\inn(V)$ to $\inndgm(V)$; see \cref{prop: prhi in inndgm}.
\end{enumerate}
\end{remark}

\begin{example}
    Consider the intersection-monotone space functions $\overline{\Ffunc}$ and $\overline{\Gfunc}$ introduced in~\cref{ex: monotone space functions and morphism} and depicted in~\cref{fig:monotone space functions}. The $\times$-Linear Orthogonal Inverses of $\overline{\Ffunc}$ and $\overline{\Gfunc}$ are the Grassmannian persistence diagrams $\Mfunc$ and $\Nfunc$ introduced in~\cref{ex: transverse space functions and  morphism} and depicted in~\cref{fig:transverse space functions}. That is, $\prhi \left(\overline{\Ffunc}\right) = \Mfunc$ and $\prhi \left(\overline{\Gfunc}\right) = \Nfunc$. Although this can be verified through the definition of $\times$-Linear Orthogonal Inversion (cf. \cref{defn: x-harmonic inversion}),~\cref{prop: linAlg Mobius} provides a more compact and easier way to do so. While~\cref{defn: x-harmonic inversion} requires employing the operation $\ominus$ three times, as a result of~\cref{prop: linAlg Mobius}, the $\times$-Linear Orthogonal Inversion can be computed by involving $\ominus$ only once. 
\end{example}

We now show that $\prhi$ aligns with the more general notion $\goi$ (\cref{defn: goi}), and as a result of this, we conclude that $\prhi$ produces monoidal M\"obius inverses.

\begin{proposition}[Equivalence of $\prhi$ and $\goi$]\label{prop: loi and goi agrees}
    Let $\lp = \{ \ell_1<\cdots<\ell_n \}$ be a finite linearly ordered metric poset. For an object $\overline{\Ffunc} : \overline{\lp}^\times \to \gr(V)$ in $\inn(V)$, we have that 
    \[
    \prhi \left( \overline{\Ffunc} \right) = \goi \left( \overline{\Ffunc} \right).
    \]
\end{proposition}

\begin{proof}
Let $(\ell_i, \ell_j) \in \Seg(\lp)$ be a segment and assume that $\ell_i < \ell_j$ and $\ell_j \neq \infty$. By~\cref{prop: linAlg Mobius}, the $\times$-Linear Orthogoal Inverse of an intersection-monotone space function $\overline{\Ffunc}$ can be written as follows.
    \begin{align}
        \prhi \left(\overline{\Ffunc}\right)((\ell_i, \ell_j)) &= \overline{\Ffunc}((\ell_i,\ell_j)) \ominus \left( \overline{\Ffunc}((\ell_{i-1}, \ell_j)) + \overline{\Ffunc}((\ell_i, \ell_{j-1})) \right) \\
        &= \overline{\Ffunc}((\ell_i,\ell_j)) \ominus \left( \sum_{I< (\ell_i, \ell_j)} \overline{\Ffunc}(I) \right), \label{eqn: road to generalization}
    \end{align}
    where the last equality follows from the fact that for any $I < (\ell_i, \ell_j)$, we have that either $I\leq (\ell_{i-1}, \ell_j)$ or $I \leq (\ell_i, \ell_{j-1})$, and thus, $$\overline{\Ffunc}(I) + \overline{\Ffunc}((\ell_{i-1}, \ell_j)) + \overline{\Ffunc}((\ell_i, \ell_{j-1})) = \overline{\Ffunc}((\ell_{i-1}, \ell_j)) + \overline{\Ffunc}((\ell_i, \ell_{j-1})).$$   
    Therefore, 
    \[
    \prhi \left(\overline{\Ffunc}\right)((\ell_i, \ell_j)) = \overline{\Ffunc}((\ell_i,\ell_j)) \ominus \left( \sum_{I< (\ell_i, \ell_j)} \overline{\Ffunc}(I) \right) = \goi \left(\overline{\Ffunc} \right) ((\ell_i,\ell_j)).
    \]
    While the argument above is only presented with the assumption that $\ell_i < \ell_i \neq \infty$, similar arguments work when $\ell_i = \ell_j$ and $\ell_i < \ell_j =\infty$.
\end{proof}

Recall that our notion of $\times$-Linear Orthogonal Inversion, $\prhi$, is motivated by the algebraic M\"obius inversion formula on the poset of segments of a linear poset as shown in~\cref{prop: algebraic mobius inversion formulas}. We now present our result, \cref{thm: prhi is monoidal inversion}, that relates these two notions. This result also serves as the primary tool utilized in proving the functoriality of $\times$-Linear Orthogonal Inversion,~\cref{prop: prhi is functor}.

\begin{remark}
    For the following results~\cref{thm: prhi is monoidal inversion},~\cref{prop: prhi in inndgm}, and~\cref{prop: prhi is functor}, the metric on the poset $\lp = \{ \ell_1 < \cdots < \ell_n \}$ is indeed irrelevant. We only require a metric structure on $P$ to ensure that the relevant functions associated with these statements fall within the appropriate categories $\inn(V)$ and $\inndgm(V)$. We eventually exploit the metric structure in proving the stability result; see~\cref{thm: general stability}.
\end{remark}

\begin{theorem}\label{thm: prhi is monoidal inversion}
    Let $\lp = \{ \ell_1<\cdots<\ell_n \}$. For an object $\overline{\Ffunc} : \overline{\lp}^\times \to \gr(V)$ in $\inn(V)$, its $\times$-Linear Orthogonal Inverse $\prhi \left(\overline{\Ffunc}\right)$ is a monoidal M\"obius inverse of $\overline{\Ffunc}$, i.e., $\prhi \left(\overline{\Ffunc}\right) \in \mmi{\overline{\lp}^\times}{\overline{\Ffunc}}$. 
\end{theorem}

\begin{proof}
    This result follows from the fact that $\goi$ and $\prhi$ agree (\cref{prop: loi and goi agrees}) and that $\goi$ produces monoidal M\"obius inverses (\cref{prop: goi is monoidal mi}). 
\end{proof}

As noted in \cref{rem: importance and generalization}, we now use \cref{cor: linAlg Mobius dim} to conclude that $\prhi\left( \overline{\Ffunc} \right)$ is an object in $\inndgm(V)$.

\begin{proposition}\label{prop: prhi in inndgm}
        Let $\lp = \{ \ell_1<\cdots<\ell_n \}$. For an object $\overline{\Ffunc} : \overline{\lp}^\times \to \gr(V)$ in $\inn(V)$, its $\times$-Linear Orthogonal Inverse $\prhi \left(\overline{\Ffunc}\right) : \overline{\lp}^\times \to \gr(V)$ is an object in $\inndgm(V)$. 
\end{proposition}

\begin{proof}
    We need to check that $\{ \prhi \left(\overline{\Ffunc}\right) (I) \}_{I\in \overline{\lp}^\times}$ is a transversal family. By~\cref{cor: linAlg Mobius dim}, for every $(\ell_i, \ell_j) \in \overline{\lp}^\times$ with $i<j$, we have that
    \[
    \dim \left(\prhi\left( \overline{\Ffunc}\right)((\ell_i, \ell_j))\right) = \dim \overline{\Ffunc} ((\ell_i,\ell_j)) - \dim \overline{\Ffunc} ((\ell_i,\ell_{j-1}))+\dim \overline{\Ffunc} ((\ell_{i-1},\ell_{j-1})) - \dim \overline{\Ffunc} ((\ell_{i-1},\ell_j)).
    \]
    This means that the function $\dim \left(\prhi \left(\overline{\Ffunc}\right)\right) : \overline{\lp}^\times \to \Z$ given by $(\ell_i, \ell_j)\mapsto \dim \left(\prhi \left(\overline{\Ffunc}\right)((\ell_i,\ell_j))\right)$ is the algebraic M\"obius inverse of the function $\dim (\overline{\Ffunc}) : \overline{\lp}^\times \to \Z$ given by $(\ell_i,\ell_j) \mapsto \dim (\overline{\Ffunc}((\ell_i,\ell_j)))$. Thus, as the segment $(\ell_n,\infty)$ is the maximum element of $\overline{\lp}^\times$, we have that
    \[
    \dim \left(\overline{\Ffunc} ((\ell_n,\infty)\right) = \sum_{I \in \overline{\lp}^\times} \dim\left(\prhi \left(\overline{\Ffunc}\right)(I)\right).
    \]
     On the other hand, by~\cref{thm: prhi is monoidal inversion}, we have that
    \[
    \sum_{I\in \overline{\lp}^\times} \prhi \left(\overline{\Ffunc}\right)(I) = \overline{\Ffunc}((\ell_n,\infty))
    \]
    Thus, we have that
    \[
    \dim \left ( \sum_{I \in \overline{\lp}^\times} \prhi \left(\overline{\Ffunc}\right) (I) \right) = \dim \left(\overline{\Ffunc} ((\ell_n,\infty))\right) = \sum_{I \in \overline{\lp}^\times} \dim \left(\prhi \left(\overline{\Ffunc}\right)(I)\right)
    \]
    Therefore, $\{ \prhi \left(\overline{\Ffunc}\right)(I) \}_{I\in \overline{\lp}^\times}$ is a transversal family. Hence $\prhi \left(\overline{\Ffunc}\right) : \overline{\lp}^\times \to \gr(V)$ is an object in $\inndgm(V)$. 
\end{proof}

Note that, in~\cref{prop: prhi in inndgm}, we have only shown that the family $\{ \prhi \left(\overline{\Ffunc}\right)(I) \}_{I\in \overline{\lp}^\times}$ is a transversal family as this is enough to conclude that $\prhi \left(\overline{\Ffunc}\right)$ is an object in $\inndgm(V)$. However, a finer property is true:  certain subspaces in the family $\{ \prhi \left(\overline{\Ffunc}\right)(I) \}_{I\in \overline{\lp}^\times}$ are orthogonal to each other, as we show in the following.

\begin{proposition}\label{prop: comparable orhtogonal}
    Let $\lp = \{ \ell_1 < \cdots < \ell_n\}$ and let $\overline{\Ffunc} : \overline{\lp}^\times \to \gr(V)$ be an object in $\inn(V)$. Let $(\ell_i,\ell_j) <_\times (\ell_k,\ell_l) \in \overline{\lp}^\times$ be two distinct comparable segments, i.e., $(\ell_i,\ell_j) \neq (\ell_k,\ell_l)$ and $(\ell_i,\ell_j) \prodord (\ell_k,\ell_l)$. Then, $\prhi \left(\overline{\Ffunc}\right) ((\ell_i, \ell_j))$ and $\prhi \left(\overline{\Ffunc}\right)( (\ell_k, \ell_l))$ are orthogonal to each other.
\end{proposition}

\begin{proof}
    By~\cref{prop: linAlg Mobius}, we have that 
    \[
    \prhi \left(\overline{\Ffunc}\right) ((\ell_k, \ell_l)) = \overline{\Ffunc}((\ell_k, \ell_l)) \ominus \left ( \overline{\Ffunc}((\ell_{k-1}, \ell_l)) + \overline{\Ffunc}((\ell_k, \ell_{l-1})) \right ).
    \]
    Thus, $\prhi \left(\overline{\Ffunc}\right) ((\ell_k, \ell_l))$ is orthogonal to $\left ( \overline{\Ffunc}((\ell_{k-1}, \ell_l)) + \overline{\Ffunc}((\ell_k, \ell_{l-1})) \right )$. On the other hand, since $(\ell_i, \ell_j) \prodord (\ell_k, \ell_l)$, we have that $$\prhi \left(\overline{\Ffunc}\right) ((\ell_i, \ell_j)) \subseteq \overline{\Ffunc} ((\ell_i, \ell_j)) \subseteq \left ( \overline{\Ffunc}((\ell_{k-1}, \ell_l)) + \overline{\Ffunc}((\ell_k, \ell_{l-1})) \right ).$$ Hence, $\prhi \left(\overline{\Ffunc}\right) ((\ell_i, \ell_j))$ and $\prhi \left(\overline{\Ffunc} \right)((\ell_k, \ell_l))$ are orthogonal to each other.
\end{proof}

By~\cref{prop: prhi in inndgm}, we have that $\prhi$ maps an object in $\inn(V)$ to an object in $\inndgm(V)$. We now verify that this assignment is indeed a functor.

\begin{proposition}[Functoriality of $\prhi$]\label{prop: prhi is functor}
    $\prhi$ is a functor from $\inn(V)$ to $\inndgm(V)$.
\end{proposition}

\begin{proof}
    Let $ \overline{\Ffunc} : \overline{\lp_1}^\times \to \gr(V)$ be an object in $\inn(V)$. By~\cref{prop: prhi in inndgm}, we have that $\prhi \left(\overline{\Ffunc}\right)$ is an object in $\inndgm(V)$. Now, let $\overline{\Gfunc} : \overline{\lp_2}^\times \to \gr(V)$ be another object in $\inn(V)$ and let $(\ladj{f}, \radj{f})$ be a morphism from $\overline{\Ffunc}$ to $\overline{\Gfunc}$. This means that $ \left(\radj{\overline{f}}\right)^\sharp \overline{\Ffunc}=  \overline{\Ffunc} \circ \radj{\overline{f}} = \overline{\Gfunc}$. By~\cref{thm: prhi is monoidal inversion}, we have that $\prhi \left(\overline{\Ffunc}\right)$ and $\prhi \left(\overline{\Gfunc}\right)$ are monoidal M\"obius inverses of $\overline{\Ffunc}$ and $\overline{\Gfunc}$ respectively. Then, by the monoidal RGCT,~\cref{thm: monoidal rgct}, we have that 
    \[
    \left(\ladj{\overline{f}}\right)_\sharp \prhi \left(\overline{\Ffunc}\right) \mobeq \prhi \left(\overline{\Gfunc}\right).
    \]
    Therefore, the pair $(f, \zeta_{\lp_2} := 0)$, where $f := (\ladj{f}, \radj{f})$ is the Galois connection, is a morphism from $\prhi \left(\overline{\Ffunc}\right)$ to $\prhi \left(\overline{\Gfunc}\right)$.
\end{proof}

We now show that the functor $\times$-Linear Orthogonal Inversion is $1$-Lipschitz with respect to the edit distances in $\inn(V)$ and $\inndgm(V)$.

\begin{theorem}[Stability of $\prhi$]\label{thm: general stability}
    Let $\overline{\Ffunc}$ and $\overline{\Gfunc}$ be two intersection-monotone space functions. Then, 
    $$d_{\inndgm(V)}^E \left(\prhi \left(\overline{\Ffunc}\right), \prhi \left(\overline{\Gfunc}\right) \right) \leq d_{\inn(V)}^E \left(\overline{\Ffunc}, \overline{\Gfunc}\right).$$
\end{theorem}
    
\begin{proof}
Recall that the edit distance, $d^E_{-}$, between two objects in a category is defined as the infimum of the cost of paths between the two objects in the category (see~\cref{defn: edit dist}). Any path $\mathcal{P}$ between $\overline{\Ffunc}$ and $\overline{\Gfunc}$ in $\inn(V)$ induces a path, $\prhi \left(\mathcal{P}\right)$, between $\prhi \left(\overline{\Ffunc}\right)$ and $\prhi \left(\overline{\Gfunc}\right)$ in $\inndgm (V)$ by~\cref{prop: prhi is functor}. Moreover, the cost of morphisms is preserved under $\prhi$. This is because if $(\ladj{f}, \radj{f})$ is a morphism from $\overline{\Hfunc}_1$ to $\overline{\Hfunc}_2$ in $\inn(V)$, then the same Galois connection $(\ladj{f}, \radj{f})$ determines a morphism from $\prhi \left(\overline{\Hfunc}_1\right)$ to $\prhi \left(\overline{\Hfunc}_2\right)$ as described in the proof of~\cref{prop: prhi is functor}. Therefore, the cost of $\mathcal{P}$ and the cost of the induced path $\prhi \left(\mathcal{P}\right)$ are the same. Thus, we conclude $d_{\inndgm(V)}^E \left(\prhi \left(\overline{\Ffunc}\right), \prhi \left(\overline{\Gfunc}\right)\right) \leq d_{\inn(V)}^E (\overline{\Ffunc}, \overline{\Gfunc})$.
\end{proof}

    \begin{remark}[Difference between $1$-parameter and general Grassmannian persistence diagrams]
    Note that there is an inherent difference between 1-parameter Grassmannian persistence diagrams, as  introduced in~\cref{sec: orthogonal inversion}, and downward transverse functions. In~\cref{sec: orthogonal inversion}, $1$-parameter Grassmannian persistence diagrams are functions defined on $\Seg(\lp)$, the poset of segments of a finite linear poset $\lp$. On the other hand, downward transverse functions are defined on any finite poset $R$. The poset of segments $\Seg(\lp)$ has a distinguished subset, the diagonal $\diag(\lp)$. Due to the way we define morphisms in the category of $1$-parameter Grassmannian persistence diagrams, $\inndgm(V)$, the edit distance between Grassmannian persistence diagrams is insensitive to the values of the Grassmannian persistence diagram on the diagonal $\diag(\lp)\subseteq \Seg(\lp)$, as explained in~\cref{rem: diagonal blind edit}. 
    
    On the other hand, in the category of downward transversal functions, the domain can be any finite poset $R$ without a distinguished subset such as a diagonal. As a result, the edit distance between two downward transversal function is sensitive to the values of these functions everywhere, even if the two downward transversal functions are defined on the set segments of some finite posets. It is important to note that both sensitivity or insensitivity of the distance could be desirable. For example, the bi-Lipschitz equivalence between the bottleneck distance and the edit distance between classical persistence diagrams~\cite[Proposition 9.1]{edit} requires the edit distance to be insensitive to the diagonal. On the other hand, diagonal-sensitive distances can be more discriminative as shown in~\cite[Examples 28 and 29]{ephemral}.
\end{remark}

\section{Grassmannian Persistence Diagrams of~\texorpdfstring{$1$}{}-Parameter Filtrations}\label{sec: applications}
In this section, we continue to build upon and expand the tools introduced in~\cref{sec: orthogonal inversion}. Here, we introduce the notion of \emph{degree-$\dgr$ Grassmannian persistence diagram of a $1$-parameter filtration}, an extension of the classical notion of persistence diagram \cite{cohen-steiner2007}, as $\times$-Linear Orthogonal Inverse of $\dgr$-th birth-death spaces in~\cref{defn: 1p gpd from zb}. 

In~\cref{subsec: orthogonal inversion of bd}, we provide a functorial way of obtaining degree-$\dgr$ Grassmannian persistence diagram by utilizing the results described in~\cref{sec: orthogonal inversion}. As a result of this functoriality, we establish the edit distance stability of such degree-$\dgr$ Grassmannian persistence diagrams in~\cref{thm: stability of prhi of zb}. 

In~\cref{subsec: inter and canon}, we explore the interpretation and canonicality of $1$-parameter Grassmannian persistence diagrams. For a $1$-parameter filtration $\Ffunc : \{ \ell_1 <\cdots < \ell_n \} \to \subcx(K)$, as shown in~\cref{prop: born and dies exactly}, every segment $(\ell_i, \ell_j)$ is assigned a subspace of the cycle space of $K$ consisting of cycles that are born at $\ell_i$ and die at $\ell_j$. Combining this fact with the later-explored relation of Grassmannian persistence diagrams and the classical persistence diagrams in~\cref{subsec: relation to classical pd}, we conclude that every segment $(\ell_i, \ell_j)$ is assigned a subspaces whose dimension is precisely the multiplicity of the segment $(\ell_i, \ell_j)$ in the classical persistence diagram of $\Ffunc$. Consequently, for segments with multiplicity one, the Grassmannian persistence diagram determines a canonical cycle representative (up to scalar multiplication) for that segment.

In~\cref{subsec: computation of loi zb}, we present~\cref{algo: gpd}, an algorithm for computing the 1-parameter Grassmannian persistence diagram of a given filtration and analyze its time complexity. The complexity result is formally stated in~\cref{prop: algo complexity} and proven in~\cref{subsec: computational complexity}.

\begin{remark}
    Let $K$ be a finite simplicial complex and let $\Ffunc = \{ K_i \}_{i=1}^n$ be a filtration of $K$. When defining the degree-$\dgr$ Grassmannian persistence diagram of $\Ffunc$, we utilize the standard inner product on $C_\dgr^K$, as described in~\cref{sec: prelim}. If one decides to choose another inner product on $C_\dgr^K$, the resulting Grassmannian persistence diagrams will be different. However, it is important to note that our results in this section  remain valid, regardless of the choice of the inner product. 
\end{remark}

\subsection{The~\texorpdfstring{$\times$}{}-Linear Orthogonal Inverse of Birth-Death Spaces}\label{subsec: orthogonal inversion of bd}
Let $K$ be a finite simplicial complex and recall that $\subcx (K)$ denotes the poset of subcomplexes of $K$ ordered by inclusion. 

\begin{definition}[Category of $1$-parameter filtrations]

We define $\Fil(K)$ to be the category where
\begin{itemize}
    \item Objects are $1$-parameter filtrations $\Ffunc : \lp  \to \subcx(K)$ where $\lp$ is a finite linearly ordered metric poset,
    \item Morphisms from $\Ffunc : \lp_1\to \subcx(K)$ to $\Gfunc : \lp_2 \to \subcx(K)$ are given by a Galois connections $\ladj{f} : \lp_1 \leftrightarrows \lp_2 : \radj{f}$,~\cref{defn: galois connections}, such that $\Ffunc \circ \radj{f} = \Gfunc$. That is, the solid arrows in the following diagram commute.
\end{itemize}
\begin{center}
        \begin{tikzcd}
P \arrow[rr, "f_\diamond", dashed, bend right] \arrow[rd, "\Ffunc"'] &          & Q \arrow[ld, "\Gfunc"] \arrow[ll, "f^\diamond"', bend right] \\
                                                                    & \subcx (K) &                                                            
\end{tikzcd}
    \end{center}   
    
\end{definition}
\nomenclature[30]{$\Fil(\cdot)$}{Category of $1$-parameter filtrations of a fixed simplicial complex}

\begin{definition}[Cost of a Morphism in $\Fil(K)$]
    The cost of a morphism $(\ladj{f}, \radj{f})$ in $\Fil(K)$ is given by $\dis(\ladj{f})$, the distortion of the left adjoint $\ladj{f}$.
\end{definition}

Recall from~\cref{defn: bd space} that, given a filtration $\Ffunc$, the birth-death spaces associated to $\Ffunc$ produces a map $\ZB_\dgr^\Ffunc : \overline{P}^\times \to \gr(C_\dgr^K)$. This assignment is actually a functor from $\Fil(K)$ to $\inn(C_\dgr^K)$. 

\begin{proposition}\label{prop: zb functor}
    For any degree $\dgr\geq 0$ and for any filtration $\Ffunc$ in $\Fil (K)$, $\ZB_\dgr^\Ffunc$ is an object in $\inn\left(C_\dgr^K\right)$. Moreover, the assignment
    \[
    \Ffunc \mapsto \ZB_\dgr^\Ffunc
    \]
    is a functor from $\Fil(K)$ to $\inn \left ( C_\dgr^K \right )$.
\end{proposition}

\begin{proof}
    Let $\lp_1 = \{ \ell_1 < \cdots < \ell_n \}$ and let $\Ffunc : \lp_1 \to \subcx (K)$ be a filtration. For two segments $(\ell_i,\ell_j) \prodord (\ell_k,\ell_l) \in \overline{\lp_1}^\times$ we have that $\Zfunc_\dgr (K_i) \subseteq \Zfunc_\dgr(K_k)$ and $\Bfunc_\dgr(K_j)\subseteq \Bfunc_\dgr(K_l)$. Therefore, $\ZB_\dgr^\Ffunc((\ell_i,\ell_j)) \subseteq \ZB_\dgr^\Ffunc((\ell_k,\ell_l))$. Hence, $\ZB_\dgr^\Ffunc$ is order-preserving. It is straightforward to check that $\ZB_\dgr^\Ffunc ((\ell_{i+1}, \ell_j)) \cap \ZB_\dgr^\Ffunc  ((\ell_i, \ell_{j+1})) = \ZB_\dgr^\Ffunc  ((\ell_{i}, \ell_{j}))$ for every $1\leq i < j \leq n$. Thus, $\ZB_\dgr^\Ffunc$ is an object in $\inn\left(C_\dgr^K\right)$. 
    
    Now, let $\lp_2 = \{r_1 < \cdots < r_m \}$ and $\Gfunc : \lp_2 \to \Fil(K)$ be another filtration. Assume that $\ladj{f} : \lp_1 \leftrightarrows \lp_2 : \radj{f}$ is a morphism from $\Ffunc : \lp_1 \to \Fil(K)$ to $\Gfunc : \lp_2 \to \Fil(K)$. Then, for any $(r_i, r_j) \in \overline{\lp_2}^\times$, $\ZB_\dgr^\Gfunc((r_i, r_j)) = \Zfunc_\dgr(\Gfunc(r_i)) \; \cap \;\Bfunc_\dgr(\Gfunc(r_j)) = \Zfunc_\dgr(\Ffunc \circ \radj{f}(r_i)) \; \cap \; \Bfunc_\dgr(\Ffunc \circ \radj{f}(r_j)) = \ZB_\dgr^\Ffunc \left(\overline{\radj{f}}\right) ([r_i, r_j])$. Therefore, $\ZB_\dgr^\Gfunc = \ZB_\dgr^\Ffunc \circ \left(\overline{\radj{f}}\right)$. Hence, $(\ladj{f}, \radj{f})$ is a morphism from $\ZB_\dgr^\Ffunc$ to $\ZB_\dgr^\Gfunc$.
\end{proof}

We now apply $\times$-Linear Orthogonal Inversion to the intersection-monotone space function $\ZB_\dgr^\Ffunc : \overline{\lp}^\times \to \gr(C_\dgr^K)$ and obtain the Grassmannian persistence diagram $$\prhi \left(\ZB_\dgr^\Ffunc\right) : \overline{\lp}^\times \to \gr\left( C_\dgr^K \right).$$

\begin{definition}[Degree-$\dgr$ Grassmannian persistence diagram]\label{defn: 1p gpd from zb}
    Let $\Ffunc : \lp \to \subcx (K)$ be a filtration. For any $\dgr\geq 0$, the map $$\prhi \left(\ZB_\dgr^\Ffunc\right) : \overline{\lp}^\times \to \gr \left (C_\dgr^K \right)$$ is called \emph{the degree-$\dgr$ Grassmannian persistence diagram} of $\Ffunc$ (obtained from  birth-death spaces).
\end{definition}

Observe that the degree-$\dgr$ Grassmannian persistence diagram of a filtration $\Ffunc$ obtained from birth-death spaces is indeed a Grassmannian persistence diagram in the sense of~\cref{defn: grassmannian persistence diagrams}. This can be seen from the fact that $\ZB_\dgr^\Ffunc$ is an object in $\inn\left(C_\dgr^K\right)$ (\cref{prop: zb functor}) and $\prhi$ is a functor from $\inn\left(C_\dgr^K\right)$ to the category of Grassmannian persistence diagrams $\inndgm\left(C_\dgr^K\right)$ (\cref{prop: prhi is functor}).

By~\cref{thm: general stability} and~\cref{prop: zb functor}, we immediately conclude that degree-$\dgr$ Grassmannian persistence diagrams are edit distance stable. 

\begin{theorem}[Stability]{\label{thm: stability of prhi of zb}}
    Let $\Ffunc$ and $\Gfunc$ be two filtrations of a fixed finite simplicial complex $K$. Then, for any degree $\dgr \geq  0$, we have
    \[
    d_{\inndgm\left(C_\dgr^K\right)}^E \left(\prhi \left(\ZB_\dgr^\Ffunc\right), \prhi \left(\ZB_\dgr^\Gfunc\right) \right) \leq d_{\Fil(K)}^E (\Ffunc, \Gfunc).
    \]
\end{theorem}

\begin{proof}
    Let $\Ffunc : \lp_1 \to \subcx (K)$ and $\Gfunc : \lp_2 \to \subcx (K)$ be two filtrations. By~\cref{prop: zb functor}, any path between the filtrations $\Ffunc$ and $\Gfunc$ in the category $\Fil (K)$ induces a path between the intersection-monotone space functions $\ZB_\dgr^\Ffunc$ and $\ZB_\dgr^\Gfunc$ in the category $\inn\left(C_\dgr^K\right)$ with the same cost. Thus, 
    \[
    d_{\inn\left(C_\dgr^K\right)}^E \left(\ZB_\dgr^\Ffunc, \ZB_\dgr^\Gfunc\right) \leq d_{\Fil(K)}^E (\Ffunc, \Gfunc). 
    \]
    By~\cref{thm: general stability}, we have that
    \[
    d_{\inndgm\left(C_\dgr^K\right)}^E \big(\prhi \left(\ZB_\dgr^\Ffunc\right), \prhi \left(\ZB_\dgr^\Gfunc\right)\big) \leq d_{\inn\left(C_\dgr^K\right)}^E \left(\ZB_\dgr^\Ffunc, \ZB_\dgr^\Gfunc\right).
    \]
    Thus, we obtain the desired inequality:
    \[
    d_{\inndgm\left(C_\dgr^K\right)}^E \left(\prhi \left(\ZB_\dgr^\Ffunc\right), \prhi \left(\ZB_\dgr^\Gfunc\right) \right) \leq d_{\Fil(K)}^E (\Ffunc, \Gfunc).
    \]
\end{proof}

\subsection{Interpretation and Canonicality}\label{subsec: inter and canon}
Notice that $\prhi \left(\ZB_\dgr^\Ffunc\right)$ assigns a vector subspace of $C_\dgr^K$ to every segment in $\Seg(\lp)$. As will be proven in~\cref{prop: dim of loi is classical pd}, the dimension of the vector space $\prhi \left(\ZB_\dgr^\Ffunc \right)((\ell_i, \ell_j))$ is exactly the multiplicity of the segment $(\ell_i,\ell_j)$ in the classical degree-$\dgr$ persistence diagram of the filtration $\Ffunc: \lp \to \subcx(K)$. Since the multiplicity of the segment $(\ell_i, \ell_j)$ in the persistence diagram counts the number of topological features that are born at $\ell_i$ and die at $\ell_j$, we expect that every cycle in $\prhi \left(\ZB_\dgr^\Ffunc \right)((\ell_i, \ell_j))$ is born exactly at $\ell_i$ and dies exactly at $\ell_j$ in the sense of~\cref{def:bd-cycles}. Indeed, this is the case as we show now.

\begin{theorem}\label{prop: born and dies exactly}
    Let $\Ffunc : \lp =  \{\ell_1 < \cdots <\ell_n\} \to \Fil(K)$ be a $1$-parameter filtration and let $z\in \prhi \left(\ZB_\dgr^\Ffunc\right) ((\ell_i, \ell_j))$ be a nonzero cycle. Then, $z$ is born precisely at $\ell_i$ and dies precisely at $\ell_j$.
\end{theorem}

\begin{proof}
Let $z \in \prhi \left( \ZB_\dgr^\Ffunc \right)((\ell_i, \ell_j))$ be a nonzero cycle. As noted in~\cref{rem: importance and generalization}, by~\cref{prop: linAlg Mobius}, we have
\begin{align*}
    \prhi \left( \ZB_\dgr^\Ffunc \right)((\ell_i, \ell_j)) &= \ZB_\dgr^\Ffunc ((\ell_i, \ell_j)) \ominus \left( \ZB_\dgr^\Ffunc ((\ell_{i-1}, \ell_j)) + \ZB_\dgr^\Ffunc ((\ell_i, \ell_{j-1}))\right) \\
    &= \ZB_\dgr^\Ffunc ((\ell_i, \ell_j)) \ominus \left(\sum_{(\ell_{i'}, \ell_{j'}) <_\times (\ell_i, \ell_j)} \ZB_\dgr^\Ffunc ((\ell_{i'}, \ell_{j'})) \right)\\
    &= \ZB_\dgr^\Ffunc ((\ell_i, \ell_j)) \cap \left( \sum_{(\ell_{i'}, \ell_{j'}) <_\times (\ell_i, \ell_j)} \ZB_\dgr^\Ffunc ((\ell_{i'}, \ell_{j'})) \right)^\perp.
\end{align*}
As $z$ is nonzero, we conclude that $z \notin \sum_{(\ell_{i'}, \ell_{j'}) <_\times (\ell_i, \ell_j)} \ZB_\dgr^\Ffunc ((\ell_{i'}, \ell_{j'}))$ and $z\in \ZB_\dgr^\Ffunc ((\ell_i, \ell_j))$. Therefore, $z$ is born at $\ell_i$ and dies at $\ell_j$.
\end{proof}

\begin{remark}
    Note that the result above could also be deduced from the equivalence between $\prhi$ and $\goi$, along with the fact that $\goi$ produces cycle spaces consisting of cycles that are born precisely at $\ell_i$ and die precisely at $\ell_j$, as established in \cite[Theorem 5]{gpd-multi}.
\end{remark}

\begin{remark}[Canonicality] 
A key property of the degree-$\dgr$ Grassmannian persistence diagram is that it assigns a vector subspace of the chain space $C_\dgr^K$ to each segment in $\lp$. Furthermore,  this assignment provides a consistent choice of cycles in the sense of \cref{prop: born and dies exactly}.  Moreover, it is canonical in the sense that it remains independent of superfluous choices, such as relabeling (i.e., permuting) the vertices of $K$, as established in \cite[Proposition 4.4]{gpd-multi}.
\end{remark}

\begin{figure}
    \centering
    \includegraphics[scale=14.5]{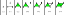}
    \caption{Filtration $\Ffunc : \{ 0 < 1 < 2 < 3 < 4 < 5 < 6 \} \to \subcx(K)$.}
    \label{fig:filtration}
\end{figure}

\begin{example}\label{ex: table 1p gpd}
    Let $\Ffunc$ be the filtration depicted in~\cref{fig:filtration}. For $\dgr=0,1$, we compute the degree-$\dgr$ Grassmannian persistence diagram of $\Ffunc$, i.e. the function $\prhi \left(\ZB_\dgr^\Ffunc\right)$. For the segments in the support of $\prhi \left(\ZB_\dgr^\Ffunc\right)$, we list the corresponding nonzero vector space below. 
    \begin{center}
    \begin{tabular}{|c|c|}
    \hline
    \multicolumn{2}{|c|}{$\dgr=0$} \\
    \hline
         $\prhi \left(\ZB_0^\Ffunc\right) ((1,1))$ & $\spn\{ b-a \}$ \\
         \hline
         $\prhi \left(\ZB_0^\Ffunc\right) ((1,2))$ & $\spn\{ 2c - (a+b) \}$  \\
         \hline
         $\prhi \left(\ZB_0^\Ffunc\right) ((3,4))$ & $\spn\{ 3d - (a+b+c) \}$  \\
         \hline
         $\prhi \left(\ZB_0^\Ffunc\right) ((0,\infty))$ & $\spn\{ a \}$  \\
         \hline
    \end{tabular}

    \begin{tabular}{|c|c|}
    \hline
    \multicolumn{2}{|c|}{$\dgr=1$} \\
    \hline
         $\prhi \left(\ZB_1^\Ffunc\right) ((2,2))$ & $\spn\{ ab - ac + bc \}$ \\ 
         \hline
         $\prhi \left(\ZB_1^\Ffunc\right) ((5,6))$ & $\spn\{ 3cd - 3bd + 2bc - ab + ac \}$ \\
         \hline
    \end{tabular}
    \end{center}
    Notice that the generators for an segment $(i,j)$ in the table above correspond to a cycle that is born precisely at $i$ and dies precisely at $j$, as claimed in~\cref{prop: born and dies exactly}.
\end{example}

\begin{remark}

    Note that, in their paper \cite{harmonicph} on \emph{Harmonic Persistent Homology}, Basu and Cox provide an alternative way to obtain a vector subspace of $C_\dgr^K$ for every segment in the persistence diagram of a filtration. Our Grassmannian persistence diagram has two main advantages over their construction:
    \begin{enumerate}
    \item[(1)] First, for every segment $(b,d)$ in the persistence diagram of a filtration, every nonzero cycle in the subspace assigned to this segment through the Grassmannian persistence diagram is guaranteed to be born at time $b$ and become a boundary at time $d$; see~\cref{prop: born and dies exactly}. This is not the case with the construction of Basu and Cox; see \cite[page 193]{harmonicph}. 
    \item[(2)] Second, their stability result requires certain genericity conditions on the persistence diagrams whereas our edit distance stability result,~\cref{thm: stability of prhi of zb}, assumes no such genericity condition.

    \end{enumerate}

    Despite these apparent dissimilarities, our Grassmannian persistence diagram construction can be related to the construction of Basu and Cox in~\cite{harmonicph}. In~\cref{subsec: harmonicph}, we provide an explicit isomorphism, defined as a projection, from $\prhi \left(\ZB_\dgr^\Ffunc\right)((\ell_i, \ell_j))$ to the construction from~\cite{harmonicph}. Note that item (1) above indicates that this isomorphism is, in general, not trivial.  

\end{remark}

\begin{remark}[Strength]
    Grassmannian persistence diagrams are stronger invariants than classical persistence diagrams. That the Grassmannian persistence diagram is at least as strong as the classical one follows from the fact that the degree-$\dgr$ persistence diagram of a filtration $\Ffunc$ can be recovered from the degree-$\dgr$ Grassmannian persistence diagram $\prhi \left(\ZB_\dgr^\Ffunc\right) : \overline{\lp}^\times \to \gr\left(C_\dgr^K\right)$ by recording the dimension $\dim \left(\prhi \left(\ZB_\dgr^\Ffunc\right) ((\ell_i, \ell_j))\right)$ for every segment $(\ell_i, \ell_j) \in \Seg(\lp) \setminus \diag (\lp)$. To justify that Grassmannian persistence diagrams are indeed strictly stronger than classical persistence diagrams, in~\cref{appendix: edit} (see \cref{fig:degree0example}), we exhibit two different filtrations whose degree-$0$ persistence diagrams coincide whereas their degree- $0$ Grassmannian persistence diagrams are distinct. Moreover, in~\cref{appendix: edit}, we also show that the edit distance stability of Grassmannian persistence diagrams improves upon the edit distance stability of persistence diagrams~\cite[Theorem 8.4]{edit} in a sense made precise by \cref{thm: classical pd lower bound} below.
\end{remark}

\subsection{Computation}\label{subsec: computation of loi zb}

    Note that the definition of degree-$\dgr$ Grassmannian persistence diagram of a filtration $\Ffunc:\{\ell_1<\cdots<\ell_m\}\to \subcx(K)$ directly yields an algorithm for its computation. That is, for every segment $(\ell_i,\ell_j)$ we compute and store $\ZB_\dgr^\Ffunc((\ell_i,\ell_j))$ and then apply the definition of $\times$-Linear Orthogonal Inversion,~\cref{defn: x-harmonic inversion}. An immediate improvement of this naive algorithm would be to make use of~\cref{prop: linAlg Mobius}, which states that the $\times$-Linear Orthogonal Inversion of birth-death spaces can be computed by employing the operation $\ominus$ once, as opposed to the definition of $\times$-Linear Orthogonal Inversion which requires involving the operation $\ominus$ three times. That is,
    \begin{align*}
        \prhi \left(\ZB_\dgr^\Ffunc\right) ((\ell_i, \ell_j)) :&= \big ( \ZB_\dgr^\Ffunc((\ell_i, \ell_{j})) \ominus \ZB_\dgr^\Ffunc((\ell_i, \ell_{j-1})) \big ) \ominus \big ( \ZB_\dgr^\Ffunc((\ell_{i-1}, \ell_{j})) \ominus \ZB_\dgr^\Ffunc((\ell_{i-1}, \ell_{j-1}))  \big ) \\
        &= \ZB_\dgr^\Ffunc ((\ell_i, \ell_j)) \ominus \left( \ZB_\dgr^\Ffunc((\ell_{i-1},\ell_j)) + \ZB_\dgr^\Ffunc((\ell_{i},\ell_{j-1})) \right).
    \end{align*}

More precisely, we have the following algorithm to compute the degree-$\dgr$ Grassmannian persistence diagram of a filtration $\Ffunc$.

\begin{algorithm}[H]
    \caption{Compute degree-$\dgr$ Grassmannian persistence diagram of a filtration}\label{algo: gpd}
    \begin{algorithmic}[1]
     \STATE \textbf{Input:} $\Ffunc : \lp = \{ \ell_1<\cdots <\ell_m \} \to \subcx(K)$
     \STATE \textbf{Output:} $\prhi \left(\ZB_\dgr^\Ffunc\right) : \Seg(\lp) \to \gr(C_\dgr^K)$
     \FOR{$i$ in $\{ 1,\ldots , m \}$}
        \FOR{$j$ in $\{i,\ldots , m , \infty \}$}
            \STATE Compute and store $\ZB_\dgr^\Ffunc((\ell_i, \ell_j))$
            \STATE Compute $\prhi \left(\ZB_\dgr^\Ffunc\right) ((\ell_i, \ell_j)) = \ZB_\dgr^\Ffunc ((\ell_i, \ell_j)) \ominus \left( \ZB_\dgr^\Ffunc((\ell_{i-1},\ell_j)) + \ZB_\dgr^\Ffunc((\ell_{i},\ell_{j-1})) \right)$
        \ENDFOR
    \ENDFOR
    \RETURN{}  $\prhi \left(\ZB_\dgr^\Ffunc\right) : \Seg(\lp) \to \gr(C_\dgr^K)$
    \end{algorithmic}
\end{algorithm}

    In~\cref{subsec: computational complexity}, we study this algorithm and conclude that its time complexity is
\[
O\left( m^2 \cdot \left( n_\dgr^K \cdot n_{\dgr-1}^K \cdot \min\left(n_\dgr^K, n_{\dgr-1}^K\right) + n_{\dgr+1}^K \cdot n_{\dgr}^K \cdot \min\left(n_{\dgr+1}^K, n_{\dgr}^K\right) +    \left(n_\dgr^K\right)^3 \right)    \right),
\]
where $n_\dgr^K$ denotes the number of $\dgr$-simplices of $K$. If we assume that $n_{\dgr-1}^K$ and $n_\dgr^{K}$ are bounded by $n_{\dgr+1}^K$, i.e., $n_{\dgr-1}^K = O\left( n_{\dgr+1}^K \right)$ and $n_\dgr^K = O\left( n_{\dgr+1}^K \right)$, then, the computational complexity boils down to $$O\left(m^2 \cdot \left ( n_{\dgr+1}^K \right )^3\right).$$ Note that the conditions $n_{\dgr-1}^K = O\left( n_{\dgr+1}^K \right)$ and $n_\dgr^K = O\left( n_{\dgr+1}^K \right)$ hold in many practical scenarios, especially for Vietoris-Rips and \v{C}ech complexes. For instance, considering Vietoris-Rips filtration of a finite metric space $(X,d_X)$, for $\dgr=1$, one finds that $n_{\dgr-1}^K = |X|$, $n_\dgr^K = O \left( |X|^2 \right)$ and $n_{\dgr+1}^K = O \left( |X|^3\right)$.

\section{Relations to Other Constructions}\label{sec: relation to other}

In this section we explore the relations between the Grassmannian persistence diagrams and other constructions. 

In~\cref{subsec: relation to classical pd}, we show that 1-parameter Grassmannian persistence diagrams generalize the classical notion of persistence diagrams. We establish this by proving that the classical persistence diagram of a 1-parameter filtration $\Ffunc$ can be derived from its Grassmannian persistence diagram; see~\cref{prop: dim of loi is classical pd}. Additionally, we demonstrate that the edit distance between classical persistence diagrams provides a lower bound for the edit distance between the Grassmannian persistence diagrams of the corresponding filtrations; see~\cref{thm: classical pd lower bound}. We illustrate the fact that Grassmannian persistence diagrams are strictly more discriminative than the classical persistence diagrams in~\cref{example: same pd different 0-prhi}.

In~\cref{subsec: harmonicph}, we examine the relationship between Grassmannian persistence diagrams and the notion of \emph{Harmonic Barcodes}, introduced by Basu and Cox in~\cite{harmonicph}. Harmonic Barcodes closely resemble Grassmannian persistence diagrams in that they also associate a subspace of the cycle space to each segment $(\ell_i, \ell_j)$. We prove that the subspaces determined by Grassmannian persistence diagrams and Harmonic Barcodes are isomorphic via a specific projection; see~\cref{thm: proj isom}.

In~\cref{subsec: orthogonal inversion of Laplacian kernels}, we establish a connection between Grassmannian persistence diagrams and persistence Laplacians. Specifically, we show in~\cref{thm: equality of different HIs} that the Grassmannian persistence diagram of a $1$-parameter filtration can also be constructed from persistent Laplacian kernels via another variant of orthogonal inversion, namely $\supseteq$-Linear Orthogonal Inversion (\cref{defn: sup-harmonic inversion}), which is tailored for invariants that are \emph{compatible} with the reverse inclusion order.

In~\cref{subsec: treegrams}, we demonstrate that the notion of treegrams, which generalizes dendrograms, is equivalent to degree-0 Grassmannian persistence diagrams. While we first establish this equivalence through a direct but non-constructive argument in~\cref{thm: equivalence of treegrams and degree 0 orthogonal inversions}, we later provide an algorithmic/constructive proof in~\cref{appendix: construction}.

\subsection{Classical Persistence Diagrams}\label{subsec: relation to classical pd}

In this section, we demonstrate that classical persistence diagrams can be recovered from $1$-parameter Grassmannian persistence diagrams; see~\cref{prop: dim of loi is classical pd}. Consequently, Grassmannian persistence diagrams constitute strictly stronger invariants than classical persistence diagrams. This enhanced discriminating power is further illustrated in~\cref{example: same pd different 0-prhi}. In addition, we show that the edit distance between classical persistence diagrams (as defined in~\cref{defn: cost of morphism between nonnegative functions}) provides a lower bound for the edit distance between Grassmannian persistence diagrams; see~\cref{thm: classical pd lower bound}.

\begin{definition}[Classical Persistence Diagrams~\cite{cohen-steiner2007}]\label{defn: classical 1-parameter pd}
    Let $\Ffunc : \lp  \to \subcx(K)$ be a $1$-parameter filtration. Let $\dgr\geq 0$ be an integer and write $\lp = \{ \ell_1<\cdots<\ell_n\}$. The \emph{classical degree-$\dgr$ persistence diagram} of $\Ffunc$ is defined to be the function $\pd_\dgr^\Ffunc : \Seg(\lp) \to \Z_{\geq 0}$ given by
    \begin{align*}
        \pd_\dgr^\Ffunc ((\ell_i, \ell_j)) &:= \beta_\dgr^{\ell_{i}, \ell_{j-1}} -\beta_\dgr^{\ell_{i-1}, \ell_{j-1}} + \beta_\dgr^{\ell_{i-1}, \ell_{j}} - \beta_\dgr^{\ell_{i}, \ell_{j}},\\
        \pd_\dgr^\Ffunc ((\ell_i, \infty)) &:= \beta_\dgr^{\ell_{i}, \ell_{n}} - \beta_\dgr^{\ell_{i-1}, \ell_{n}}, \\
        \pd_\dgr^\Ffunc ((\ell_i, \ell_i)) &:= 0,
    \end{align*}
    where $\beta_q^{\cdot, \cdot}$ denote the persistent Betti numbers as defined in~\cref{defn: persistent betti rank inv}.
\end{definition}

We now present our result that Grassmannian persistence diagrams recover classical persistence diagrams.

\begin{proposition}\label{prop: dim of loi is classical pd}
    Let $\Ffunc : \lp  \to \subcx(K)$ be a $1$-parameter filtration. Let $\dgr\geq 0$ be an integer and write $\lp = \{ \ell_1<\cdots<\ell_n\}$. Then, for any $(\ell_i, \ell_j) \in \Seg(\lp) \setminus \diag(\lp)$, we have
    \[
    \dim \left( \prhi \left( \ZB_\dgr^\Ffunc \right) ((\ell_i, \ell_j))  \right) = \pd_\dgr^\Ffunc((\ell_i, \ell_j)).
    \]
\end{proposition}

\begin{proof}
    By~\cref{cor: linAlg Mobius dim}, for every segment $(\ell_i, \ell_j) \in \Seg(\lp) \setminus \diag(\lp)$, the dimension of the vector space $\prhi \left(\ZB_\dgr^\Ffunc \right)((\ell_i, \ell_j))$ is given by 
\[
\dim \left(\ZB_\dgr^\Ffunc ((\ell_i, \ell_j))\right) - \dim \left(\ZB_\dgr^\Ffunc ((\ell_i, \ell_{j-1}))\right) +\dim \left(\ZB_\dgr^\Ffunc ((\ell_{i-1}, \ell_{j-1}))\right) - \dim \left(\ZB_\dgr^\Ffunc ((\ell_{i-1}, \ell_j))\right).
\]

As shown in~\cite[Section 9.1]{edit}, this number is precisely the multiplicity of the segment $(\ell_i,\ell_j)$ in the classical degree-$\dgr$ persistence diagram of the filtration $\Ffunc : \lp \to \Fil(K)$. 
\end{proof}

We now present the edit distance stability of classical persistence diagrams and present our result that the edit distance between classical persistence diagrams is a lower bound the edit distance between Grassmannian persistence diagrams; see~\cref{thm: classical pd lower bound}.

\begin{definition}[Charge-preserving morphisms]
    Let $\lp_1$ and $\lp_2$ be two finite linearly ordered metric posets. Let $\omega_1 : \Seg\left(\lp_1\right) \to \Z_{\geq 0}$ and $\omega_2: \Seg\left(\lp_2\right) \to \Z_{\geq 0}$ be two (not necessarily order-preserving) non-negative integral functions. A~\emph{charge-preserving morphism} from $\omega_1$ to $\omega_2$ is any Galois connection $\ladj{f} : \lp_1 \leftrightarrows \lp_2 : \radj{f}$ such that
\[
\omega_2(J) = \sum_{I \in \left(\overline{\ladj{f}}\right)^{-1} (J)} \omega_1(I)
\]
for every $J \in \Seg(\lp_2) \setminus \diag(\lp_2)$.
\end{definition} 

\begin{notation}
    We denote by $\fnc_{\geq 0}$ the category where
    \begin{itemize}
        \item Objects are non-negative integral functions $\omega : \Seg(\lp)\to \Z_{\geq 0}$ where $\lp$ is any finite linearly ordered metric poset,
        \item Morphisms are charge-preserving morphisms.
    \end{itemize}
\end{notation}

\begin{definition}[Cost of a morphism in $\fnc_{\geq 0}$~{\cite[Section 7.3]{edit}}]\label{defn: cost of morphism between nonnegative functions}
    The \emph{cost} of a morphism $f = (\ladj{f},\radj{f})$ in $\fnc_{\geq 0}$, denoted $\cost_{\fnc_{\geq 0}} (f)$, is defined to be $\cost_{\fnc_{\geq 0}} (f) := \dis(\ladj{f})$, the distortion of the left adjoint $\ladj{f}$. 
\end{definition}

As the classical persistence diagrams are non-negative integral functions, they are objects in $\fnc_{\geq 0}$. Moreover, for a filtration $\Ffunc : \lp \to \subcx(K)$, $\partial_{(\Seg(\lp),\prodord)} \left( \ZB_\dgr^\Ffunc\right) = \pd_\dgr^\Ffunc$ on $\Seg(L) \setminus \diag(\lp)$ as shown in~\cite[Section 9.1]{edit}. Hence, the functorial pipeline of obtaining persistence diagrams, as outlined in~\cite{edit}, leads to the following stability result. 

\begin{theorem}[Edit distance stability of classical persistence diagrams]
    Let $\Ffunc : \lp_1 \to \subcx(K)$ and $\Gfunc : \lp_2 \to \subcx(K)$ be two $1$-parameter filtrations of a finite simplicial complex $K$ indexed by finite linearly ordered metric posets and let $\pd_\dgr^\Ffunc : \Seg(\lp_1) \to \Z_{\geq 0}$ and $\pd_\dgr^\Gfunc : \Seg(\lp_2) \to \Z_{\geq 0}$ be their respective degree-$\dgr$ persistence diagrams. Then, 
    \[
    d_{\fnc_{\geq 0}}^E \left( \pd_\dgr^\Ffunc ,\pd_\dgr^\Gfunc \right) \leq d_{\Fil(K)}^E (\Ffunc, \Gfunc).
    \]
\end{theorem}

\begin{proof}
    This result directly follows from the functorial pipeline established in~\cite[Section 8]{edit}.
\end{proof}

\begin{remark}
    It is important to emphasize the differences in the setup used here compared to~\cite[Section 8]{edit}. A key distinction is that here $\fnc_{\geq 0}$ consists of nonnegative functions defined over segments of finite linearly ordered posets, whereas the setup in~\cite[Section 8]{edit} was more general, considering functions valued in $\mathbb{Z}$ (i.e., signed persistence diagrams) defined over segments of finite lattices.

    In the more general setting, it was later observed that the edit distance between signed persistence diagrams could become trivial even when the diagrams differ. To address this issue, the authors proposed a modification in~\cite[Erratum]{edit-arxiv}. Nevertheless, their original approach remains valid and produces a meaningful, nontrivial distance when restricted to linearly ordered posets and nonnegative functions. In fact, the nontriviality of the edit distance between classical persistence diagrams is implicitly established in~\cite[Theorem 9.1]{edit}, where it is shown that the edit distance between classical persistence diagrams is bi-Lipschitz equivalent to the well-known bottleneck distance between these diagrams.
\end{remark}

We now show that the edit distance between classical persistence diagrams is a lower bound for the edit distance between $1$-parameter Grassmannian persistence diagrams.

\begin{theorem}[Lower bound]\label{thm: classical pd lower bound}
    Let $\Ffunc : \lp_1 \to \subcx(K)$ and $\Gfunc : \lp_2 \to \subcx(K)$ be two filtrations of a finite simplicial complex $K$ indexed by finite linearly ordered metric posets. Then, for any degree $\dgr\geq 0$, we have
    \begin{align*}
    d_{\fnc_{\geq 0}}^E \left( \pd_\dgr^\Ffunc, \pd_\dgr^\Gfunc\right) \leq d_{\inndgm\left(C_\dgr^K\right)}^E \left(\prhi \left(\ZB_\dgr^\Ffunc\right), \prhi \left(\ZB_\dgr^\Gfunc\right)\right).
    \end{align*}
\end{theorem}

The proof of~\cref{thm: classical pd lower bound} is given in~\cref{appendix: edit}.

\begin{remark}\label{rmk: gpd more discriminative}
    Combining the stability of $1$-parameter Grassmannian persistence diagrams (\cref{thm: stability of prhi of zb}) with the lower bound result (\cref{thm: classical pd lower bound}), one can see that 
    \[
    d_{\fnc_{\geq 0}}^E \left( \pd_\dgr^\Ffunc, \pd_\dgr^\Gfunc\right) \leq d_{\inndgm\left(C_\dgr^K\right)}^E \left(\prhi \left(\ZB_\dgr^\Ffunc\right), \prhi \left(\ZB_\dgr^\Gfunc\right)\right) \leq d_{\Fil(K)}^E (\Ffunc, \Gfunc).
    \]

In other words, the edit distance between the degree-$\dgr$ Grassmannian persistence diagrams mediates between the edit distance between classical degree-$\dgr$ persistence diagrams and the edit distance between filtrations. We provide an example to further illustrate that $d_{\fnc_{\geq 0}}^E \left( \pd_\dgr^\Ffunc, \pd_\dgr^\Gfunc\right)$ can be $0$ while the term $d_{\inndgm\left(C_\dgr^K\right)}^E \left(\prhi \left(\ZB_\dgr^\Ffunc\right), \prhi \left(\ZB_\dgr^\Gfunc\right)\right)$ is positive; see~\cref{example: same pd different 0-prhi}.
    
\end{remark}

\subsection{Harmonic Barcodes}\label{subsec: harmonicph}

In this section, we connect our construction of Grassmannian persistence diagrams with a similar construction—harmonic barcodes—introduced in~\cite{harmonicph}. The main result in this section is presented in~\cref{thm: proj isom}, which states that Grassmannian persistence diagrams and harmonic barcodes are related through a particular projection. We start by recalling the definition of \emph{persistent homology group} from~\cite{robins1999towards,edelsbrunner2010computational} and some related definitions from~\cite{harmonicph}.

\begin{definition}[Persistent homology group]\label{defn: m n and p spaces}
    Let $\Ffunc : \lp = \{ \ell_1 < \cdots < \ell_n \} \to \subcx(K)$ be a filtration. For any $(\ell_i,\ell_j) \in \Seg(P)$, let $\iota_\dgr^{i,j} : H_\dgr (K_i) \to H_\dgr (K_j)$ denote the homomorpshim induced by the inclusion $K_i \hookrightarrow K_j$. The~\emph{persistent homology group}, $H_\dgr^{i,j}(\Ffunc)$, of $\Ffunc$ is defined by 
    \[
    H_\dgr^{i,j} (\Ffunc) := \Ima \left(\iota_\dgr^{i,j}\right).
    \]
For $(\ell_i,\ell_j)\in \Seg(\lp)$, also define
    \begin{align*}
        M_\dgr^{i,j}(\Ffunc) &:= \left(\iota_\dgr^{i,j}\right)^{-1} \left(H_\dgr^{i-1, j} (\Ffunc)\right)\subseteq H_\dgr(K_i).
    \end{align*}
For $(\ell_i,\ell_j) \in \Seg(\lp) \setminus \diag(\lp)$, i.e., $i < j$, define
    \begin{align*}
        N_\dgr^{i,j}(\Ffunc) &:= M_\dgr^{i,j-1}(\Ffunc) = \left(\iota_\dgr^{i,j-1}\right)^{-1} \left(H_\dgr^{i-1, j-1} (\Ffunc)\right)\subseteq H_\dgr(K_i)\\
        &\,\,\text{and}\\
        P_\dgr^{i,j}(\Ffunc) &:= \frac{M_\dgr^{i,j}(\Ffunc)}{ N_\dgr^{i,j}(\Ffunc)}.
    \end{align*}
    where $K_0$ is taken to be equal to $K_1$ by convention. 
\end{definition}

As shown in~\cite[Proposition 3.8]{harmonicph}, the interpretation of the subspaces $M^{i,j}_\dgr(F),N^{i,j}_\dgr(F)$ and $P^{i,j}_\dgr(F)$ are as follows.
    \begin{itemize}
        \item $M_\dgr^{i,j}(\Ffunc)$ is a subspace of $H_\dgr(K_i)$ consisting of homology classes in $H_\dgr(K_i)$ which  
        \begin{center}
            ``$\big($are born before $\ell_i \big)$ or $\big($(born at $\ell_i$) and (die at $\ell_j$ or earlier)$\big)$.''
        \end{center}
        \item $N_\dgr^{i,j}(\Ffunc)$ is a subspace of $H_\dgr(K_i)$ consisting of homology classes in $H_\dgr(K_i)$ which  
        \begin{center}
            ``$\big($are born before $\ell_i\big)$ or $\big($(born at $\ell_i$) and (die strictly earlier than $\ell_j$)$\big)$.''
        \end{center}
        \item $P_\dgr^{i,j}(\Ffunc)$ is the space of equivalence classes of $\dgr$-dimensional cycles which 
        \begin{center}
            ``are born exactly at $\ell_i$ and die exactly at $\ell_j$.''
        \end{center}
    \end{itemize}

\begin{definition}[{\cite[Definition 4.12]{gpd-multi}}]
    For a simplicial complex $K$ and any degree $ \dgr \geq 0$, let $$\phi_\dgr^K : \Zfunc_\dgr(K) \to \frac{\Zfunc_\dgr(K)}{  \Bfunc_\dgr(K)} = H_\dgr(K)$$ denote the canonical quotient map. For a filtration $\Ffunc : \lp = \{\ell_1 < \cdots <\ell_n \} \to \subcx(K)$, we define, for $\ell_i<\ell_j\in \lp$,
    \begin{align*}
        \tilde{M}_\dgr^{i,j}(\Ffunc) &:= \left(\phi_\dgr^{K_i}\right)^{-1} \left(M_\dgr^{i,j}(\Ffunc)\right) \subseteq \Zfunc_\dgr (K), \\
        \tilde{N}_\dgr^{i,j}(\Ffunc) &:= \left(\phi_\dgr^{K_i}\right)^{-1} \left(N_\dgr^{i,j}(\Ffunc)\right) \subseteq \Zfunc_\dgr (K).
    \end{align*}
    Observe that $\tfrac{\tilde{M}_\dgr^{i,j}(\Ffunc)}{ \Bfunc_\dgr (K)} = M_\dgr^{i,j}(\Ffunc)$ and $\tfrac{\tilde{N}_\dgr^{i,j}(\Ffunc)}{ \Bfunc_\dgr (K)} = N_\dgr^{i,j}(\Ffunc)$.
\end{definition}

\begin{definition}[Harmonic homology space~{\cite[Definition 2.6]{harmonicph}}]
    The~\emph{harmonic homology space} of $K$ is the subspace $\mathcal{H}_\dgr(K) \subseteq C_\dgr^K$ defined by
    \[
    \mathcal{H}_\dgr(K) := \Zfunc_\dgr(K) \cap \Bfunc_\dgr(K)^\perp,
    \]
    where $C_\dgr^K$ is endowed with the standard inner product as described in~\cref{sec: prelim}.
\end{definition}

\begin{proposition}[{\cite[Proposition 2.7]{harmonicph}}]
    The map $\mathfrak{f}_\dgr(K) : H_\dgr(K) \to \mathcal{H}_\dgr(K)$ defined by 
    \[
    z + \Bfunc_\dgr(K) \mapsto \proj_{\left(\Bfunc_\dgr(K)\right)^\perp} (z)
    \]
    is an isomorphism of vector spaces.
\end{proposition}

\begin{proposition}[{\cite[Proposition 2.11]{harmonicph}}]\label{prop: commutative harmonic}
    Let $\Ffunc = \{ K_i \}_{i=1}^n$ be a filtration (i.e., $\Ffunc : \lp = \{\ell_1< \cdots <\ell_n\} \to \subcx(K)$)). For $ i \leq j$, the restriction of $\proj_{(\Bfunc_\dgr(K_j))^\perp}$ to $\mathcal{H}_\dgr(K_i)$ gives a linear map
    \[
    \gamma_\dgr^{i,j} := \left. \left(\proj_{(\Bfunc_\dgr(K_j))^\perp}\right) \right\vert_{\mathcal{H}_\dgr(K_i)} : \mathcal{H}_\dgr(K_i) \to \mathcal{H}_\dgr(K_j)
    \]
    which makes the following diagram commute.
    \begin{center}
        \begin{tikzcd}
H_\dgr(K_i) \arrow[rr, "{\iota_\dgr^{i,j}}"] \arrow[dd, "\mathfrak{f}_\dgr(K_i)"] &  & H_\dgr(K_j) \arrow[dd, "\mathfrak{f}_\dgr(K_j)"] \\
                                                                         &  &                                            \\
\mathcal{H}_\dgr(K_i) \arrow[rr, "{\gamma_\dgr^{i,j}}"]                  &  & \mathcal{H}_\dgr(K_j)                        
\end{tikzcd}
    \end{center}
\end{proposition}

\begin{definition}[Harmonic persistent homology group~{\cite[Definitions 3.11 and 3.12]{harmonicph}}]
    Let $\Ffunc = \{ K_i\}_{i=1}^n$ be a filtration. Let $\gamma_\dgr^{i,j} : \mathcal{H}_\dgr (K_i) \to \mathcal{H}_\dgr (K_j)$ denote the maps defined in~\cref{prop: commutative harmonic}. The~\emph{harmonic persistent homology group}, $\mathcal{H}_\dgr^{i,j}(\Ffunc)$, of $\Ffunc$ is defined by 
    \[
    \mathcal{H}_\dgr^{i,j} (\Ffunc) := \Ima \left(\gamma_\dgr^{i,j}\right).
    \]
For $i \leq j$, also define
    \begin{align*}
        \mathcal{M}_\dgr^{i,j}(\Ffunc) &:= \left(\gamma_\dgr^{i,j}\right)^{-1} \left(\mathcal{H}_\dgr^{i-1, j} (\Ffunc)\right).
    \end{align*}
For $i < j$, define
    \begin{align*}
        \mathcal{N}_\dgr^{i,j}(\Ffunc) &:= \left(\gamma_\dgr^{i, j-1}\right)^{-1} \left(\mathcal{H}_\dgr^{i-1, j-1} (\Ffunc)\right) \\
        \mathcal{P}_\dgr^{i,j}(\Ffunc) &:= \mathcal{M}_\dgr^{i,j}(\Ffunc) \cap \left(\mathcal{N}_\dgr^{i,j}(\Ffunc)\right)^\perp,
    \end{align*}
    where $K_0$ is taken to be equal to $K_1$ by convention. The map 
    \begin{align*}
        \HB_\dgr^\Ffunc : \Seg(\lp)\setminus \diag(\lp) &\to \gr(C_\dgr^K) \\
                                    (\ell_i,\ell_j)  &\mapsto \mathcal{P}_\dgr^{i,j}(\Ffunc)
    \end{align*}
    is called the degree-$\dgr$ \emph{Harmonic barcode} of $\Ffunc$.
\end{definition}

We now establish the connection between the degree-$\dgr$ Grassmannian persistence diagram and the degree-$\dgr$ Harmonic barcode of a filtration $\Ffunc$.

\begin{theorem}\label{thm: proj isom}
    Let $\Ffunc : \lp \to \subcx(K)$ be a $1$-parameter filtration. For $i<j$, $$\proj_{\left(\mathcal{N}_\dgr^{i,j}\right)^\perp} : \prhi \left(\ZB_\dgr^\Ffunc\right) ((\ell_i,\ell_j)) \to \mathcal{P}_\dgr^{i,j}(\Ffunc)$$ is an isomorphism.
\end{theorem}

\begin{remark}\label{rmk: non-isometry}
    Note that the linear isomorphism stated in the theorem above is not necessarily an isometry between the two subspaces of $C_\dgr^K$, as it is defined by a projection. A projection would be an isometry if the subspaces were identical, but this is not generally the case for $\prhi \left(\ZB_\dgr^\Ffunc\right) ((\ell_i,\ell_j))$ and $\mathcal{P}_\dgr^{i,j}(\Ffunc)$; see~\cite[Example 1.1]{harmonicph}. 
\end{remark}

Note that the dimensions of the vector spaces $\prhi \left(\ZB_\dgr^\Ffunc \right)((i,j))$ and $\mathcal{P}_\dgr^{i,j}(\Ffunc)$ are the same as both are equal to the number of linearly independent cycles that are born at $i$ and die at $j$. So, it is already known that $\prhi \left(\ZB_\dgr^\Ffunc\right) ((i,j))$ and $\mathcal{P}_\dgr^{i,j}(\Ffunc)$ are isomorphic.  \cref{thm: proj isom} shows that this isomorphism can be written explicitly as a projection.

We will need the following proposition to prove~\cref{thm: proj isom}.

\begin{proposition}\label{lem: linear varphi isom}
    Let $\Ffunc : \lp = \{\ell_1 < \cdots < \ell_n \} \to \subcx(K)$ be a $1$-parameter filtration. Then, the map
    \begin{align}
        \varphi : \frac{\ZB_\dgr^\Ffunc((\ell_i,\ell_j))}{\ZB_\dgr^\Ffunc((\ell_{i-1}, \ell_j)) + \ZB_\dgr^\Ffunc ((\ell_i, \ell_{j-1}))} &\to \frac{\tilde{M}_\dgr^{i,j}(\Ffunc)}{\tilde{N}_\dgr^{i,j}(\Ffunc)} \\
        &\\
        z + \left( \ZB_\dgr^\Ffunc((\ell_{i-1}, \ell_j)) + \ZB_\dgr^\Ffunc ((\ell_i, \ell_{j-1})) \right) &\mapsto z + \tilde{N}_\dgr^{i,j}(\Ffunc)
    \end{align}
    is an isomorphism.
\end{proposition}

\begin{proof}
    For notational simplicity, we will use $(i, j)$ to denote the segment $(\ell_i, \ell_j) \in \Seg(\lp)$.
    Observe that we have 
    \[
    \ZB_\dgr^\Ffunc((i-1, j)) + \ZB_\dgr^\Ffunc ((i, j-1)) = \sum_{(a,c)<_\times (i,j)} \ZB_\dgr^\Ffunc((a,c)).
    \]
    Therefore, by~\cite[Lemma 4.14]{gpd-multi}, we have that the map $\varphi$ is a surjection. Observe that the surjectivity of $\varphi$ in~\cite[Lemma 4.14]{gpd-multi} is proved by showing that 
    \[
    \sum_{(a,c)<_\times (i,j)} \ZB_\dgr^\Ffunc((a,c)) \subseteq \ker(\psi),
    \]
    where $\psi : \ZB_\dgr^\Ffunc((i,j)) \to \frac{\tilde{M}_\dgr^{i,j}(\Ffunc)}{\tilde{N}_\dgr^{i,j}(\Ffunc)}$ is defined through $\psi(z) := z + \tilde{N}_\dgr^{i,j}(\Ffunc)$. Therefore, to show that $\varphi$ is an isomorphism, it suffices to prove 
    \[
    \ker(\psi) \subseteq \ZB_\dgr^\Ffunc((i-1, j)) + \ZB_\dgr^\Ffunc ((i, j-1)),
    \]
    which would imply injectivity. Observe that, since $\Ffunc$ is a $1$-parameter filtration, \cite[Lemma 4.13]{gpd-multi} boils down to following:
    \begin{align*}
        \tilde{M}_\dgr^{i,j}(\Ffunc) &= \left\{ z \in \Zfunc_\dgr (K_i) \mid \exists z' \in \Zfunc_\dgr(K_{i-1})  \text{ such that } z-z' \in \Bfunc_\dgr(K_j) \right\} \\
        \tilde{N}_\dgr^{i,j}(\Ffunc) &= \{ z \in \Zfunc_\dgr (K_i) \mid \exists z' \in \Zfunc_\dgr(K_{i-1})  \text{ such that } z-z' \in \Bfunc_\dgr(K_{j-1}) \} = \tilde{M}_\dgr^{i,j-1}(\Ffunc) 
    \end{align*}
    Let $x\in \ker(\psi)$. Then, $x\in \tilde{N}_\dgr^{i,j}(\Ffunc)$. Therefore, there exists $z' \in \Zfunc_\dgr(K_{i-1})$ and $\eta \in \Bfunc_\dgr(K_{j-1})$ such that $x -z' =\eta$, i.e., $x = z' + \eta$. Observe that $z' = x-\eta \in \ZB_\dgr^\Ffunc((i-1,j))$ as $z\in \Zfunc_\dgr(K_{i-1})$ and $x,\eta \in \Bfunc_\dgr(K_j)$. Observe also that $\eta = x-z' \in \ZB_\dgr^\Ffunc((i, j-1))$ as $\eta \in \Bfunc_\dgr(K_{j-1})$ and $x, z' \in \Zfunc_\dgr(K_i)$. Therefore, $x\in \ZB_\dgr^\Ffunc ((i-1, j)) +\ZB_\dgr^\Ffunc((i,j-1))$. Thus, $\ker(\psi) \subseteq \ZB_\dgr^\Ffunc ((i-1, j)) +\ZB_\dgr^\Ffunc((i,j-1))$.
\end{proof}

\begin{proof}[Proof of~\cref{thm: proj isom}]
    Let
    \[
    \mu_\dgr^{i,j} : \frac{M_\dgr^{i,j}(\Ffunc)}{N_\dgr^{i,j}(\Ffunc)} \to \frac{\mathcal{M}_\dgr^{i,j}(\Ffunc)}{\mathcal{N}_\dgr^{i,j}(\Ffunc)}
    \]
    be defined by 
    $$\mu_\dgr^{i,j} ([z] + N_\dgr^{i,j}(\Ffunc)) := \proj_{(\Bfunc_\dgr(K_i))^\perp} (z) + \mathcal{N}_\dgr^{i,j}(\Ffunc)$$ 
    for $[z] + N_\dgr^{i,j}(\Ffunc) \in \frac{M_\dgr^{i,j}(\Ffunc)}{N_\dgr^{i,j}(\Ffunc)}$. By~\cref{prop: commutative harmonic}, $\mu_\dgr^{i,j}$ is a well-defined isomorphism. Moreover, we have the following isomorphism
    \[
    \proj_{\left(\mathcal{N}_\dgr^{i,j}(\Ffunc)\right)^\perp} : \frac{\mathcal{M}_\dgr^{i,j}(\Ffunc)}{\mathcal{N}_\dgr^{i,j}(\Ffunc)} \to \mathcal{M}_\dgr^{i,j}(\Ffunc) \cap (\mathcal{N}_\dgr^{i,j}(\Ffunc))^\perp.
    \]
    Let 
    \[
    \theta : \frac{\tilde{M}_\dgr^{i,j}(\Ffunc)}{\tilde{N}_\dgr^{i,j}(\Ffunc)} \to \frac{\tilde{M}_\dgr^{i,j}(\Ffunc) / \Bfunc_\dgr (\Ffunc(i))}{\tilde{N}_\dgr^{i,j}(\Ffunc) / \Bfunc_\dgr(\Ffunc(i))} = \frac{M_\dgr^{i,j}(\Ffunc)}{N_\dgr^{i,j}(\Ffunc)}
    \]
    be the canonical isomorphism.
    Combining the fact that 
    \begin{align*}
         \prhi \left(\ZB_\dgr^\Ffunc\right) ((i,j)) = \goi \left( \ZB_\dgr^\Ffunc\right)((i,j)) &= \ZB_\dgr^\Ffunc((i,j)) \ominus \left( \ZB_\dgr^\Ffunc((i-1,j)) + \ZB_\dgr^\Ffunc ((i,j-1)) \right)   \\
         &\simeq \frac{\ZB_\dgr^\Ffunc((i,j))}{\ZB_\dgr^\Ffunc((i-1,j)) + \ZB_\dgr^\Ffunc ((i,j-1)) }
    \end{align*}
    with the isomorphisms $\mu_\dgr^{i,j}$, $\theta$ and the one described in~\cref{lem: linear varphi isom} we obtain the isomorphism
    \begin{align*}
    \prhi \left(\ZB_\dgr^\Ffunc\right) ((i,j)) &\to \mathcal{P}_\dgr^{i,j} (\Ffunc) \\
                          z  &\mapsto \proj_{\left(\mathcal{N}_\dgr^{i,j}(\Ffunc)\right)^\perp} \circ \proj_{\left(\Bfunc_\dgr(K_i)\right)^\perp} (z).
    \end{align*}
Observe that, as $z \in \prhi \left(\ZB_\dgr^\Ffunc\right) ((i,j))$, we have that $z\in (\Bfunc_\dgr(K_i))^\perp$. Therefore, $\proj_{\left(\Bfunc_\dgr(K_i)\right)^\perp} (z) = z$. Thus, the isomorphism above is given by
    \begin{align*}
     \\
                          z  &\mapsto \proj_{\left(\mathcal{N}_\dgr^{i,j}(\Ffunc)\right)^\perp}  (z)
    \end{align*}
    as stated in~\cref{thm: proj isom}.

\end{proof}

\subsection{Persistent Laplacians}\label{subsec: orthogonal inversion of Laplacian kernels}
Recall that our $\times$-Linear Orthogonal Inversion definition (\cref{defn: x-harmonic inversion}) was inspired by the algebraic M\"obius inversion formula (with respect to product order; see~\cref{eqn: times algebraic mob inv1,eqn: times algebraic mob inv2,eqn: times algebraic mob inv3} in~\cref{prop: algebraic mobius inversion formulas}). Indeed, we applied the $\times$-Linear Orthogonal Inversion to $\ZB_\dgr^\Ffunc$ in order to give rise to the degree-$\dgr$ Grassmannian persistence diagram of $\Ffunc$ (obtained from birth-death spaces). Recall also that the classical definition of (generalized) persistence diagrams~\cite{cohen-steiner2007, Patel2018} arises when applying the algebraic M\"obius inversion formula (with respect to the reverse inclusion order)
to the persistent Betti numbers. Different choices of orders in these scenarios are made in order to render each invariant compatible with the chosen order.
Persistent Laplacians enjoy a notion of functoriality that is compatible with the reverse inclusion order $\supseteq$, see~\cite[Section 5.3]{persLap}.

In this section, we introduce the notion of $\supseteq$-Linear Orthogonal Inversion, and apply it to persistent Laplacian kernels of a filtration $\Ffunc$. We refer to the resulting objects as the \emph{degree-$\dgr$ Grassmannian persistence diagram} of $\Ffunc$ (obtained from persistent Laplacian kernels). We show that the degree-$\dgr$ Grassmannian persistence diagram obtained from the persistent Laplacian kernels coincides with the degree-$\dgr$ Grassmannian persistence diagram obtain from the  birth-death spaces away from the diagonal (\cref{thm: equality of different HIs}). As a result of this correspondence, we establish stability for degree-$\dgr$ Grassmannian persistence diagrams obtained from persistent Laplacian kernels; see~\cref{thm: stability of rhi lap}.

\begin{definition}[$\supseteq$-Linear Orthogonal Inversion]\label{defn: sup-harmonic inversion}
    Let $\lp = \{ \ell_1<\cdots<\ell_n \}$ be a finite linearly ordered metric poset and let $\Lfunc : \overline{\lp}^\supseteq \to \gr(V)$ be a function. We define its \emph{$\supseteq$-Linear Orthogonal Inverse}, denoted $\rhi \left(\Lfunc\right)$, to be the function $\rhi \left(\Lfunc\right) : \overline{\lp}^\supseteq \to \gr(V)$ given by
    \begin{align*}
       \rhi \left(\Lfunc\right) ((\ell_i, \ell_j)) &:= \big ( \Lfunc((\ell_i,\ell_j)) \ominus \Lfunc((\ell_i, \ell_{j+1})) \big ) \ominus \big ( \Lfunc((\ell_{i-1},\ell_j)) \ominus \Lfunc((\ell_{i-1}, \ell_{j+1})) \big) \\
       \rhi \left(\Lfunc\right) ((\ell_i, \infty)) &:=  \Lfunc((\ell_i,\infty)) \ominus \Lfunc((\ell_{i-1}, \infty)),
    \end{align*}  
    for $1\leq i  < j \leq n$. 
\end{definition}
\nomenclature[31]{$\rhi$}{$\supseteq$-Linear Orthogonal Inversion}

Notice that after rearranging terms appearing in~\cref{eqn: times algebraic mob inv4,eqn: times algebraic mob inv5} in~\cref{prop: algebraic mobius inversion formulas}, our $\supseteq$-Linear Orthogonal Inversion definition is analogous to the algebraic M\"obius inversion formula with respect to the reverse inclusion order. Moreover, we follow the same convention for the boundary cases as described in~\cref{remark: convention edge cases}. To be precise,
\begin{align*}
        \rhi \left(\Lfunc\right) ((\ell_1, \infty)) &:= \Lfunc((\ell_1, \infty)), \\
        \rhi \left(\Lfunc\right)  ((\ell_1, \ell_n)) &:= \Lfunc ((\ell_1, \ell_{n})) \ominus \Lfunc((\ell_1,\infty)), \\
        \rhi \left(\Lfunc\right)  \Lfunc ((\ell_1, \ell_j)) &:= \Lfunc((\ell_1, \ell_{j})) \ominus \Lfunc((\ell_1, \ell_{j+1})) \text{ for } j < n,
\end{align*}

We now recall the definition of persistent Laplacians. Let $K$ be a finite simplicial complex and suppose that we have a simplicial filtration $\Ffunc = \{K_i \}_{i=1}^n$ of $K$. For $1\leq i \leq j \leq n$ and $\dgr\geq 0$, consider the subspace
\[
C_\dgr^{K_j, K_i}:= \left \{c \in  C_\dgr^{K_j} \mid \partial_\dgr^{K_j}(c) \in C_{\dgr-1}^{K_i}\right \} \subseteq C_\dgr^{K_j}
\]
consisting of $\dgr$-chains such that their image under the boundary map $\partial_\dgr^{K_j}$ lies in the subspace $C_{\dgr-1}^{K_i} \subseteq C_{\dgr-1}^{K_j}$. Let $\partial_\dgr^{K_j,K_i} $ denote the restriction of $\partial_\dgr^{K_j}$ onto $C_\dgr^{K_j,K_i}$ and let $\left ( \partial_\dgr^{K_j,K_i}\right )^*$ denote its adjoint with respect to the standard inner products on $C_\dgr^K$ and $C_{\dgr-1}^K$, as introduced in~\cref{sec: prelim}. 

\begin{center}
\begin{tikzcd}
     C_{\dgr+1}^{K_i}\arrow{rr}{\partial_{\dgr+1}^{K_i}}\arrow[hookrightarrow,dashed,gray]{dd} && C_\dgr^{K_i} \arrow["{\Delta_{\dgr}^{K_i, K_j}}", loop, distance=2.25em, in=55, out=125, blue, thick]\arrow[rr,shift left=.75ex,blue,"\partial_\dgr^{K_i}"]\arrow[hookrightarrow,dashed,gray]{dd}\arrow[dl,shift left=.75ex,blue,"\left(\partial_{\dgr+1}^{K_j,K_i}\right)^*"] && C_{\dgr-1}^{K_i}\arrow[hookrightarrow,dashed,gray]{dd}\arrow[ll,shift left=.75ex,blue,"\left(\partial_\dgr^{K_i}\right)^*"]\\
      &C_{\dgr+1}^{K_j,K_i}\arrow[hookrightarrow,dashed,gray]{dl}\arrow[ur,shift left=.75ex,blue,"\partial_{\dgr+1}^{K_j,K_i}"]&& \,\,\,\,\,\,\,\,\,\, & \\
        C_{\dgr+1}^{K_j}\arrow{rr}{\partial_{\dgr+1}^{K_j}}  && C_\dgr^{K_j}\arrow{rr}{\partial_\dgr^{K_j}} && C_{\dgr-1}^{K_j}
\end{tikzcd}
\end{center}

\noindent One can define the $\dgr$-th~\emph{persistent Laplacian}~\cite{persSpecGraph, lieutier-pdf} $\Delta_\dgr^{K_i, K_j} : C_\dgr^{K_i} \to C_\dgr^{K_i}$ by
\[
\Delta_\dgr^{K_i,K_j} := \partial_{\dgr+1}^{K_j,K_i} \circ \left(\partial_{\dgr+1}^{K_j,K_i}\right)^* + \left(\partial_\dgr^{K_i}\right)^* \circ \partial_\dgr^{K_i}.
\]
\nomenclature[32]{$\Delta_\dgr^{\cdot,\cdot}$}{$\dgr$-th persistent Laplacian}

It was proved~\cite[Theorem 2.7]{persLap} that $$\dim \left(\ker\left(\Delta_\dgr^{K_i, K_j}\right)\right) = \rank \left( H_\dgr(K_i) \to H_\dgr(K_j) \right) = \beta_\dgr^{i, j},$$ where $\beta_\dgr^{i, j}$ is the $\dgr$-th persistent Betti number for the segment $(i,j)$, see~\cref{defn: persistent betti rank inv}. Thus, the kernel of the persistent Laplacian $\Delta_\dgr^{K_i, K_j}$ provides canonical representatives of the cycle classes that persist through the inclusion $K_i \hookrightarrow K_j$. Hence, we now introduce the function that records the kernel of the $\dgr$-th persistent Laplacian for every segment.

\begin{definition}[Laplacian kernel]\label{defn: laplacian kernel}
    Let $\Ffunc = \{ K_{i}\}_{i=1}^n $ be a filtration. For any degree $\dgr\geq 0$, the $\dgr$-th~\emph{Laplacian kernel} of $\Ffunc$ is defined as the function $\LK_\dgr^\Ffunc :\overline{\lp}^\supseteq \to \gr (C_\dgr^K)$ given by

    \begin{align*}
        \LK_\dgr^\Ffunc ((\ell_i, \ell_j)) &:= \ker \left(\Delta_\dgr^{K_i, K_{j-1}}\right) \text{ for } 1\leq i < j \leq n, \\
        \LK_\dgr^\Ffunc ((\ell_i, \infty)) &:= \ker \left(\Delta_\dgr^{K_i, K_{n}}\right).
    \end{align*}
\end{definition}
\nomenclature[33]{$\LK_\dgr^\Ffunc$}{$\dgr$-th Laplacian kernel of a filtration}

\begin{remark}
    Observe that there is a shift in the second coordinate when defining $\LK_\dgr^\Ffunc$. This shift is analogous to the one used when defining the rank function in~\cite[Section 7]{Patel2018} and it ensures the equality in~\cref{thm: equality of different HIs} without requiring any additional shifts.
\end{remark}

As noted in~\cite{persLap}, the kernel of the persistent Laplacian is the intersection of two subspaces, a fact which we recall in the following proposition.

\begin{proposition}[{\cite[Claim A.1]{persLap}}]\label{prop: kernel of persistent laplacians}
    Let $\{ K_i\}_{i=1}^n$ be a simplicial filtration. Then, for any degree $\dgr\geq 0$ and $1\leq i \leq j \leq n$, \[
    \ker \left(\Delta_\dgr^{K_i, K_j}\right) = \ker\left(\partial_\dgr^{K_i}\right) \cap \Ima \left(\partial_{\dgr+1}^{K_j, K_i}\right)^\perp.
    \]
\end{proposition}

We now apply $\supseteq$-Linear Orthogonal Inversion to the map $\LK_\dgr^\Ffunc : \overline{\lp}^\supseteq \to \gr (C_\dgr^K)$ to obtain the Grassmannian persistence diagram $\rhi\left( \LK_\dgr^\Ffunc\right)$.

\begin{definition}[Degree-$\dgr$ Grassmannian persistence diagram from Laplacian kernels]\label{defn: 1p gpd lk}
    Let $\Ffunc : \lp \to \subcx(K)$ be a filtration. For any $\dgr\geq 0$, the map
    \[
    \rhi \left(\LK_\dgr^\Ffunc\right) : \overline{\lp}^\supseteq \to \gr\left(C_\dgr^K\right)
    \]
    is called \emph{the degree-$\dgr$ Grassmannian persistence diagram} of $\Ffunc$ (obtained from the Laplacian kernels).
\end{definition}

Recall that, as shown in~\cite[Section 9.1]{edit}, the algebraic M\"obius inverse (with respect to reverse inclusion order) of persistent Betti numbers coincides with the algebraic M\"obius inverse (with respect to product order) of the dimensions of the birth-death spaces for every $(\ell_i, \ell_j) \in \Seg(\lp) \setminus \diag(\lp)$. The following analogous result relates the functions $\prhi \left(\ZB_\dgr^\Ffunc\right)$ and $\rhi \left(\LK_\dgr^\Ffunc\right)$.

\begin{theorem}\label{thm: equality of different HIs}
    Let $\lp = \{ \ell_1<\cdots<\ell_n \}$. Let $\Ffunc = \{ K_i \}_{i=1}^n$ be a filtration over $\lp$. Then, for any degree $\dgr\geq 0$ and for every segment $(\ell_i, \ell_j) \in \Seg(\lp) \setminus \diag (\lp)$, we have
    \[
    \rhi\left(\LK_\dgr^\Ffunc\right) ((\ell_i,\ell_j)) = \prhi \left(\ZB_\dgr^\Ffunc\right)((\ell_i,\ell_j)).
    \]
\end{theorem}

We will need the following Lemma in order to prove~\cref{thm: equality of different HIs}

\begin{lemma}\label{lem: cancellation}
    Let $C \subseteq B \subseteq A \subseteq V$ be subspaces of an inner product space $V$. Then, $(A \ominus C) \ominus (A \ominus B) = B \ominus C$.
\end{lemma}

\begin{proof}
    Let $\mathcal{B}_C := \{ c_1,\ldots,c_k \}$ be a basis for $C$, $\mathcal{B}_{B\ominus C} :=\{ b_1,\ldots,b_l \}$ be a basis for $B \ominus C$ and $\mathcal{B}_{A \ominus B} := \{a_1,\ldots,a_m \}$ be a basis for $A \ominus B$. Then, $\{a_1,\ldots,a_m,b_1,\ldots,b_l \}$ is a basis for $A \ominus C$. Thus, $\{ b_1,\ldots,b_l\}$ is a basis for $(A\ominus C)\ominus (A \ominus B)$. On the other hand, $\{ b_1,\ldots,b_l\}$ is also a basis for $B \ominus C$. Thus, $(A \ominus C) \ominus (A \ominus B) = B \ominus C$.
\end{proof}

\begin{proof}[Proof of~\cref{thm: equality of different HIs}]
    For notational simplicity, we will write
    \begin{itemize}
        \item $(i,j):= (\ell_i, \ell_j)$,
        \item $\Delta_\dgr^{i,j} := \Delta_\dgr^{K_i, K_j}$,
        \item $\partial_\dgr^i := \partial_\dgr^{K_i}$ and $\partial_\dgr^{j,i}:= \partial_\dgr^{K_j, K_i}$,
        \item $\Zfunc_\dgr (i) := \Zfunc_\dgr (K_i)$ and $\Bfunc_\dgr(i) := \Bfunc_\dgr (K_i)$.
        
    \end{itemize}

    Recall that by~\cref{prop: kernel of persistent laplacians}, $\LK_\dgr^\Ffunc((i,j)) = \ker\left(\Delta_\dgr^{i,j-1}\right) = \ker\left(\partial_\dgr^i\right)\cap \Ima\left(\partial_{\dgr+1}^{j-1,i}\right)^\perp = \ker\left(\partial_\dgr^i\right) \ominus \Ima\left(\partial_{\dgr+1}^{j-1,i}\right)$. Observe that $\Ima\left(\partial_{\dgr+1}^{j-1,i}\right) = \ker\left(\partial_\dgr^i\right)\cap \Ima\left(\partial_{\dgr+1}^{j-1}\right)$. Thus, we can write $$\LK_\dgr^\Ffunc((i,j)) = Z_\dgr(i) \ominus \left(Z_\dgr(i)\cap B_\dgr(j-1)\right).$$ Then, by~\cref{lem: cancellation}
    \begin{align*}
        \LK_\dgr^\Ffunc((i,j)) \ominus \LK_\dgr((i,j+1))  &= \big(\Zfunc_\dgr(i) \ominus (\Zfunc_\dgr(i)\cap \Bfunc_\dgr(j-1))\big) \ominus \big(\Zfunc_\dgr(i) \ominus (\Zfunc_\dgr(i)\cap \Bfunc_\dgr(j)) \big)\\
                                                             &= \big(\Zfunc_\dgr(i)\cap \Bfunc_\dgr(j)\big) \ominus \big(\Zfunc_\dgr(i)\cap \Bfunc_\dgr(j-1)\big) \\
                                                             &= \ZB_\dgr^\Ffunc((i,j)) \ominus \ZB_\dgr^\Ffunc((i,j-1)).
    \end{align*}
    Similarly, we have that 
    \[
    \LK_\dgr^\Ffunc((i-1,j)) \ominus \LK_\dgr((i-1,j+1)) = \ZB_\dgr^\Ffunc((i-1,j)) \ominus \ZB_\dgr^\Ffunc((i-1,j-1)).
    \]
    Then, 
    \begin{align*}
        \rhi \left(\LK_\dgr^\Ffunc\right)((i,j)) =& \big ( \LK_\dgr^\Ffunc((i,j)) \ominus \LK_\dgr^\Ffunc((i,j+1)) \big ) \ominus \big ( \LK_\dgr^\Ffunc ((i-1,j)) \ominus \LK_\dgr^\Ffunc((i-1,j+1)) \big ) \\
                                              =& \big(\ZB_\dgr^\Ffunc((i,j)) \ominus \ZB_\dgr^\Ffunc((i,j-1))\big) \ominus \big(\ZB_\dgr^\Ffunc((i-1,j)) \ominus \ZB_\dgr^\Ffunc((i-1,j-1))\big) \\
                                              =& \prhi \left(\ZB_\dgr^\Ffunc\right)((i,j)).
    \end{align*}
    For the boundary cases when $i=1$ or $j=\infty$, similar arguments show that we obtain the desired result.
\end{proof}

Now, we can regard $\rhi \left(\LK_\dgr^\Ffunc\right)$ as an object in $\inndgm(C_\dgr^K)$ by extending its domain from $\overline{\lp}^\supseteq$ to $\overline{\lp}^\times$ by defining $\rhi \left(\LK_\dgr^\Ffunc \right)((\ell_i,\ell_i)) = \{ 0\}$ for every diagonal segment $(\ell_i, \ell_i) \in \diag(\lp)$. Then, we have that $\prhi \left(\ZB_\dgr^\Ffunc\right)$ and $\rhi \left(\LK_\dgr^\Ffunc\right)$ are two objects in $\inndgm(C_\dgr^K)$ that only differ along the diagonal. Therefore, as explained in~\cref{rem: diagonal blind edit}, we get that $$d_{\inndgm\left(C_\dgr^K\right)}^E \left(\prhi \left(\ZB_\dgr^\Ffunc\right), \rhi \left(\LK_\dgr^\Ffunc\right)\right) = 0.$$ Therefore, we have the following stability result.

\begin{theorem}[Stability]\label{thm: stability of rhi lap}
    Let $\Ffunc, \Gfunc \in \Fil(K)$ be two filtrations. Then, for any degree $\dgr \geq  0$, we have
    \[
    d_{\inndgm\left(C_\dgr^K\right)}^E \left(\rhi \left(\LK_\dgr^\Ffunc\right), \rhi \left(\LK_\dgr^\Gfunc\right) \right) \leq d_{\Fil(K)}^E (\Ffunc, \Gfunc).
    \]
\end{theorem}

\begin{proof}
    By the triangle inequality, we have that 
    \[
    d_{\inndgm\left(C_\dgr^K\right)}^E \left(\rhi \left(\LK_\dgr^\Ffunc\right), \rhi \left(\LK_\dgr^\Gfunc\right)\right) = d_{\inndgm\left(C_\dgr^K\right)}^E \left(\prhi \left(\ZB_\dgr^\Ffunc\right), \prhi \left(\ZB_\dgr^\Gfunc\right)\right),
    \]
    as $d_{\inndgm\left(C_\dgr^K\right)}^E \left(\prhi \left(\ZB_\dgr^\Ffunc\right), \rhi \left(\LK_\dgr^\Ffunc\right)\right) = 0 = d_{\inndgm\left(C_\dgr^K\right)}^E \left(\prhi \left(\ZB_\dgr^\Gfunc\right), \rhi \left(\LK_\dgr^\Gfunc\right)\right)$.
    Then, by~\cref{thm: stability of prhi of zb}, we conclude that 
    \[
    d_{\inndgm\left(C_\dgr^K\right)}^E \left(\rhi \left(\LK_\dgr^\Ffunc\right), \rhi \left(\LK_\dgr^\Gfunc\right)\right) \leq d_{\Fil(K)}^E (\Ffunc, \Gfunc).
    \]
\end{proof}

\subsection{Treegrams}\label{subsec: treegrams}
The degree-$0$ persistence diagram of a filtration is incapable of tracking the evolution of the clustering structure throughout the filtration (see the two filtrations depicted in~\cref{fig:degree0example} have the same degree-$0$ persistence diagrams but the hierarchical clustering structures are different). In the case of Vietoris-Rips filtration of a finite metric space, the clustering structure is captured by the notion of \emph{dendrograms}, which represents a hierarchy of clusters. In a more general filtration, \emph{treegrams}, a generalization of dendrograms, can be used to represent the clustering structure of the filtration. In this subsection, we show that the degree-$0$ Grassmannian persistence diagram of a filtration is equivalent to the treegram of the filtration; see~\cref{thm: equivalence of treegrams and degree 0 orthogonal inversions}. Namely, they can be obtained from each other. This equivalence also shows that Grassmannian persistence diagrams are stronger than the persistence diagrams. For a more thorough discussion about dendrograms/treegrams and hierarchical clustering, see~\cite{cluster, treegrampaper}.

Given a finite set $X$, a~\emph{partition} of $X$ is any collection $\pi = \{ B_1, \ldots, B_k \}$ such that
\begin{itemize}
    \item $B_i \cap B_j = \emptyset $ for $i\neq j$.
    \item $\cup_{i=1}^k B_i = X$.
\end{itemize}

We denote the set of all partitions of $X$ by $\parti (X)$. A~\emph{sub-partition} of $X$ is a pair $(X', \pi')$ such that $X' \subseteq X$ and $\pi'$ is a partition of $X'$. We denote by $\spart (X)$ the set of all sub-partitions of $X$. For two sets $A' \subseteq A$ and $\pi = \{ B_1,\ldots,B_k \} \in \parti (A)$, the restricted partition $\pi |_{A'} := \cup _{i=1}^k (B_i \cap A')$ is a partition of $A'$. We refer to the elements $B_1 ,\ldots,B_k$ of the partition as the blocks of the partition.

\begin{figure}
    \centering
    \includegraphics[scale=20]{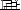}
    \caption{A graphical representation of a treegram. The slanted line segments emerging from $u$ and $v$ indicate the singletons $\{ u \}$ and $\{ v \}$ are never blocks of the treegram. The points $u$ and $v$ appear for the first time as elements of a strictly larger block.}
    \label{fig:treegram example}
\end{figure}

\begin{definition}[Treegrams {\cite{treegrampaper}}]\label{defn: treegram}
    Let $X = \{ x_1,x_2,\ldots ,x_\ell \}$ be a finite set. A~\emph{treegram} over $X$ is a function 
    \begin{align*}
        T_X : \R &\to \spart (X) \\
               t &\mapsto (X_t, \pi_t)
    \end{align*}
    such that
    \begin{enumerate}
        \item For $t\leq t'$, $X_t \subseteq X_{t'}$ and $\pi_{t'} |_{X_t}$ is coarser than $\pi_t$,
        \item $\exists t_F \in \R$ such that for all $t \geq t_F$, $X_t = X$ and $\pi_t = \{ X \}$,
        \item $\exists t_I<t_F \in \R$ such that for all $t < t_I$, $X_t = \emptyset$,
        \item\label{property: min guarantee} For all $t\in \R$, there exists $\varepsilon > 0$ such that $T_X(t) = T_X(t')$ for all $t' \in [t, t+\varepsilon]$.
    \end{enumerate}
    The parameter $t$ is referred to as \emph{time}. A treegram is called a \emph{dendrogram} if $t_I = 0$, $X_0 = X$, and $\pi_0 = \{ \{x_1\}, \{x_2\},\ldots , \{ x_\ell \}   \}$ is the finest partition of $X_0 = X$.
\end{definition}

\begin{definition}[Birth time]
    Let $T_X$ be a treegram. For $x\in X$, we define, the~\emph{birth time} of x as
    \[
    b_x := \min \{ t\in \R \mid x \in X_t \}.
    \]
    Note that the minimum exists by~\cref{property: min guarantee} in~\cref{defn: treegram}.
\end{definition}

\begin{example}
    Treegrams can be graphically represented. In~\cref{fig:treegram example}, we illustrate a treegram, $T_X$, over the set $X = \{ x,y,z,u,v \}$. For $t\in (-\infty, b_x)$ we have that $T_X(t) = \emptyset$. Also, $T_X (b_x) = \{ \{ x \} \}$, $T_X (b_z) = \{ \{ x \}, \{ z \} \}$, $T_X (b_y) = \{ \{ x\} ,\{y \} , \{ z \}\}$, $T_X (t_{xy}) = \{ \{ x ,y \}, \{z \} \}$ and $T_X (t) = \{ \{ x,y,z,u,v \} \}$ for $t \in (b_{u}, \infty)$. Notice that we use a shorthand notation here by only recording the partition component, $\pi_t$, of the sub-partition $T_X(t) = (X_t, \pi_t)$. In this example, at time $t = b_{u}=b_v$, the blocks $\{ x,y \}$ and $\{ z \}$ merge together and the points $u$ and $v$ appears for the first time and immediately merge with $x$, $y$, and $z$.
\end{example}

Let $K$ be a finite connected simplicial complex and let $\Ffunc = \{ K_i \}_{i=1}^n$ be a filtration of $K$ over a linearly ordered metric poset $P = \{ p_1,\ldots,p_n \} \subseteq \R$. Let $V(K_i)$ denote the set of vertices of $K_i$. 
Then, the filtration $\Ffunc$ determines a treegram $T_\Ffunc : \R \to \spart (V(K))$ as follows:
\begin{enumerate}
    \item For $t<p_1$, $T_\Ffunc (t) := (\emptyset, \emptyset)$,
    \item For $p_i \leq t < p_{i+1}$, $T_\Ffunc (t) := (V(K_i), \conn (K_i))$,
    \item For $t \geq p_n$, $T_\Ffunc (t) = (V(K_n), \conn (K_n)) := (V(K), \{ V(K) \})$,
\end{enumerate}
where $\conn (K_i)$ is the partition of $V(K_i)$ whose blocks consist of vertices that are in the same connected component of $K_i$.

\begin{definition}[Treegram of a filtration]
    The treegram $T_\Ffunc$ constructed from a filtration $\Ffunc$ is called the~\emph{treegram of $\Ffunc$}.
\end{definition}
\nomenclature[34]{$T_\Ffunc$}{Treegram of a filtration $\Ffunc$}

The main result in this subsection is the equivalence of treegrams and degree-$0$ Grassmannian persistence diagrams. Here, we use the term ``equivalence'' to indicate that they can be obtained from each other.

\begin{theorem}\label{thm: equivalence of treegrams and degree 0 orthogonal inversions}
    For a filtration, $\Ffunc := \{ K_i \}_{i=1}^n$, of a finite connected simplicial complex $K$, $\prhi \left(\ZB_0^\Ffunc\right)$ and the treegram $T_\Ffunc$ are equivalent.
\end{theorem}

Note that since, by~\cref{thm: prhi is monoidal inversion}, $\prhi \left(\ZB_0^\Ffunc\right)$ is a monoidal M\"obius inverse of $\ZB_0^\Ffunc$, the information we gain from $\prhi \left(\ZB_0^\Ffunc\right)$ is the same as the information we gain from the $0$-th birth-death spaces $\ZB_0^\Ffunc$. Similarly, the information gained from the treegram $T_\Ffunc$ is equivalent to the information gained from $\ZB_0^\Ffunc$. This is because, for each $i=1,\ldots, n$, $T_\Ffunc (i) = (V(K_i), \conn (K_i))$ and $V(K_i)$ determines the cycles and $\conn(K_i)$ determines the boundaries at $t=i$. So, the equivalence of $\prhi \left(ZB_0^\Ffunc\right)$ and $T_\Ffunc$ is obtained through the fact that both are equivalent to $\ZB_0^\Ffunc$. More formally, the proof of the~\cref{thm: equivalence of treegrams and degree 0 orthogonal inversions} follows from the next two propositions.

\begin{proposition}\label{prop: treegram to zb0}
    Let $\Ffunc := \{ K_i \}_{i=1}^n$ be a filtration. Then, $\prhi \left(\ZB_0^\Ffunc\right)$ can be recovered from the treegram $T_\Ffunc$.
\end{proposition}

\begin{proof}
    Let $T_\Ffunc$ be given. For any $i = 1,\ldots,n$, we have that $T_\Ffunc (i) = (V(K_i), \conn (K_i))$. Then, $V(K_i)$ determines $\Zfunc_0 (K_i)$ (as $V(K_i)$ is the canonical basis of $\Zfunc_0 (K_i)$) and $\conn(K_i)$ determines $\Bfunc_0 (K_i)$ (as two vertices $v_1,v_2 \in V(K_i)$ are in the same connected component of $K_i$ if and only if $v_1-v_2 \in \Bfunc_0 (K_i)$). Thus, the birth-death spaces $\ZB_0^\Ffunc ((i,j)) = \Zfunc_0 (K_i) \cap \Bfunc_0 (K_j)$ can be recovered for every $i$ and $j$. Hence, $\prhi \left(\ZB_0^\Ffunc\right)$ can also be recovered.
\end{proof}

\begin{remark}
    Note that the proof we provided above is nonconstructive. With the goal of having an algorithm for computing the degree-$0$ Grassmannian persistence diagram (i.e. $\prhi \left(\ZB_0^\Ffunc\right)$) from the treegram $T_\Ffunc$, we provide a constructive proof of~\cref{prop: treegram to zb0} in~\cref{appendix: construction}.
\end{remark}

\begin{proposition}
    Let $\Ffunc := \{ K_i \}_{i=1}^n$ be a filtration. Then, the treegram $T_\Ffunc$ can be recovered from $\prhi \left(\ZB_0^\Ffunc\right)$.
\end{proposition}

\begin{proof}
    By~\cref{thm: prhi is monoidal inversion}, we can recover the birth-death spaces, $\ZB_0^\Ffunc$, of the filtration $\Ffunc$ from $\prhi \left(\ZB_0^\Ffunc\right)$. In particular, we can recover $\Zfunc_0 (K_i) \cap \Bfunc_0 (K_i) = \Bfunc_0 (K_i)$ for all $i$. Then, the connected components of $K_i$, $\conn (K_i)$, can be reconstructed from  $ \Bfunc_0 (K_i)$. Observe that $V(K_i) = \cup_{B \in \conn(K_i)} B$ because $\conn(K_i)$ is a partition of $V(K_i)$. Hence, $V(K_i)$ is also recovered. Thus, the treegram $T_\Ffunc$, which is defined by the collections $\{ V(K_i)\}_{i=1}^n$ and $\{ \conn (K_i) \}_{i=1}^n$, can also be recovered.
\end{proof}

\begin{remark}
    Let $(X, u_X)$ be a finite ultrametric space. That is, $u_X : X\times X \to \R_{\geq 0}$ is a metric and $u_X$ satisfies the ultrametric inequality: $u_X (x,z) \leq \max \{ u_X(x,y), u_X(y,z) \}$ for all $x,y,z \in X$. As discussed in~\cite[Section 3.3]{cluster}, the dendrograms over $X$ and the ultrametrics on $X$ are equivalent, i.e. there is a bijection between the set of all ultrametrics on $X$ and the set of all dendrograms over $X$ such that the corresponding ultrametrics and dendrograms generates the same hierarchical decomposition~\cite[Theorem 9]{cluster}. Let $D_{u_X} : \R \to \parti(X)$ be the dendrogram corresponding to $u_X$ that is determined by this equivalence. Let $\VR(X,u_X)$ be the Vietoris-Rips filtration of the metric space $(X,u_X)$ and let $T_{\VR(X,u_X)}$ be the treegram of $F$. One can see that $T_{\VR(X,u_X)}$ is indeed a dendrogram. Indeed, $T_{\VR(X,u_X)} = D_{u_X}$. Note that, by~\cref{thm: equivalence of treegrams and degree 0 orthogonal inversions}, the degree-$0$ Grassmannian persistence diagram of the filtration $\VR(X,u_X)$ is equivalent to $T_{\VR(X,u_X)} = D_{u_X}$. 
    Hence, by combining these facts: $T_{\VR(X,u_X)} = D_{u_X}$, \cite[Theorem 9]{cluster} and \cref{thm: equivalence of treegrams and degree 0 orthogonal inversions}, we conclude that the degree-$0$ Grassmannian persistence diagram of the Vietoris-Rips filtration of a finite ultrametric space $(X,u_X)$ recovers the ultrametric $u_X$. This also highlights the superior discriminating power of Grassmannian persistence diagrams compared to classical persistence diagrams.
\end{remark}

The key insight from the previous remark is summarized in the following Corollary.

\begin{corollary}\label{cor: finite ums recovered}
    $\VR(X,u_X)$ be a finite ultrametric space. Then, degree-$0$ Grassmannian persistence diagram of the Vietoris-Rips filtration of $(X,u_X)$ recovers the ultrametric $u_X$.
\end{corollary}

\section{Discussion}\label{sec:disc}

When comparing two filtrations and their Grassmannian persistence diagrams, we are required that there is a fixed simplicial complex $K$ that each filtration eventually stabilizes at. It is a natural question to ask for a framework that can handle filtrations over different vertex sets. Moreover, while the motivation behind the concept of Orthogonal Inversions primarily stems from its applications in TDA, there is an inherent interest in broadening the utility of orthogonal inversions beyond the scope of TDA.

The equivalence of treegrams and degree-$0$ $\times$-Linear Orthogonal Inverses of birth-death spaces (\cref{subsec: treegrams}) suggests that for dimensions $\dgr\geq 1$, $\prhi \left(\ZB_\dgr^\Ffunc\right)$ can be thought of as a higher dimensional generalization of treegrams. This raises the question of whether there is a useful graphical description of $\prhi \left(\ZB_\dgr^\Ffunc\right)$ in that case.  In a similar vein, one wonders whether multidimensional dendrograms, that is functors $D:P\to \mathrm{Part}(X)$, where $\mathrm{Part}(X)$ is the set of partitions of a finite set $X$ \cite{dynamic-kim2023}, might be perfectly encoded (up to isomorphism) by Grassmanian persistence diagrams. In relation to this thread, we expect that the equivalence of Grassmannian persistence diagrams and treegrams can be extended to \emph{merge trees} under suitable assumptions. More generally, it is interesting to explore to what extent Reeb graphs, in the cosheaf representation of \cite{de2016categorified}, can be recovered from appropriately induced Grassmannian persistence diagrams.

We studied the Orthogonal Inversion of two different combinations of invariants and partial orders. Namely, birth-death spaces with the product order and persistent Laplacians with the reverse inclusion order. We expect to see Orthogonal Inversions of other combinations of invariants and partial orders will lead to interesting constructions.

Finally, we demonstrated that the $\supseteq$-Linear Orthogonal Inverse of $0$-eigenspace (i.e. kernel) of persistent Laplacians boils down to $\times$-Linear Orthogonal Inverse of birth-death spaces. However, both nonzero eigenvalues and the corresponding eigenspaces of the Laplacian have applications in general, such as partitioning~\cite{specGraphTh, ng2002sca, uvon, LeeGT12} and shape matching~\cite{Reuter_2005, articulated}. This suggests further investigation of Orthogonal Inversion(s) of other eigenspaces of the persistent Laplacian.

\bibliographystyle{alpha}
 \bibliography{ref}{}
% \addcontentsline{toc}{secti\addcontentsline{toc}{section}{References}on}{References}

\appendix

\section{Grothendieck Group Completion}\label{appendix:details}

Let $(\mathcal{M}, + , 0)$ be a commutative monoid. Consider the equivalence relation $\sim$ defined on $\mathcal{M}\times \mathcal{M}$ given by
\[
(m_1, n_1) \sim (m_2, n_2) \iff \text{ there exists } k\in \mathcal{M} \text{ such that } m_1 + n_2 + k = m_2 + n_1 +k.
\]
We denote by $[(m_1,n_1)]$ the equivalence class containing $(m_1, n_1)$. Let $\kappa (\mathcal{M}) := \mathcal{M} \times \mathcal{M} / \sim$ be the set of equivalence classes of $\sim$. $\kappa(\mathcal{M})$ inherits the binary operation of $\mathcal{M}$
\[
+ : \kappa(\mathcal{M}) \times \kappa (\mathcal{M}) \to \kappa(\mathcal{M})
\]
by applying it component-wisely  
\begin{align*}
        [(m_1, n_1)] + [(m_2, n_2)] := [(m_1+m_2, n_1 +n_2)].
\end{align*}
The tuple $(\kappa(\mathcal{M}), +, [(0,0)])$ determines an abelian group, called 
the~\emph{Grothendieck group completion of $\mathcal{M}$}. Observe that there is a canonical morphism 
\begin{align*}
    \varphi_\mathcal{M} : \mathcal{M} &\to \kappa (\mathcal{M}) \\
    m           &\mapsto [(m,0)].
\end{align*}
\nomenclature[35]{$\kappa(\cdot)$}{Grothendieck group completion of a commutative monoid}

\begin{definition}[Absorbing element]
    An element $\infty_\mathcal{M} \in \mathcal{M}$ is called an~\emph{absorbing element} if $m+\infty_\mathcal{M} = \infty_\mathcal{M}$ for every $m\in \mathcal{M}$.
\end{definition}

\begin{proposition}\label{prop: absorbing implies trivial}
    Let $\mathcal{M}$ be a commutative monoid with an absorbing element $\infty_\mathcal{M}$. Then, the Grothendieck group completion of $\mathcal{M}$ is the trivial group.
\end{proposition}

\begin{proof}
    Let $(m_1, n_1), (m_2, n_2) \in \mathcal{M} \times \mathcal{M}$. Observe that $(m_1, n_1) \sim (m_2, n_2)$ because
    \[
    m_1 + n_2 + \infty_\mathcal{M} = \infty_\mathcal{M} = m_2 + n_1 + \infty_\mathcal{M}.
    \]
    As $(m_1, n_1), (m_2, n_2) \in \mathcal{M} \times \mathcal{M}$ were arbitrary, we conclude that there is only one equivalence class. Namely, $\kappa(\mathcal{M}) = \{ [(0,0)] \}$.
\end{proof}

\begin{corollary}
    The Grothendieck group completion of $\gr(V)$ is the trivial group.
\end{corollary}

\begin{proof}
    $V \in \gr(V)$ is an absorbing element. Thus, the result follows from~\cref{prop: absorbing implies trivial}.
\end{proof}

\section{Details from~\texorpdfstring{\cref{sec: orthogonal inversion}}{}}\label{appendix: details for orthogonal inversion}

In this section, we present the missing details and proofs from \cref{sec: orthogonal inversion}.

\begin{lemma}\label{lem: mobeq respects pushforward}
    Let $\ladj{f} : P \leftrightarrows Q : \radj{f}$ be a Galois connection between two finite posets $P$ and $Q$. Let $\alpha \mobeq \beta : P \to \mathcal{M}$ be two M\"obius equivalent functions from $P$ to a commutative monoid $M$. Then, $$(\ladj{f})_\sharp \alpha \mobeq (\ladj{f})_\sharp \beta.$$
\end{lemma}

\begin{proof}
    Let $q\in Q$. Then, 
    \begin{align*}
        \sum_{q'\leq q} (\ladj{f})_\sharp \alpha (q') &= \sum_{q'\leq q} \sum_{\substack{p \in P \\ \ladj{f}(p)=q'}} \alpha(p) \\
        &= \sum_{\substack{p\in P \\ \ladj{f}(p)\leq q}} \alpha (p)  = \sum_{\substack{p\in P \\ p \leq \radj{f}(q)}} \alpha (p)
    \end{align*}
    Similarly, we have that 
    \[
    \sum_{q'\leq q} (\ladj{f})_\sharp \beta (q') = \sum_{\substack{p\in P \\ p \leq \radj{f}(q)}} \beta (p)
    \]
    By our assumption that $\alpha \mobeq \beta$, we have that
    \[
    \sum_{\substack{p\in P \\ p \leq \radj{f}(q)}} \alpha (p) = \sum_{\substack{p\in P \\ p \leq \radj{f}(q)}} \beta (p).
    \]
    Thus, it follows that 
    \[
    \sum_{q'\leq q} (\ladj{f})_\sharp \alpha (q') = \sum_{\substack{p\in P \\ p \leq \radj{f}(q)}} \alpha (p) = \sum_{\substack{p\in P \\ p \leq \radj{f}(q)}} \beta (p) = \sum_{q'\leq q} (\ladj{f})_\sharp \beta (q').
    \]
    As $q\in Q$ was arbitrary, we conclude that $(\ladj{f})_\sharp \alpha \mobeq (\ladj{f})_\sharp \beta$.
\end{proof}

\begin{lemma}\label{lem: subfamilies transversal}
    Assume that two families $\{W_i \}_{i\in \mathcal{I}}$ and $\{ U_j \}_{j \in \mathcal{J}}$ are transversal to each other where $\mathcal{I}$ and $\mathcal{J}$ are finite sets. Then, for any $\mathcal{K} \subseteq \mathcal{I}$ and $\mathcal{L} \subseteq \mathcal{J}$, the subfamilies $\{W_k \}_{k\in \mathcal{K}}$ and $\{U_\ell \}_{\ell\in \mathcal{L}}$ are also transversal to each other.
\end{lemma}

\begin{proof}
    Assume that there are two subfamilies $\{W_k \}_{k\in \mathcal{K}}$ and $\{U_\ell \}_{\ell\in \mathcal{L}}$ that are not transversal to each other. That is, 
    \[
    \dim \left ( \sum_{k\in \mathcal{K}} W_k + \sum_{\ell \in \mathcal{L}} U_\ell \right ) < \sum_{k\in \mathcal{K}} \dim (W_k) + \sum_{\ell \in \mathcal{L}} \dim (U_\ell).
    \]
    Then, it follows that 
    \begin{align*}
        \dim \left (\sum_{i\in \mathcal{I}}W_i + \sum_{j\in \mathcal{J}} U_j \right ) &= \dim \left ( \sum_{k\in \mathcal{K}}W_k + \sum_{i\in \mathcal{I}\setminus \mathcal{K}}W_i + \sum_{\ell\in \mathcal{L}} U_\ell + \sum_{j \in \mathcal{J}\setminus \mathcal{L}} U_j \right ) \\
        &\leq \dim \left ( \sum_{k\in \mathcal{K}}W_k + \sum_{\ell\in \mathcal{L}} U_\ell \right )  + \dim \left ( \sum_{i\in \mathcal{I}\setminus \mathcal{K}}W_i + \sum_{j \in \mathcal{J}\setminus \mathcal{L}} U_j \right ) \\
        &< \sum_{k\in \mathcal{K}} \dim (W_k) + \sum_{\ell\in \mathcal{L}} \dim (U_\ell) + \sum_{i\in \mathcal{I} \setminus \mathcal{K}} \dim (W_i) + \sum_{j\in \mathcal{J} \setminus \mathcal{L}} \dim (U_\ell) \\
        &= \sum_{i\in \mathcal{I}} \dim(W_i) + \sum_{j \in \mathcal{J}} \dim (U_j).
    \end{align*}
    Therefore, $\{ W_i \}_{i\in \mathcal{I}}$ and $\{U_j \}_{j\in \mathcal{J}}$ are not transversal to each other.
\end{proof}

\begin{corollary}\label{cor: subfamily transversal}
    Let $\{W_i \}_{i\in \mathcal{I}}$ be a transverse family. Then, for any $\mathcal{J} \subseteq \mathcal{I}$, the subfamily $\{W_j \}_{j\in \mathcal{J}}$ is also a transverse family.    
\end{corollary}

\begin{proof}
    Apply~\cref{lem: subfamilies transversal} to $\{W_i \}_{i\in \mathcal{I}}$ and $\{ U := \{ 0 \}\}$
\end{proof}

\begin{lemma}\label{lem: sum dim push}
    Let $\mathcal{I}$ be a finite set and $\Mfunc : \mathcal{I} \to \gr(V)$ be any function such that $\{ \Mfunc(i) \}_{i\in \mathcal{I}}$ is a transverse family. Let $\mathcal{J}$ be any finite set and $h : \mathcal{I} \to \mathcal{J}$ be any function. Then, 
    \[
    \sum_{j \in \mathcal{J}} \dim (h_\sharp \Mfunc (j) ) = \sum_{i \in \mathcal{I}} \dim (\Mfunc (i)).
    \]
    In particular, the family $\{ h_\sharp \Mfunc (i) \}_{j \in J}$ is a transverse family.
\end{lemma}

\begin{proof}
    The claim follows from the following calculation:
    \begin{align*}
        \sum_{j\in \mathcal{J}} \dim (h_\sharp \Mfunc (j)) =& \sum_{j\in \mathcal{J}} \dim \left ( \sum_{i\in h^{-1}(j)} \Mfunc(i) \right ) && \\
        &= \sum_{j\in \mathcal{J}} \sum_{i\in h^{-1}(j)} \dim (\Mfunc(i)) && \text{by~\cref{cor: subfamily transversal}} \\
        &= \sum_{i\in \mathcal{I}} \dim (\Mfunc (i)) && \\
        &= \dim \left ( \sum_{i \in \mathcal{I}} \Mfunc(i)\right ) = \dim \left ( \sum_{j \in \mathcal{J}} h_\sharp \Mfunc(j) \right ).
    \end{align*}
\end{proof}

\begin{proof}[Proof of \cref{prop: comp of transversity-pres.}]
     For $i=1,2,3$, let $\Mfunc_i : \overline{\lp}_i^\times \to \gr(V)$ be Grassmannian persistence diagrams. Let $(f, \zeta_{\lp_2})$ be a morphism from $\Mfunc_1$ to $\Mfunc_2$ and let $(g, \zeta_{\lp_3})$ be a morphism from $\Mfunc_2$ to $\Mfunc_3$. Thus, we have that $\left(\overline{\ladj{f}}\right)_\sharp \Mfunc_1 \mobeq (\Mfunc_2 + \zeta_{\lp_2})$ and $\left(\overline{\ladj{g}}\right)_\sharp \Mfunc_2 \mobeq (\Mfunc_3 + \zeta_{\lp_3})$. Then, it follows that
    \begin{align*}
        \left(\overline{\ladj{(g \circ f)}}\right)_\sharp \Mfunc_1 &= \left(\overline{\ladj{g}\circ \ladj{f}}\right)_\sharp \Mfunc_1 \\
                                                        &= \left(\overline{\ladj{g}} \circ \overline{\ladj{f}}\right)_\sharp \Mfunc_1 \\
                                                        &= \left(\overline{\ladj{g}}\right)_\sharp \left(\left(\overline{\ladj{f}}\right)_\sharp \Mfunc_1\right) \\
                                                        &\mobeq \left(\overline{\ladj{g}}\right)_\sharp \left(\Mfunc_2 + \zeta_{\lp_2}\right) && \text{by~\cref{lem: mobeq respects pushforward}}\\
                                                        &= \left(\overline{\ladj{g}}\right)_\sharp \left(\Mfunc_2\right) + \left(\overline{\ladj{g}}\right)_\sharp \left(\zeta_{\lp_2}\right) \\
                                                        &\mobeq \Mfunc_3 + \zeta_{\lp_3} + \left(\overline{\ladj{g}}\right)_\sharp \left(\zeta_{\lp_2}\right).
    \end{align*}
    Observe that $\zeta_{{\lp_3}'}:= \zeta_{\lp_3} + \left(\overline{\ladj{g}}\right)_\sharp (\zeta_{\lp_2})$ is supported on $\diag(\lp_3)$ and the families $\left \{ \zeta_{\lp_3} (J) \right \}_{J \in \overline{\lp}_3^\times}$ and $\{ (\overline{\ladj{g}})_\sharp (\zeta_{\lp_2} (J)) \}_{J \in \overline{\lp}_3^\times}$ are transversal to each other. The latter can be seen from the following argument. The families $\{ \zeta_{\lp_2} (I) \}_{I \in \overline{\lp}_2^\times}$ and $\{ \Mfunc_2 (I) \}_{I \in \overline{\lp}_2^\times}$ are transversal to each other. Thus, $\{ (\overline{\ladj{g}})_\sharp (\zeta_{\lp_2} (J)) \}_{J \in \overline{\lp}_3^\times}$ and $\{ \Mfunc_2 (I) \}_{I \in \overline{\lp}_2^\times}$ are transversal to each other. Moreover, we have that $\sum_{I \in \overline{\lp}_2^\times} \Mfunc_2 (I) = \sum_{J \in \overline{\lp}_3^\times} \Mfunc_3 (J) + \zeta_{\lp_3} (J)$ and $\{ \Mfunc_3 (J) \}_{J \in \overline{\lp}_3^\times} \cup \{ \zeta_{\lp_3} (J) \}_{J \in \overline{\lp}_3^\times}$ is a transversal family. Hence, $\{ (\overline{\ladj{g}})_\sharp (\zeta_{\lp_2} (J)) \}_{J \in \overline{\lp}_3^\times}$ and $\{ \Mfunc_3 (J) \}_{J \in \overline{\lp}_3^\times} \cup \{ \zeta_{\lp_3} (J) \}_{J \in \overline{\lp}_3^\times}$ are transversal to each other. Thus, $\{ (\overline{\ladj{g}})_\sharp (\zeta_{\lp_2} (J)) \}_{J \in \overline{\lp}_3^\times}$ and $\{ \zeta_{\lp_3} (J) \}_{J \in \overline{\lp}_3^\times}$ are transversal to each other.
    
    It remains to show that $\{ \Mfunc_3 (J)\}_{J \in \overline{\lp}_3^\times} $ and $ \{ \zeta_{{\lp_3}'} (J)\}_{J \in \overline{\lp}_3^\times}$ are transversal to each other. That is, we need to show that
    \[
    \dim \left ( \sum_{J \in \overline{\lp}_3^\times} \Mfunc_3 (J) + \sum_{J\in \overline{\lp}_3^\times} \zeta_{{\lp_3}'} (J) \right ) = \sum_{J \in \overline{\lp}_3^\times} \dim (\Mfunc_3 (J)) + \sum_{J \in \overline{\lp}_3^\times} \dim (\zeta_{{\lp_3}'} (J)).
    \]
    We have that LHS is equal to
    \begin{align*}
        &=\dim \left ( \sum_{J \in \overline{\lp}_3^\times} \Mfunc_3 (J) + \sum_{J\in \overline{\lp}_3^\times} \zeta_{{\lp_3}'} (J) \right ) &&\\
        &= \dim \left ( \sum_{J \in \overline{\lp}_3^\times} \Mfunc_3 (J) + \sum_{J\in \overline{\lp}_3^\times} \zeta_{{\lp_3}} (J) + \sum_{J\in \overline{\lp}_3^\times}(\overline{\ladj{g}})_\sharp (\zeta_{\lp_2})  (J) \right ) &&\\
        &= \dim \left (  \sum_{J \in \overline{\lp}_3^\times} \big ( \Mfunc_3 (J) +  \zeta_{{\lp_3}} (J) \big ) +  \sum_{J\in \overline{\lp}_3^\times}(\overline{\ladj{g}})_\sharp (\zeta_{\lp_2})  (J)  \right ) &&\\
        &= \dim \left (  \sum_{J \in \overline{\lp}_3^\times} (\overline{\ladj{g}})_\sharp (\Mfunc_2) (J) + \sum_{J\in \overline{\lp}_3^\times}(\overline{\ladj{g}})_\sharp (\zeta_{\lp_2})  (J)        \right ) && \text{as $(\overline{\ladj{g}})_\sharp \Mfunc_2 \mobeq (\Mfunc_3 + \zeta_{\lp_3})$ } \\
        &= \dim \left ( \sum_{I \in \overline{\lp}_2^\times} \Mfunc_2 (I) + \sum_{I\in \overline{\lp}_2^\times} \zeta_{\lp_2} (I) \right ) &&\\
        &= \sum_{I \in \overline{\lp}_2^\times} \dim(\Mfunc_2 (I)) + \sum_{I\in \overline{\lp}_2^\times} \dim (\zeta_{\lp_2}(I)) &&\\
        &= \dim \left ( \sum_{I\in \overline{\lp}_2^\times} \Mfunc_2 (I) \right ) + \sum_{J \in \overline{\lp}_3^\times} \dim ((\overline{\ladj{g}})_\sharp (\zeta_{\lp_2}) (J)) && \text{by~\cref{lem: sum dim push}}\\
        &= \dim \left ( \sum_{J\in \overline{\lp}_3^\times} \Mfunc_3 (J) + \sum_{J \in \overline{\lp}_3^\times} \zeta_{\lp_3} (J) \right ) + \sum_{J \in \overline{\lp}_3^\times} \dim ((\overline{\ladj{g}})_\sharp (\zeta_{\lp_2}) (J)) &&\\
        &= \sum_{J\in \overline{\lp}_3^\times} \dim (\Mfunc_3 (J)) + \sum_{J \in \overline{\lp}_3^\times} \dim (\zeta_{\lp_3} (J) ) + \sum_{J \in \overline{\lp}_3^\times} \dim ((\overline{\ladj{g}})_\sharp (\zeta_{\lp_2}) (J)) &&\\
        &= \sum_{J \in \overline{\lp}_3^\times} \dim (\Mfunc_3 (J)) + \sum_{J \in \overline{\lp}_3^\times} \dim (\zeta_{{\lp_3}'} (J)).&&
    \end{align*}
    The last equality follows from~\cref{lem: subfamilies transversal} as we have that the families $\{ \zeta_{\lp_3} (J) \}_{J \in \overline{\lp}_3^\times}$ and $\{ (\overline{\ladj{g}})_\sharp (\zeta_{\lp_2} (J) \}_{J \in \overline{\lp}_3^\times}$ are transversal to each other, therefore,
    \[
    \dim \left(\zeta_{\lp_3} (J)\right) + \dim \left((\overline{\ladj{g}})_\sharp (\zeta_{\lp_2}) (J)\right) = \dim (\zeta_{\lp_3} (J) + (\overline{\ladj{g}})_\sharp (\zeta_{\lp_2}) (J)) = \dim (\zeta_{{\lp_3}'} (J))
    \]
    for all $J \in \overline{\lp}_3^\times$.
\end{proof}

\begin{lemma}\label{lemma: minusProjection}
    If $W_1, W_2 \subseteq V$, are subspaces of an inner product space $V$, then $$W_1 \ominus W_2 = W_1 \ominus \proj_{W_1}(W_2).$$
\end{lemma}

\begin{proof}
    Let $u \in W_1 \cap W_2^\perp $ and $ w_2 \in W_2 $. Then,  
    \begin{align*}
        0 =& \langle u , w_2 \rangle \\
          =& \langle u , \proj_{W_1}(w_2) + (w_2 - \proj_{W_1}(w_2)) \rangle \\
          =& \langle u , \proj_{W_1}(w_2) \rangle + \langle u , (w_2 - \proj_{W_1}(w_2)) \rangle \\
          =& \langle u , \proj_{W_1}(w_2) \rangle + 0.
    \end{align*}
    Thus, $\langle u , \proj_{W_1}(w_2) \rangle = 0$. Therefore, $u \in W_1 \cap (\proj_{W_1}(W_2))^\perp = W_1 \ominus \proj_{W_1}(W_2).$ Let $s \in W_1 \cap (\proj_{W_1}(W_2))^\perp$ and let $w_2 \in W_2$. Then, 
    \begin{align*}
        \langle s, w_2 \rangle =& \langle s , \proj_{W_1}(w_2) + (w_2 - \proj_{W_1}(w_2)) \rangle \\
                               =& \langle s , \proj_{W_1}(w_2) \rangle + \langle s , (w_2 - \proj_{W_1}(w_2)) \rangle \\
                               =& 0 + 0 =0
    \end{align*}
    Thus, $s\in W_1\cap W_2^\perp = W_1 \ominus W_2$. Therefore, $W_1 \ominus W_2 = W_1 \ominus \proj_{W_1}(W_2)$.
\end{proof}

\begin{lemma}\label{lem: b perp c perp}
    Let $V$ be a finite-dimensional inner product space. Let $B,C \subseteq V$ be two linear subspaces. Then, 
    \[
    B^\perp \cap (C\cap (B\cap C)^\perp)^\perp = B^\perp \cap C^\perp.
    \]
\end{lemma}

\begin{proof}
    The containment $B^\perp \cap (C\cap (B\cap C)^\perp)^\perp \supseteq B^\perp \cap C^\perp$ is clear because $(C\cap (B\cap C)^\perp)^\perp \supseteq C^\perp$ as $C\cap (B\cap C)^\perp \subseteq C$. To see the other containment  $ B^\perp \cap (C\cap (B\cap C)^\perp)^\perp \subseteq B^\perp \cap C^\perp$, let $x \in B^\perp \cap (C\cap (B\cap C)^\perp)^\perp$. Then, $x \in B^\perp$ and $x \in (C\cap (B\cap C)^\perp)^\perp = C^\perp + (B\cap C)$. Thus, we can write $x = w + y$ where $w \in C^\perp$ and $y\in B\cap C$. Then, we have
    \begin{align*}
        0 &= \langle x , y \rangle && \text{as $x \in B^\perp$ and $y \in B$} \\
          &= \langle w + y , y \rangle \\
          &= \langle w , y \rangle + \langle y , y \rangle \\
          &= 0 + \langle y , y \rangle && \text{as $w \in C^\perp$ and $y\in C$} \\
          &= \langle y , y \rangle .
    \end{align*}
    Thus, $y=0$ and $x = w \in C^\perp$. Therefore, $x\in B^\perp \cap C^\perp$. Hence, $B^\perp \cap (C\cap (B\cap C)^\perp)^\perp \subseteq B^\perp \cap C^\perp$.
\end{proof}

\begin{proof}[Proof of~\cref{prop: linAlg Mobius}]

Unraveling the definition of $\ominus$, we see that the desired equality is equivalent to
\[
(A\cap B^\perp) \cap (C\cap (B\cap C)^\perp)^\perp = A\cap  (B + C)^\perp.
\]

By~\cref{lem: b perp c perp}, we have that
\[
B^\perp \cap (C\cap (B\cap C)^\perp)^\perp =(B + C)^\perp
\]
as $ B^\perp \cap C^\perp = (B + C)^\perp$. Intersecting both sides with $A$ provides the desired equality.

\end{proof} 

\section{Computational Complexity of~\texorpdfstring{\cref{algo: gpd}}{}}\label{subsec: computational complexity}

In this section, we analyze the computational complexity of~\cref{algo: gpd} which computes the degree-$\dgr$ Grassmannian persistence diagram of a filtration $\Ffunc : \lp = \{ \ell_1<\cdots <\ell_m \} \to \subcx(K)$.

\begin{proposition}\label{prop: algo complexity}
    The time complexity of~\cref{algo: gpd} is
    \[
    O\left( m^2 \cdot \left( n_\dgr^K \cdot n_{\dgr-1}^K \cdot \min\left(n_\dgr^K, n_{\dgr-1}^K\right) + n_{\dgr+1}^K \cdot n_{\dgr}^K \cdot \min\left(n_{\dgr+1}^K, n_{\dgr}^K\right) +    \left(n_\dgr^K\right)^3 \right)    \right),
    \]
    where $n_\dgr^K$ is the number of $\dgr$-simplices of $K$.
\end{proposition}

The remainder of this section will be dedicated to proving~\cref{prop: algo complexity}. We begin with the following auxiliary lemmas.

\begin{lemma}\label{lem: ominus complexity}
    Let $A,B\subseteq \R^d$ be subspaces such that $B \subseteq A$. Let $ \{\overrightarrow{a}_1, \ldots, \overrightarrow{a}_k \}$ and $ \{ \overrightarrow{b}_1 \ldots, \overrightarrow{b}_\ell \}$ be bases for $A$ and $B$ respectively. Then, an orthonormal basis for $A \ominus B$ can be computed in $O\left(d^3\right)$ time.
\end{lemma}

\begin{proof}
    By applying the Gram–Schmidt process to the basis $\{\overrightarrow{b}_1, \ldots , \overrightarrow{b}_\ell \}$, we obtain an orthonormal basis $ \{ \overrightarrow{u}_1 \ldots , \overrightarrow{u}_\ell \}$ for $B$ in $O \left ( d^3 \right)$ time. Now, applying the Gram-Schmidt process to the set $\{ \overrightarrow{u}_1 ,\ldots , \overrightarrow{u}_\ell, \overrightarrow{a}_1, \ldots, \overrightarrow{a}_k \}$, we obtain an orthonormal basis $\{ \overrightarrow{u}_1\ldots , \overrightarrow{u}_\ell, \overrightarrow{v}_1, \ldots, \overrightarrow{v}_{k-t} \}$ for $A$. This is achieved in $O \left ( d^3 \right)$ time because $t, k\leq d$. Then, by construction, $\{ \overrightarrow{v}_1 , \ldots, \overrightarrow{v}_{k-t}\}$ is an orthonormal basis for $A\ominus B$ and it is computed in $O \left( d^3 \right)$ time.
\end{proof}

\begin{lemma}\label{lem: basis for sum complexity}
    Let $V,W\subseteq \R^d$ be two subspaces. Let $ \{\overrightarrow{v}_1, \ldots, \overrightarrow{v}_k \}$ and $ \{ \overrightarrow{w}_1 \ldots, \overrightarrow{w}_\ell \}$ be bases for $V$ and $W$ respectively. Then, a basis for $V+W$ can be computed in $O\left(d^3\right)$ time.
\end{lemma}

\begin{proof}
    By computing the reduced row echelon form of the following matrix
    \[
    \begin{bmatrix}
\overrightarrow{v}_1 \cdots \overrightarrow{v}_\ell \;  \; \overrightarrow{w}_1 \cdots \overrightarrow{w}_k 
\end{bmatrix}
\in \R^{d\times(\ell + k)},
    \]
    one can obtain a basis for $V+W$ by picking the pivot columns. Since $k,\ell \leq d$, this computation can be done via Guassian elimination in $O(d^3)$ time.
\end{proof}

For the rest of this section, we assume that, for every $\dgr\geq 0$, an ordering is fixed on the set of oriented $\dgr$-simplices of $K$, $\mathfrak{s}_\dgr^K$, in order to make $\mathfrak{s}_\dgr^K$ an ordered basis for $C_\dgr^K$ and the ordering on $\mathfrak{s}_\dgr^{K_i}$ is obtained by restricting the ordering on $\mathfrak{s}_\dgr^K$ onto $\mathfrak{s}_\dgr^{K_i}$. We identify $C_\dgr^K$ with $\R^{n_\dgr^K}$, where $n_\dgr^K := |\mathfrak{s}_\dgr^K| = \dim_{\R} C_\dgr^K$, and $C_\dgr^{K_i}$ with $\R^{n_\dgr^{K_i}} \subseteq \R^{n_\dgr^K}$, where $K_i := F(\ell_i)$. As the input of~\cref{algo: gpd} is the filtration $\Ffunc : \lp \to \subcx(K)$, we assume that we are given the matrix representation, denoted $\mathtt{B}_\dgr^{K_i} \in \R^{n_{\dgr-1}^K \times n_\dgr^{K_i}}$, of the boundary map
\[
\partial_\dgr^{K_i} : C_\dgr^{K_i} \to C_{\dgr-1}^{K_i} \subseteq C_{\dgr-1}^K
\]
with respect to the ordered bases $\mathfrak{s}_\dgr^{K_i}$ and $\mathfrak{s}_{\dgr-1}^K$ for every degree $\dgr\geq 0$ and for every $i=1,\ldots,m$.

\begin{lemma}\label{lem: zb time complexity}
    For every $\dgr\geq 0$ and every segment $(\ell_i, \ell_j) \in \Seg(\lp)$, a basis for $\ZB_\dgr^\Ffunc((\ell_i, \ell_j))$ can be computed in
\[
O\left( n_\dgr^K \cdot n_{\dgr-1}^K \cdot \min\left(n_\dgr^K, n_{\dgr-1}^K\right) + n_{\dgr+1}^K \cdot n_{\dgr}^K \cdot \min\left(n_{\dgr+1}^K, n_{\dgr}^K\right) +    \left(n_\dgr^K\right)^3     \right).
\]
\end{lemma}

\begin{proof}
By~\cref{defn: bd space}, we have that
\[
\ZB_\dgr^\Ffunc ((\ell_i, \ell_j)) := \Zfunc_\dgr(K_i)\cap \Bfunc_\dgr(K_j). 
\]

We compute $\Zfunc_\dgr(K_i)$ as the null space of $\mathtt{B}_\dgr^{K_i}\in \R^{n_{\dgr-1}^K \times n_\dgr^{K_i}}$ via Gaussian elimination. Since $n_\dgr^{K_i} \leq n_\dgr^K$, this process can be computed in $O\left(n_\dgr^K \cdot n_{\dgr-1}^K \cdot \min\left(n_\dgr^K, n_{\dgr-1}^K\right)\right)$ time. Note that, as the output of this Gaussian elimination process, we obtain a set of column vectors $\{ \overrightarrow{u}_1,\ldots,\overrightarrow{u}_{\ell} \} \subseteq \R^{n_\dgr^{K_i}}$ that serves as a basis for $\Zfunc_\dgr(K_i)$. The rows of these column vectors are indexed by oriented $\dgr$-simplices of $K_i$. We extend these column vectors $\overrightarrow{u}_1,\ldots,\overrightarrow{u}_\ell \in \R^{n_\dgr^{K_i}}$ to vectors $\overrightarrow{z}_1,\ldots,\overrightarrow{z}_\ell \in \R^{n_\dgr^{K}}$ by padding zeros to the indices that corresponds to oriented $\dgr$-simplices in $\mathfrak{s}_\dgr^K \setminus \mathfrak{s}_\dgr^{K_i}$.

We compute $\Bfunc_\dgr(K_i)$ as the column space of $\mathtt{B}_{\dgr+1}^{K_i}\in \R^{n_{\dgr}^K \times n_{\dgr+1}^{K_i}}$ via Gaussian elimination. Since $n_{\dgr+1}^{K_i} \leq n_{\dgr+1}^K$, this process can be computed in $O\left(n_{\dgr+1}^K \cdot n_{\dgr}^K \cdot \min\left(n_{\dgr+1}^K, n_{\dgr}^K\right)\right)$ time. As the output of this Gaussian elimination process, we obtain a set of column vectors $\{ \overrightarrow{b}_1,\ldots,\overrightarrow{b}_{k} \} \subseteq \R^{n_\dgr^{K}}$ that serves as a basis for $\Bfunc_\dgr(K_i)$.

By combining $\{ \overrightarrow{z}_1,\ldots,\overrightarrow{z}_\ell \}$ and $\{ \overrightarrow{b}_1,\ldots,\overrightarrow{b}_{k} \}$, we now compute a basis for $\ZB_\dgr^\Ffunc((\ell_i, \ell_j)) = \Zfunc_\dgr(K_i)\cap \Bfunc_\dgr(K_j)$ as follows. Form the following matrix
\[
\begin{bmatrix}
    Z\;\;| \;\; B
\end{bmatrix}
:= 
\begin{bmatrix}
\overrightarrow{z}_1 \cdots \overrightarrow{z}_\ell \; \big| \; \overrightarrow{b}_1 \cdots \overrightarrow{b}_k 
\end{bmatrix}
\in \R^{n_\dgr^K\times(\ell + k)}.
\]

Observe that the null space of $\begin{bmatrix}
    Z\;\;| \;\;B
\end{bmatrix}$ determines $\Zfunc_\dgr(K_i)\cap \Bfunc_\dgr(K_j)$. This is because for any $u \in \R^{\ell } $ and $w \in \R^{k}$,
\begin{gather*}
    \begin{bmatrix}
    u \\
    w
    \end{bmatrix} \text{ is in the null space of } \begin{bmatrix}
    Z\;\;| \;\;B
\end{bmatrix} \\
    \iff \\
    \exists v\in \Zfunc_\dgr(K_i) \cap \Bfunc_\dgr(K_j) \text{ such that } 
    \begin{bmatrix}
    \overrightarrow{z}_1\cdots \overrightarrow{z}_\ell 
    \end{bmatrix}
    u = v = 
    -\begin{bmatrix}
    \overrightarrow{b}_1\cdots \overrightarrow{b}_k 
    \end{bmatrix}
    w.
\end{gather*}
Therefore, in order to obtain a basis for $\ZB_\dgr^\Ffunc((\ell_i, \ell_j))$, we first compute a basis for the null space of $\begin{bmatrix}
    Z\;\;| \;\;B
\end{bmatrix}$ via Gaussian elimination. Let $$\mathcal{B}_{[Z,B]}:=\left\{ \begin{bmatrix}
    u_s \\
    w_s
\end{bmatrix} \in \R^{(\ell+k)} \; \Big| \; u_s \in \R^{\ell}, w_s\in \R^{k} \text{ for } s=1,\ldots,r \right\}$$
be the basis of the null space of $\begin{bmatrix}
    Z\;\;| \;\;B
\end{bmatrix}$ that is obtained from the Gaussian elimination process. Then, the set
\[
\mathcal{B}_{\ZB_\dgr^\Ffunc((\ell_i, \ell_j))}:=\left \{
v_s \in \R^{n_\dgr^K} \Big | \begin{bmatrix}
    \overrightarrow{z_1}\ldots\overrightarrow{z_\ell}
\end{bmatrix}u_s = v_s = -
 \begin{bmatrix}
    \overrightarrow{b_1}\ldots\overrightarrow{b_k}
\end{bmatrix}w_s \text{ for } s=1,\ldots,r
\right\}
\]
is a basis for $\ZB_\dgr^\Ffunc((\ell_i,\ell_j)) = \Zfunc_\dgr(K_i)\cap \Bfunc_\dgr(K_j)$. Since $\ell, k\leq n_\dgr^K$, the Gaussian elimination process for computing the basis $\mathcal{B}_{[Z,B]}$ takes $O\left(\left(n_\dgr^K\right)^3\right)$ time. Similarly, since $r\leq n_\dgr^K$, computing the basis $\mathcal{B}_{\ZB_\dgr^\Ffunc((\ell_i, \ell_j))}$ takes $O\left(\left(n_\dgr^K\right)^3\right)$ time.
Hence, the overall time complexity for computing a basis for $\ZB_\dgr^\Ffunc((\ell_i, \ell_j))$ is
\[
O\left( n_\dgr^K \cdot n_{\dgr-1}^K \cdot \min\left(n_\dgr^K, n_{\dgr-1}^K\right) + n_{\dgr+1}^K \cdot n_{\dgr}^K \cdot \min\left(n_{\dgr+1}^K, n_{\dgr}^K\right) +    \left(n_\dgr^K\right)^3     \right).
\]

\end{proof}

\begin{proof}[Proof of~\cref{prop: algo complexity}]
    In order to compute an (orthonormal) basis for $$\prhi \left(\ZB_\dgr^\Ffunc\right) ((\ell_i, \ell_j)) = \ZB_\dgr^\Ffunc((\ell_i, \ell_j)) \ominus \left( \ZB_\dgr^\Ffunc((\ell_{i-1}, \ell_j)) + \ZB_\dgr^\Ffunc((\ell_i, \ell_{j-1})) \right),$$
we first compute bases for $\ZB_\dgr^\Ffunc((\ell_i, \ell_j))$, $\ZB_\dgr^\Ffunc((\ell_{i-1}, \ell_j))$, and $\ZB_\dgr^\Ffunc((\ell_i, \ell_{j-1}))$ in
\[
O\left( n_\dgr^K \cdot n_{\dgr-1}^K \cdot \min\left(n_\dgr^K, n_{\dgr-1}^K\right) + n_{\dgr+1}^K \cdot n_{\dgr}^K \cdot \min\left(n_{\dgr+1}^K, n_{\dgr}^K\right) +    \left(n_\dgr^K\right)^3     \right),
\]
by~\cref{lem: zb time complexity}. We then compute a basis for 
$\left(\ZB_\dgr^\Ffunc((\ell_{i-1}, \ell_j)) +\ZB_\dgr^\Ffunc((\ell_i, \ell_{j-1}))\right)$ in $O\left( \left(n_\dgr^K\right)^3 \right)$ time by~\cref{lem: basis for sum complexity}. Finally, we compute a basis for
\[
\ZB_\dgr^\Ffunc((\ell_i, \ell_j)) \ominus \left( \ZB_\dgr^\Ffunc((\ell_{i-1}, \ell_j)) + \ZB_\dgr^\Ffunc((\ell_i, \ell_{j-1})) \right)
\]
in $O \left( \left(n_\dgr^K\right)^3  \right)$ by~\cref{lem: ominus complexity}.

Since there are $O\left( m^2 \right)$ segments in $\Seg(P) = \Seg (\{ \ell_1 < \cdots < \ell_m \})$, the total time complexity for computing the function $\prhi \left(\ZB_\dgr^\Ffunc\right) : \Seg(P) \to \gr(C_\dgr^K)$ is 
\[
O\left( m^2 \cdot \left( n_\dgr^K \cdot n_{\dgr-1}^K \cdot \min\left(n_\dgr^K, n_{\dgr-1}^K\right) + n_{\dgr+1}^K \cdot n_{\dgr}^K \cdot \min\left(n_{\dgr+1}^K, n_{\dgr}^K\right) +    \left(n_\dgr^K\right)^3 \right)    \right).
\]
\end{proof}

\section{Edit Distance Stability of Classical Persistence Diagrams and 1-Parameter Grassmannian Persistence Diagrams}\label{appendix: edit}

In this section, we present an example that illustrates how Grassmannian persistence diagrams are more discriminative than classical persistence diagrams. Additionally, we provide the proof of~\cref{thm: classical pd lower bound}.

\begin{example}\label{example: same pd different 0-prhi}
    Consider the filtrations $\Ffunc$ and $\Gfunc$ depicted in~\cref{fig:degree0example}. Their degree-$0$ persistence diagrams are the same, thus
    \[
    d_{\fnc_{\geq 0}}^E \left(\pd_0^\Ffunc, \pd_0^\Gfunc\right) = 0.
    \]
    On the other hand, $\prhi \left(\ZB_0^\Ffunc\right)$ and $\prhi \left(\ZB_0^\Gfunc\right)$ permit distinguishing the two filtrations as we have that
    \[
    d_{\inndgm\left(C_0^K\right)}^E \left(\prhi \left(\ZB_0^\Ffunc\right), \prhi \left(\ZB_0^\Gfunc\right)\right) > 0.
    \]
    Indeed, one can see that $d_{\inndgm\left(C_0^K\right)}^E \left(\prhi \left(\ZB_0^\Ffunc\right), \prhi \left(\ZB_0^\Gfunc\right)\right) > 0$ by using the equivalence of treegrams and degree-$0$ Grassmannian persistence diagrams,~\cref{thm: equivalence of treegrams and degree 0 orthogonal inversions}. Note that the two filtrations $\Ffunc$ and $\Gfunc$ yield two different treegrams, therefore $\prhi \left(\ZB_0^\Ffunc\right)$ and $\prhi \left(\ZB_0^\Gfunc\right)$ are two different Grassmannian persistence diagrams.
\end{example}

\begin{figure}
    \centering
    \includegraphics[scale=17]{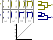}
    \caption{Two filtrations $\Ffunc$ and $\Gfunc$, their treegrams and degree-$0$ persistence diagrams; see~\cref{example: same pd different 0-prhi}}
    \label{fig:degree0example}
\end{figure}

We will need the following lemmas and proposition for proving~\cref{thm: classical pd lower bound}.

\begin{lemma}\label{lem: sum of transverse is transverse}
    Let $\mathcal{I}$ be a finite set and let the two families $\{W (I)\}_{I \in \mathcal{I}}$ and $\{U (I)\}_{I \in \mathcal{I}}$ be transversal to each other. Then, the family $\{ W (I) + U (I) \}_{I\in \mathcal{I}}$ is a transverse family.
\end{lemma}

\begin{proof}
    We have
    \begin{align*}
        \dim \left ( \sum_{i\in \mathcal{I}} ( W(i) + U (i)) \right ) &= \dim \left ( \sum_{i\in \mathcal{I}} W(i) + \sum_{i\in \mathcal{I}}  U(i) \right) \\
        &= \sum_{i\in \mathcal{I}} \dim ( W (i)) + \sum_{i\in \mathcal{I}} \dim (U (i)) \\
        &= \sum_{i\in \mathcal{I}} (\dim (W (i)) + \dim (U (i))) \\
        &= \sum_{i\in \mathcal{I}} \dim (W (i)+U (i)).
    \end{align*}
    The last equality follows from the fact that $\dim (W (i)+U (i)) = \dim (W (i)) + \dim (U (i))$, which can be derived from~\cref{lem: subfamilies transversal}.
\end{proof}

\begin{lemma}\label{lem: push mobeq implies push dim}
    Let $R$ and $S$ be two finite posets. Let $\Mfunc_1 : R \to \gr(V)$ and $\Mfunc_2 : S \to \gr(V)$ be two functions such that $\{ \Mfunc_1 (r) \}_{r\in R}$ and $\{ \Mfunc_2 (s) \}_{s \in S}$ are transversal families. Assume that $h : R \to S$ is an order-preserving map such that $h_\sharp \Mfunc_1 \mobeq \Mfunc_2$. Then, for every $s \in S$, it holds that
    \[
    \dim (\Mfunc_2 (s)) = \sum_{r \in h^{-1}(s)} \dim (\Mfunc_1 (r)).
    \]
\end{lemma}

\begin{proof}
    We will proceed by induction on $s\in S$. For the base cases, let $0_s$ be a minimal element of $S$ (note that there could be more than one minimal element of $S$). As we have that $h_\sharp \Mfunc_1 \mobeq \Mfunc_2$, it follows that
    \[
        \Mfunc_2 (0_s) = \sum_{r\in h^{-1}(0_s)} \Mfunc_1 (r).
    \]
    Considering the dimension of both sides, we obtain
    \begin{align*}
        \dim ( \Mfunc_2 (0_s) ) &= \dim \left ( \sum_{r\in h^{-1}(0_s)} \Mfunc_1 (r) \right ) \\
                                &= \sum_{r\in h^{-1}(0_s)} \dim (\Mfunc_1 (r)),
    \end{align*}        
    where the last equality follows from~\cref{cor: subfamily transversal}. Now, let $s \in S$ be fixed and assume that for every $q \in S$ such that $q < s$, it holds that
    \begin{equation}\label{eqn: induction hypo}
        \dim (\Mfunc_2 (q)) = \sum_{r \in h^{-1}(q)} \dim (\Mfunc_1 (r)).
    \end{equation}
    Using the M\"obius equivalence $\Mfunc_2 \mobeq h_\sharp \Mfunc_1$, we see that
    \begin{align*}
        \dim \left ( \sum_{q\leq s} \Mfunc_2 (q) \right ) &= \dim \left ( \sum_{q\leq s} h_\sharp \Mfunc_1 (q) \right ) &&\\
        &= \dim \left ( \sum_{\substack{r \in R \\ h(r) \leq s}} \Mfunc_1 (r) \right ) &&\\
        &= \sum_{\substack{r \in R \\ h(r) \leq s}} \dim (\Mfunc_1 (r)) && \text{by~\cref{cor: subfamily transversal}}\\
        &= \sum_{\substack{r \in R \\ h(r) < s}} \dim (\Mfunc_1 (r)) + \sum_{r\in h^{-1}(s)} \dim (\Mfunc_1 (s)) &&\\ 
        &= \sum_{q< s} \dim (\Mfunc_2 (q)) + \sum_{r\in h^{-1}(s)} \dim (\Mfunc_1 (s)), && 
    \end{align*}
    where the last equality follows from our induction hypothesis~\cref{eqn: induction hypo}. On the other hand, by~\cref{cor: subfamily transversal}, we have that 
    \begin{align*}
        \dim \left ( \sum_{q\leq s} \Mfunc_2 (q) \right ) &= \sum_{q\leq s} \dim (\Mfunc_2 (q)) \\
                                            &= \sum_{q< s} \dim (\Mfunc_2 (q)) + \dim (\Mfunc_2(s)).
    \end{align*}
    Hence, we conclude that
    \[
    \dim (\Mfunc_2 (s)) = \sum_{r \in h^{-1}(s)} \dim (\Mfunc_1 (r))
    \]
    for every $s\in S$.
\end{proof}

\begin{proposition}\label{prop: dim is functor}
    Let $V$ be a finite-dimensional inner product space and let $\Mfunc : \overline{\lp}^\times \to \gr(V)$ be an object in $\inndgm(V)$. The assignment
    \[
    \Mfunc \mapsto \dim (\Mfunc)
    \]
    is a functor from $\inndgm(V)$ to $\fnc_{\geq 0}$, where $\dim (\Mfunc) : \overline{\lp}^\times \to \Z_{\geq 0} $ is defined by $$\dim (\Mfunc) (I) := \dim (\Mfunc (I))$$ for every $I\in \overline{\lp}^\times$.
\end{proposition}

\begin{proof}
    Let $\Mfunc : \overline{\lp_1}^\times \to \gr(V)$ and $\Nfunc : \overline{\lp_2}^\times \to \gr(V)$ be two objects in $\inndgm(V)$ and let $\ladj{f} : \lp_1 \leftrightarrows \lp_2: \radj{f}$ be a morphism from $\Mfunc$ to $\Nfunc$. This means that there is $\zeta_{\lp_2} : \overline{\lp_2}^\times \to \gr(V)$ supported  on $\diag(\lp_2)$ such that the families $\{ \Nfunc(J) \}_{J\in \overline{\lp_2}^\times} $ and $ \{ \zeta_{\lp_2} (J) \}_{J\in \overline{\lp_2}^\times}$ are transversal to each other and 
    \[
    \left(\overline{\ladj{f}}\right)_\sharp \Mfunc \mobeq \Nfunc + \zeta_{\lp_2}.
    \]
    Let $\tilde{\Nfunc}(J) := \Nfunc(J) + \zeta_{\lp_2} (J)$ for every $J \in \overline{\lp_2}^\times$. Then, by~\cref{lem: sum of transverse is transverse}, we have that $\{ \tilde{\Nfunc} (J) \}_{J \in \overline{\lp_2}^\times}$ is a transverse family. Moreover, we have that $\left(\overline{\ladj{f}}\right)_\sharp \Mfunc \mobeq \tilde{\Nfunc}$. Hence, by~\cref{lem: push mobeq implies push dim}, we conclude that 
    \[
    \dim\left (\tilde{\Nfunc}(J)\right) = \sum_{I \in \left(\overline{\ladj{f}}\right)^{-1} (J)} \dim ( \Mfunc(I)).
    \]
    In particular, for every $J \in \Seg(\lp_2) \setminus \diag(\lp_2)$, we have
    \[
    \dim\left(\Nfunc(J)\right)  = \sum_{I \in \left(\overline{\ladj{f}}\right)^{-1} (J)} \dim ( \Mfunc(I) ).
    \]
    Thus, the Galois connection $\ladj{f} : \lp_1 \leftrightarrows \lp_2 : \radj{f}$ is a morphism from $\dim(\Mfunc)$ to $\dim(\Nfunc)$.
\end{proof}

\begin{proof}[Proof of~\cref{thm: classical pd lower bound}]
    By~\cref{prop: dim is functor}, every path in $\inndgm (C_\dgr^K)$ induces a path in $\fnc_{\geq 0}$ with the same cost. Moreover, as already shown in~\cref{prop: dim of loi is classical pd}, we have $$\dim \left(\prhi \left(\ZB_\dgr^\Ffunc\right)\right) = \pd_\dgr^\Ffunc$$ on $\Seg(\lp) \setminus \diag(\lp)$. Thus, the result follows.
\end{proof}

\section{Algorithm for Reconstructing the Degree-\texorpdfstring{$0$}{} Grassmannian Persistence Diagram from the Treegram}\label{appendix: construction}

Let $\Ffunc : \{1 < \cdots < n \} \to \subcx(K)$ be a filtration of a connected finite simplicial complex $K$ and let $T_\Ffunc$ be the treegram of $\Ffunc$. In this section, we describe an algorithm for reconstructing $\prhi \left(\ZB_0^\Ffunc\right)$ from the treegram $T_\Ffunc$. 

Using the treegram $T_\Ffunc$, we first form a subspace $\mathcal{S}((b,d)) \subseteq C_0^K$ for every $1\leq b\leq d\leq n$. Then, we show that $\mathcal{S}((b,d)) = \prhi \left(\ZB_0^\Ffunc\right) ((b,d))$ for all $b\leq d$ (\cref{prop: treegram to degree 0 inversion subspaces}). The construction of $\mathcal{S}((b,d))$ involves multiple steps. We proceed through the following steps.

\begin{enumerate}
    \item Fix $1<d\leq n$. (to be treated as a death time)
    \item Fix a block of the sub-partition $T_\Ffunc (d)$, say $B_i$.
    \item Detect the blocks of $T_\Ffunc(d-1)$ that merge into $B_i$ at time $d$, and define a notion of birth times for these blocks, say $b_{i,1},\ldots, b_{i,m_i}$.
    \item For each birth time $b_{i,j}$ form a subspace $S_i((b_{i,j}, d)) \subseteq C_0^K$, which captures the connected components that are born at $b_{i,j}$ and die at $d$ in by merging into the block $B_i$.
    \item Repeat step $3$ and step $4$ for every block $B_i$ of $T_\Ffunc (d)$ and for every $1<d\leq n$.
    \item Form subspaces $\mathcal{S}((b,d))$ by appropriately organizing $S_i((b_{i,j},d))$s.
\end{enumerate}

\noindent\underline{Step $1$:} Let $1<d\leq n$ be fixed. 

\medskip
\noindent
\noindent\underline{Step $2$:} Let $B_1,\ldots,B_{N}$ be the blocks of $T_\Ffunc (d)$. Note that $N$ depends on $d$, i.e. N=N(d). Fix $1\leq i \leq N(d)$.

\medskip
\noindent
\noindent Step $1$ and Step $2$ should be seen as initiating two for loops:

\begin{lstlisting}
    for $1<d\leq n$:
        for $1\leq i \leq N(d)$:
            $\ldots$
\end{lstlisting}

\noindent where Step $3$ and Step $4$ describe what should be done inside the for loops.

\medskip
\noindent
\noindent\underline{Step $3$:} Assume that there are blocks $B_{i,1},\ldots,B_{i,m_{i}}$ at time $t=d-1$ that merge into $B_i$ at time $d$. That is, $\cup_{j=1}^{m_{i}} B_{i,j} \subseteq B_i$. Notice that $B_i$ might be strictly larger than the union since there might be \emph{ephemeral} points, i.e. points with the same birth and death time. We let $\{ v_{i,1},\ldots,v_{i,l_{i}} \}=B_i \setminus (\cup_{j=1}^{m_i} B_{i,j})$ denote all such ephemeral points. In~\cref{fig:treegram}, we illustrate the scenario described here. In the example in~\cref{fig:treegram}, the set $\{ v_{i,1},\ldots,v_{i,l_{i}} \} $ is given by $ \{ a_1, a_2\}$.
Recall that we defined the birth time of a point $x \in X$ in a treegram $T_X$ as
\[
b_x := \min \{ t\in \R \mid x\in X_t \}.
\]
Let $V := V(K)$ be the vertex set of $K$. For a non-empty subset $Y \subseteq V$, we define
\begin{align*}
    b(Y) &:= \min \{ b_y \mid y\in Y \}, \\
    c(Y) &:= \frac{1}{|Y|} \sum_{y\in Y} y  \; \; \in C_0^K.
\end{align*}
Let $b_{i,j} := b(B_{i,j})$. Without loss of generality, we may assume that 
\[
b_{i,1} = b_{i,2} = \cdots = b_{i,k_i} < b_{i, k_i+1} \leq \cdots \leq b_{i,m_i}.
\]

\medskip
\noindent
\noindent\underline{Step $4$:} For $1\leq j \leq m_i$ we define
\begin{align*}
    R_{i,j} &:= \{ x \in B_{i,j} \mid b_x = b_{i,j} \} \\
    c_{i,j} &:= c(R_{i,j}).
\end{align*}
For any $0$-chain $c\in C_0^K$, let $\spn\{ c \}\subseteq C_0^K$ denote the subspace generated by $c$. Now, we define
\[
    S_i((b_{i,k_i}, d)) := \sum_{l=2}^{k_i} \spn\{ c_{i,l} - c_{i,1}\} \; \; \subseteq C_0^K.
\]
When $k_i =1$, then the sum above is an empty sum. In that case, we let $S_i((b_{i,j}, d)) := \{ 0 \}$. Notice also that $\dim(S_i((b_{i,k_i}, d))) = k_i-1$ , which is the number of connected components that are born at $b_{k_i}$ and dead at $d$ by merging together into $B_i$.

\noindent For $k_i +1 \leq j \leq m_i$, we define
\begin{align*}
    {R'}_{i,j} &:= \{ x \in B_{i} \mid  \exists B_{i,j'} \ni x \text{ such that }b_x \leq b_{i,j} \text{ and } b_{i,j'}<b_{i,j}\}, \\
    {c'}_{i,j} &:= c({R'}_{i,j}), \\
    S_i((b_{i,j}, d))) &:= \spn\{ c_{i,j} - {c'}_{i,j}\} \; \; \subseteq C_0^K.
\end{align*}

See~\cref{example: construction-explicit treegram} for an explicit construction of what is described here in step $4$.

\medskip
\noindent
\noindent\underline{Step $5$:} Repeat steps $3$ and $4$ for every block of $T_\Ffunc (d)$ and for every $1<d\leq n$. That is, we iterate through the for loops which were initiated in steps $1$ and $2$.

\begin{lstlisting}
    for $1<d\leq n$:
        for $1\leq i \leq N(d)$:
            do Step $3$ and Step $4$.
\end{lstlisting}

\medskip
\noindent
\noindent\underline{Step $6$:}
In this final step, we are out of the nested for loops, and we construct $\mathcal{S}((b,d))$ for each segment $(b,d) \in \Seg ([n])$ by utilizing $S_i$s that were computed inside the for loops.
For any $b < d$, we define 
\[
\mathcal{S}((b,d)) := \sum_{\substack{1\leq i \leq N, \\ 1\leq j \leq m_i; \\ b_{i,j}=b}} S_i((b_{i,j}, d)) \; \; \subseteq C_0^K.
\]
Notice that
\[
\dim\left(S_i((b_{i,k_i}, d)) + \sum_{j=k_i+1}^{m_i} S_i ((b_{i,j}, d))\right) = m_i -1. 
\]
This follows from the fact the family $\{ S_i((b_{i,k_i}, d)) \} \cup \{S_i ((b_{i,j}, d)) \}_{j=k_i + 1}^{m_i}$ is transversal by construction. Moreover, observe that $S_{i_1}(I)$ and $S_{i_2}(J)$ are orthogonal to each other whenever $i_1 \neq i_2$ for every segment $I$ and $J$. Thus, we conclude that $\dim (\mathcal{S}((b,d)))$ is the number of connected components that are born at $b$ and died at $d$.

Recall that we have $\{ v_{i,1},\ldots,v_{i,l_{i}} \} = B_i \setminus (\cup_{j=1}^{m_i} B_{i,j})$. Let $B^{\mathrm{o}}_{i}:=(\cup_{j=1}^{m_i} B_{i,j})$. We define
\begin{align*}
    S_i ((d,d)) &:= \sum_{j=1}^{l_i} \spn\{ v_{i,j} - c(B^\mathrm{o}_{i})\}, \\
    \mathcal{S}((d,d)) &:= \sum_{i=1}^N S_i((d,d)).
\end{align*}

\begin{example}\label{example: construction-explicit treegram}
    We illustrate the constructions of $S((b_{i,j}, d))$ through an explicit treegram. Consider the part of a treegram depicted in~\cref{fig:treegram}. In this case, we have
    \begin{align*}
        B_{i,1} &= \{ x, y, z, u \}, \\
        B_{i,2} &= \{ v, w \}, \\
        B_{i,3} &= \{ g, h \}, \\
        B_{i,4} &= \{ k \}, \\
        B_{i,5} &= \{ l,m, n \}, \\
        B_{i,6} &= \{ p, q, r \}, \\
        B^\mathrm{o}_{i} &= (\cup_{j=1}^6 B_{i,j}), \\
        B_{i} &= B^\mathrm{o}_{i} \cup \{ a_1, a_2\}.
    \end{align*}
    In this example, it holds that $b_1 := b_{i,1} = b_{i,2} = b_{i,3} = b_{i,4} < b_{i, 5} = b_{i,6} =: b_2 $. Following the definitions of $R_{i,j}$ and ${R'}_{i,j}$ we compute
    \begin{align*}
        R_{i,1} &= \{ y \}, \\
        R_{i,2} &= \{ v, w \}, \\
        R_{i,3} &= \{ g \}, \\
        R_{i,4} &= \{ k \}, \\
        R_{i,5} &= \{ l,n \}, \\
        R_{i,6} &= \{ p,q,r \}, \\
        {R'}_{i,5} &= {R'}_{i,6} = \{ x, y, z, v, w, g, h, k \}.
    \end{align*}
    Then, we see that ${c'}_{i,5} = {c'}_{i,6} =  \frac{1}{8}(x+y+z+v+w+g+h+k)$ and $c(B^\mathrm{o}_i) = \frac{1}{15}(x+y+z+u+v+w+g+h+k+l+m+n+p+q+r)$. Hence, we have that
    \begin{align*}
        S_i ((b_1,d)) &= \spn\{\frac{1}{2}(v+w) - y\}+ \spn\{ g-y\} + \spn\{ k-y\}, \\
        S_i ((b_2, d)) &= \spn\{\frac{1}{2}(l+n) - {c'}_{i,5}\}+ \spn\{\frac{1}{3}(p+q+r) - {c'}_{i,6}\}, \\
        S_i ((d,d)) &= \spn\{ a_1 - c(B^\mathrm{o}_i)\} + \spn \{  a_2 - c(B^\mathrm{o}_i)\}
    \end{align*}
\end{example}

\begin{figure}
    \centering
    \includegraphics[scale=21]{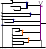}
    \caption{A treegram for illustrating the steps $3$ and $4$ of the algorithm and an explicit construction of $\mathcal{S}$; see~\cref{example: construction-explicit treegram}}
    \label{fig:treegram}
\end{figure}

\begin{proposition}\label{prop: treegram to degree 0 inversion subspaces}
    Let $\Ffunc$ be a filtration and let $T_\Ffunc$ be its treegram. Let $\mathcal{S} $ be constructed as explained above. Then, for any $b\leq d$, $\mathcal{S}((b,d)) = \prhi \left(\ZB_0^\Ffunc\right) ((b,d))$.
\end{proposition}

\begin{proof}
    Assume $b< d$ and let $i,j$ be such that $b_{i,j} = b$. Observe that it is enough to check $S_i ((b_{i,j}, d)) \subseteq \prhi \left(\ZB_0^\Ffunc\right) ((b_{i,j}, d))$. The equality would then follow from the fact that the sum in the definition of $\mathcal{S}([b, d])$ is a direct sum by the construction of $S_i$s and the fact that the dimensions of $\mathcal{S}((b,d))$ and $\prhi \left(\ZB_0^\Ffunc\right) ([b, d])$ have to be the same as they are both equal to the number of connected components that are born at time $b$ and die at time $d$, as explained during the construction of $\mathcal{S}$. 
    
    Recall that by~\cref{prop: linAlg Mobius} we have $$\prhi \left(\ZB_0^\Ffunc\right) ((b_{i,j}, d)) = \ZB_0^\Ffunc ((b_{i,j},d)) \ominus (\ZB_0^\Ffunc ((b_{i,j}-1, d)) + \ZB_0^\Ffunc((b_{i,j}, d-1))).$$ So, we need to check that $S_i ((b_{i,j}, d))$ is orthogonal to both $\ZB_0^\Ffunc ((b_{i,j}-1, d))$ and $\ZB_0((b_{i,j}, d-1))$. There are two cases to be checked, namely $1\leq j \leq k_i$ and $k_i +1 \leq j \leq m_i$. As the processes, in either case, are similar to each other, we provide details for the case  $1\leq j \leq k_i$.
    
    Assume that $1\leq j \leq k_i$. In this case, we have that 
    \[
    S_i ((b_{i,j},d)) = S_i ((b_{i,k_i},d)) = \sum_{l=2}^k \spn\{ c_{i,l} - c_{i,1}\}.
    \]
    Thus, it is enough to check that $c_{i,l} - c_{i,1}$ is orthogonal to both $\ZB_0^\Ffunc ((b_{i,j}-1, d))$ and $\ZB_0^\Ffunc((b_{i,j}, d-1))$ for $l=2,\ldots,k_i$. The support of the chain $c_{i,l} - c_{i,1}$ is a subset of $B_i$. On the other hand, support of any chain in $\ZB_0^\Ffunc ((b_{i,j}-1, d))$ is in the complement of $B_i$. Thus, $c_{i,l} - c_{i,1}$ is orthogonal to $\ZB_0^\Ffunc ((b_{i,j}-1, d))$. The subspace $\ZB_0^\Ffunc((b_{i,j}, d-1))$ is generated by elements of the form $x-y$ where $x$ and $y$ belong to the same block of $T_\Ffunc (d-1)$. Thus, either support of $c_{i,l} - c_{i,1}$ and the set $\{ x,y\}$ are disjoint or $x$ and $y$ are both in the support of $c_{i,l} - c_{i,1}$ with the same coefficients. In both scenarios, we get that $c_{i,l} - c_{i,1}$ is orthogonal to $\ZB_0^\Ffunc((b_{i,j}, d-1))$. Thus, $c_{i,l} - c_{i,1}$ is orthogonal to $\ZB_0^\Ffunc ((b_{i,j}-1, d))$ and $\ZB_0^\Ffunc((b_{i,j}, d-1))$. 
    
    When $b=d$, the proof is similar.
\end{proof}

\end{document}